\documentclass{article}

\usepackage{amsthm,amsmath}
\RequirePackage[OT1]{fontenc}
\RequirePackage[authoryear]{natbib}
\RequirePackage{hyperref}

\usepackage{fullpage}
\usepackage{amsmath,amssymb}
\usepackage{amsthm}
\usepackage{theoremref}
\usepackage{cite}
\usepackage{graphicx}
\usepackage{bm,mathrsfs,amsfonts,amsthm}
\usepackage{multicol,caption}
\usepackage[boxed]{algorithm2e}
\usepackage{bm}
\usepackage{bbm}

\DeclareMathOperator{\im}{im}						
\newcommand{\id}{\textrm{id}}						
\newcommand{\cech}{\textrm{\v Cech}}				
\newcommand{\R}{\mathbb{R}}							
\newcommand{\E}{\mathbb{E}}							
\newcommand{\Z}{\mathbb{Z}}							
\newcommand{\N}{\mathbb{N}}							
\renewcommand{\P}{\mathbb{P}}						
\renewcommand{\dim}{\mathbbm{d}}					
\newcommand{\W}{\mathcal{W}_{0:\dim-1}}				
\newcommand{\K}{\mathcal{K}}						
\newcommand{\one}{\mathbbm{1}}						
\newcommand{\goto}{\rightarrow}						
\newcommand{\del}[1]{\partial_{#1}}					
\newcommand{\LP}{\left(} \newcommand{\RP}{\right)}	
\newcommand{\LC}{\left\{} \newcommand{\RC}{\right\}}
\newcommand{\LN}{\left\|} \newcommand{\RN}{\right\|}
\newcommand{\Ln}{\left|} \newcommand{\Rn}{\right|}	
\newcommand{\LB}{\left[} \newcommand{\RB}{\right]}	

\newtheorem{remark}{Remark}[section]
\newtheorem{example}{Example}
\newtheorem{defn}{Definition}[section]
\newtheorem{theorem}{Theorem}
\newtheorem{lemma}{Lemma}[section]
\newtheorem{prop}[lemma]{Proposition}
\newtheorem{cor}[lemma]{Corollary}
\numberwithin{equation}{section}

\begin{document}
\title{Nonparametric Estimation of Probability Density Functions of Random Persistence Diagrams}

\author{
	Mike, Joshua Lee \\
	\and
	Maroulas, Vasileios
		\thanks{Corresponding Author}
		\thanks{Research has been supported by the Army Research Office (ARO) Grant \# W911NF-17-1-0313.}}
\date{University of Tennessee, Knoxville}

\maketitle

\begin{abstract}
We introduce a nonparametric way to estimate the global probability density function for a random persistence diagram.
Precisely, a kernel density function centered at a given persistence diagram and a given bandwidth is constructed.
Our approach encapsulates the number of topological features and considers the appearance or disappearance of features near the diagonal in a stable fashion.
In particular, the structure of our kernel individually tracks long persistence features, while considering features near the diagonal as a collective unit.
The choice to describe short persistence features as a group reduces computation time while simultaneously retaining accuracy.
Indeed, we prove that the associated kernel density estimate converges to the true distribution as the number of persistence diagrams increases and the bandwidth shrinks accordingly.
We also establish the convergence of the mean absolute deviation estimate, defined according to the bottleneck metric. 
Lastly, examples of kernel density estimation are presented for typical underlying datasets. 
\end{abstract}



\section{Introduction}
Topological data analysis (TDA) encapsulates a range of data analysis methods which investigate the topological structure of a dataset \citep{CompyTopo}. 
One such method, persistent homology, describes the geometric structure of a given dataset and summarizes this information as a persistence diagram.
TDA, and in particular persistence diagrams, have been employed in several studies with topics ranging from classification and clustering \citep{tda_action, tda_number, tda_clustering2015,MaMa16-2} to the analysis of dynamical systems \citep{tda_windows,VM_CellMotion,tda_signal,tda_timeseries} and complex systems such as sensor networks \citep{persissensor,fullerene,persis_brain}.
In this work, we establish the probability density function (pdf) for a random persistence diagram.

Persistence diagrams offer a topological summary for a collection of ${\dim}$-dimensional data, say $\LC x_i \RC \subset \R^{\dim}$, which focuses on their global geometric structure of the data.
A persistence diagram is a multiset of homological features $\LC (b_i,d_i,k_i) \RC$, each representing a $k_i$-dimensional hole which appears at scale $b_i \in \R^+$ and is filled at scale $d_i \in (b_i, \infty)$. 
In general, the dataset arises from any metric space, though restricting to $\LC x_i \RC \subset \R^{\dim}$ guarantees $k_i \in  \LC 0,..., \dim-1 \RC$. 
For example, if the data form a time series trajectory $x_i = f(t_i)$, the associated persistence diagram describes multistability through a corresponding number of persistent 0-dimensional features or periodicity through a single persistent 1-dimensional feature.
In a typical persistence diagram, few features exhibit long persistence (range of scales $d_i - b_i$), and such features describe important topological characteristics of the underlying dataset. 
Moreover, persistent features are stable under perturbation of the underlying dataset \citep{tda_stability}. 

Persistence diagrams have recently seen intense active research, including significant successful effort toward facilitating previously challenging computations with them;
these efforts impact evaluation of Wasserstein distance in \citep{geomhelps} and the creation of persistence diagrams with packages such as Dionysus \citep{tdaR} and Ripser \citep{ripser} which take advantage of certain properties of simplicial complexes \citep{persistwist}.
Recently, various approaches have defined specific summary statistics such as center and variance \citep{tda_kernels,Wass_Structure,Frechet_Computation, MaMa16-2}, birth and death estimates \citep{tda_parametric}, and confidence sets \citep{tda_confidence}.
Here we introduce a nonparametric method to construct density functions for a distribution of persistence diagrams. 
The development of these densities offers a consistent framework to understand the above summary statistic results through a single viewpoint.

We naturally think of a (random) persistence diagram as a random element which depends upon a stochastic procedure which is used to generate the underlying dataset that it summarizes.
Given that geometric complexes are the typical paradigms for application of persistent homology to data analysis, see for example the partial list \citep{persissensor, tda_parametric, tda_signal, MaMa16, tda_windows, tda_timeseries, fullerene, tda_action, tda_images, tda_wheeze}), we consider persistence diagrams which arise from a dataset and its associated $\cech$ filtration.
Thus, sample datasets yield sample persistence diagrams without direct access to the distribution of persistence diagrams.
In this sense, a distribution of persistence diagrams is defined by transforming the distribution of underlying data under the process used to create a 
persistence diagram, as discussed in \citep{Wass_Structure}.
The persistence diagrams are created through a complex and nonlinear process which relies on the global arrangement of datapoints (see Section \ref{sect:TDA}); thus, the structure of a persistence diagram distribution remains unclear even for underlying data with a well-understood distribution. 
Indeed, known results for the persistent homology of noise alone, such as \citep{tda_crackle}, primarily concern the asymptotics of feature cardinality at coarse scale.
With little previous knowledge, we study these distributions through nonparametric means. 
Kernel density estimation is a well known nonparametric technique for random vectors in $\R^{\dim}$ \citep{KDE_book};
however, persistence diagrams lack a vector space structure and thus these techniques cannot be applied directly here.

There has been extensive work to devise various maps from persistence diagrams into Hilbert spaces, especially Reproducing Kernel Hilbert Spaces (RKHS). 
For example, \citep{PersistImages} discretizes persistence diagrams via bins, yielding vectors in a high dimensional Euclidean space. 
The works \citep{MSKernel} and \citep{PWGK_TDA} define kernels between persistence diagrams in a RKHS. 
By mapping into a Hilbert space, these studies allow the application of machine learning methods such as principal component analysis, random forest, support vector machine, and more. 
The universality of such a kernel is investigated in \citep{Stat_TDA}; this property induces a metric on distributions of persistence diagrams (by comparing means in the RKHS), as \citep{Stat_TDA} demonstrates with a two-sample hypothesis test.
In a similar vein, \citep{PD_rep} utilizes Gibbs distributions in order to replicate similar persistence diagrams, e.g. for use in MCMC type sampling.

All previous approaches kernelize to map into a Hilbert space for typical statistical learning techniques.
In a similar vein, the studies \citep{tda_kernels} and \citep{tda_confidence} work with kernel density estimation on the underlying data to estimate a target diagram as the number of underlying datapoints goes to infinity.
In both cases, the target diagram is directly associated to the probability density function (pdf) of the underlying data (via the superlevel sets of the pdf). 
The first work constructs an estimator for the target diagram, while the second defines a confidence set. 
In either case, kernel density estimation is used to approximate the pdf of the underlying datapoints, assuming the data are independent and identically distributed (i.i.d.).
In contrast, our work considers a different kind of kernel density which directly estimates probability densities for a random persistence diagram from a sample of persistence diagrams.
This kernel density estimate converges to the true probability density as the number of persistence diagrams goes to infinity.

Instead of a transformed collection or a center diagram, the output of our method is an estimate of a probability density function (pdf) of a random persistence diagram. 
Access to a pdf facilitates definition and application of many statistical techniques, including hypothesis testing, utilization of Bayesian priors, or likelihood methods. 
The proposed kernel density is centered at a persistence diagram and describes each feature as having either short or long persistence;
by treating each long-persistence point individually and short persistence points collectively, the kernel density strikes a careful balance between accuracy and computation time.
Our method also enables expedient sampling of new persistence diagrams from the kernel density estimate. 
In contrast to previous methodologies, our kernel density estimate has the potential to describe high probability features in a random persistence diagram, even if these features have \emph{brief} persistence.
Such features are typically indicative of the geometric structure, e.g., curvature, of the dataset rather than its topology.

The homological features $(b_i,d_i,k_i)$ in a persistence diagram come without an ordering and their cardinality is variable, being bounded but not defined by the cardinality of the underlying dataset. 
Thus, any notion of density must be (i) invariant to the ordering of features and (ii) account for variability in their cardinality. 
Indeed, the approach used to analyze a collection of persistence diagrams in \citep{persis_brain} is a good step toward understanding a random persistence diagram, but requires a choice of order and considers only a fixed number of features and is therefore unsuitable for creating probability densities. 
In this work, we offer a kernel density with the desirable properties (i) and (ii), which also calls attention to the persistence of each feature.
A typical persistence diagram has many features with brief persistence and few with moderate or longer persistence; consequently, our kernel density groups features with short persistence together in order to combat the curse of dimensionality.
Indeed, the kernel density still considers features of short persistence, but simplifies their treatment in order to facilitate computation. 
The kernel density is defined on a pertinent space of finite random sets which is equipped to describe pdfs for random persistence diagrams generated from associated data with bounded cardinality of topological features.
In this sense, our kernel density provides estimation of the distribution of persistence diagrams which in turn describes the geometry of the random underlying dataset. 
The requirement of bounded feature cardinality is trivially satisfied for datasets with bounded cardinality, which is reasonable for application and theory.
Indeed, the creation of a persistence diagram from an infinite collection of data is often nonsensical (e.g., for anything with unbounded noise), and a scaling limit should be considered instead.

We establish the kernel density estimation problem through the lens of finite set statistics and we consequently begin with relevant backgrounds in topological data analysis in Section \ref{sect:TDA} and finite set statistics in Section \ref{sect:RPDs}. 
For further details about these two subjects, the reader may refer respectively to \citep{CompyTopo} and \citep{Matheron}.
Our results are presented in Section \ref{sect:KDE}.
In Subsection \ref{subsect:KDE_construction}, we construct the kernel density associated to a center persistence diagram and kernel bandwidth parameter.
This consists of decomposing the center persistence diagram into lower and upper halves, finding pdfs associated to each half, and lastly determining the pdf for their union.
After the kernel density is defined and an explicit pdf is delivered in Thm. \ref{thm_construction}, its convergence is presented in Theorem \ref{thm_KDE}.
Next, Subsection \ref{subsect:Examples} presents in detail a specific example of the kernel density.
Additionally, an example of persistence diagram kernel density estimation and its convergence are demonstrated for persistence diagrams associated to underlying data with annular distribution.
In Subsection \ref{subsect:KDE_MoD}, we define the mean absolute deviation (MAD) as a measure of dispersion, and present the convergence of its kernel density estimator (Thm. \ref{thm_moment}). 
Finally, we end with conclusions and discussion in Section \ref{sect:discussion}.
Further examples of KDE convergence and the proofs of the main theorems, Thm. \ref{thm_KDE} and Thm. \ref{thm_moment}, are given in the supplementary materials.

\section{Topological Data Analysis Background} \label{sect:TDA}
The topological background discussed here builds toward the definition of persistence diagrams, the pertinent objects in this work.
We begin by briefly discussing simplicial complexes and homology, an algebraic descriptor for coarse shape in topological spaces.
In turn, persistent homology, and its summary, persistence diagrams, are techniques for bringing the power and convenience of homology to describe subspace filtrations of topological spaces.
We first consider topological spaces of discernible dimension, called manifolds.

\begin{defn}
A topological space $X$ is called a $k$-dimensional manifold if every point $x \in X$ has a neighborhood which is homeomorphic to an open neighborhood in $k$-dimensional Euclidean space.
\end{defn}

We generalize the fixed-dimension notion of a manifold in order to define simplicial homology for simplicial complexes.
We then discuss the $\cech$ construction which is used to associate simplicial complexes to datasets. 

\begin{defn} \label{simplex}
A $k$-simplex is a collection of $k+1$ linearly independent vertices along with all convex combinations of these vertices:
\begin{equation} \label{convex_combo}
\LP v_0,...,v_k \RP = \LC \sum_{i=0}^k \alpha_i v_i : \sum_{i=0}^k \alpha_i = 1 \textrm{ and } \alpha_i \geq 0 \, \forall i \RC.
\end{equation}
Topologically, a $k$-simplex is treated as a $k$-dimensional manifold (with boundary).
An oriented simplex is typically described by a list of its vertices, such as $\LP v_0, v_1, v_2 \RP$.
The faces of a simplex consist of all the simplices built from a subset of its vertex set;
for example, the edge $(v_1, v_2)$ and vertex $(v_2)$ are both faces of the triangle $\LP v_0, v_1, v_2 \RP$. 
\end{defn}

\begin{defn} \label{simplicial_complex}
A simplicial complex $\K$ is a collection of simplices wherein \\
(i) if $\sigma \in \K$, then all its faces are also in $\K$, and \\
(ii) the intersection of any pair of simplices in $\K$ is another simplex in $\K$. \\
We denote the collection of $k$-simplices within $\K$ by $\K^{[k]}$. 
\end{defn}
A simplicial complex is realized by the union of all its simplices; an example is shown in Fig. \ref{simp_comp}.
Conditions (i) and (ii) in Defn. \ref{simplicial_complex} establish a unique topology on the realization of a simplicial complex which restricts to the subspace topology on each open simplex.
For finite simplicial complexes realized in $\R^{\dim}$, this topology is also consistent with the Euclidean subspace topology.

\begin{figure}[h]
\centering\includegraphics[scale=0.3]{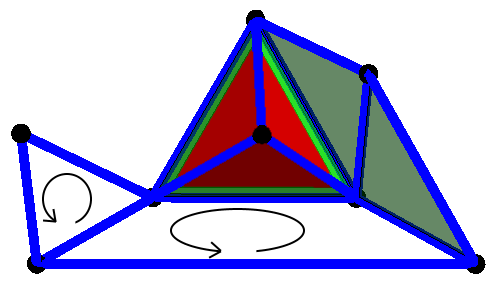}
\caption{An example of a simplicial complex realized in $\R^3$. 
This particular complex has one connected component and two cycles, which generate the 0-homology and 1-homology groups, respectively.
The other homology groups are trivial. }
\label{simp_comp}
\end{figure}

Here we define the homology groups for a simplicial complex through purely combinatorial means, which allows for automated computation. 
\begin{defn} \label{defn_chain_group}
The chain group (over $\Z$) on a simplicial complex $\K$ of dimension $k$ is denoted by $C_k(\K)$ and is defined as formal sums of $k$-simplices in $\K$:
\begin{equation} \label{eqn_chain_group}
C_k(\K) = \LC \sum_{\sigma \in \K^{[k]}} n_\sigma \sigma : n_\sigma \in \Z \RC.
\end{equation}
\end{defn}

\begin{defn} \label{defn_boundary_map}
The $k$-th boundary map is a homomorphism $\del{k} : C_k(\K) \goto C_{k-1}(\K)$ defined on each simplex as an alternating sum over the faces of one dimension less:
\begin{equation} \label{eqn_boundary_map}
\del{k}(v_0,...,v_k) = \sum_{n=0}^k (-1)^n (v_0,...,v_{n-1},v_{n+1},...,v_k).
\end{equation}
\end{defn}

\begin{remark}
Chain groups give an algebraic way to describe subsets of simplices as a formal sum.
Toward this viewpoint, the chain group is often defined over $\Z_2 = \LC 0 , 1 \RC$ instead of $\Z$.
In this case, the boundary maps can be understood classically; e.g., the boundary of a triangle yields (the sum of) its three edges and the boundary of an edge yields (the sum of) its endpoints. 
When viewed over $\Z$, the presence of sign specifies simplex orientation. 
\end{remark}

Putting chain groups of every dimension together along with the boundary maps successively defined between them, we obtain a chain complex:
\begin{equation} \label{eqn_chain_complex}
\LC 0 \RC \xleftarrow{\bm 0} C_0(\K) \xleftarrow{\del{1}} C_1(\K) \xleftarrow{\del{2}} C_2(\K) \xleftarrow{\del{3}} ... 
\end{equation}
The composition of subsequent boundary maps yields the trivial map \citep{CompyTopo};
this property is typically rephrased as $\im(\del{k+1}) \subset \ker(\del{k})$ which enables definition of the following modular groups.
\begin{defn} \label{defn_homology_group}
The homology group of dimension $k$ is given by
\begin{equation} \label{eqn_homology_group}
H_k(\K) = \ker(\del{k}) / \im(\del{k+1}) = \LC [x] = x + \im(\del{k+1}) : x \in \ker(\del{k}) \RC\!,
\end{equation}
where $[x] = \LC x + y : y \in \im(\del{k+1}) \RC$ defines the coset equivalence class of $x$.
\end{defn}
The generators of the homology group correspond to topological features of the complex $\K$;
for example, generators for the $0$-homology group correspond to connected components, generators of $1$-homology group correspond to holes in $\K$, etc. 
The interpretation of these features is exemplified by taking the topological boundary of a $k+1$ ball (that is, a $k$-sphere);
for example, the boundary of an interval is two (disconnected) points while the boundary of a disc is a loop. 

We wish to extend the notion of homology for a discrete set of data $\bm x = \LC x_i \RC_{i=1}^N$ within a metric space $(X,d_X)$.
Treating the set itself as a simplicial complex, its homology yields only the cardinality of the data points.
So, we utilize the metric to obtain more information.
Here we denote by $B(x_0,r_0)$ a metric ball centered at $x_0$ of radius $r_0$. 
Fix a radius $r > 0$ and consider the collection of neighborhoods $U = \LC U_i \RC = \LC B(x_i,r) \RC$ along with its union $\mathcal{U}_r = \cup_i B(x_i,r)$.
The filtration of sets $\LC \mathcal{U}_r \RC_{r \in \R^+}$ naturally yields information about the arrangement within $X$ of the dataset $\bm x$ at various scales. 
To make homology computations more tractable for $\mathcal{U}_r$, we instead consider the associated nerve complexes.
\begin{defn} \label{defn_nerve_and_cech}
The nerve $\mathcal{N}(U)$ of a collection of open sets $U$ is the simplicial complex where a $k$-simplex $\LP v_{i_0},..., v_{i_k} \RP$ is in $\mathcal{N}(U)$ if and only if $\cap_{j=0}^k U_{i_j} \neq \emptyset$. 
The nerve of the neighborhoods $U = \LC B(x_i,r) \RC$ is called the $\cech$ complex on the data $\LC x_i \RC$ at radius $r$ and is denoted by $\textrm{\v Cech}(\bm x, r)$. 
\end{defn}

Examples of the $\cech$ complex for the same data at different radii are depicted in Fig. \ref{grow_cech}, where they are superimposed with the associated neighborhood space.
Any nerve complex trivially satisfies the requirements for a simplicial complex \citep{CompyTopo}.
Moreover, the nerve theorem states that the nerve and union of a collection of convex sets have similar topology (they are homotopy equivalent) \citep{Hatcher};
specifically, the $\cech$ complex and neighborhood space $\mathcal{U}$ have identical homology for any given radius. 

\begin{figure}
\begin{center}
\includegraphics[scale=0.3]{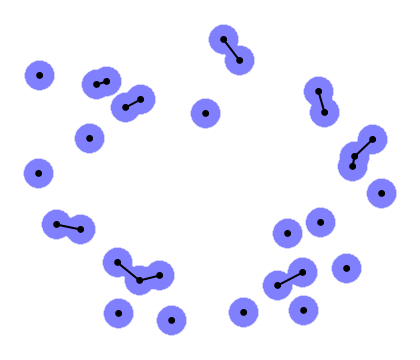} \quad 
\includegraphics[scale=0.3]{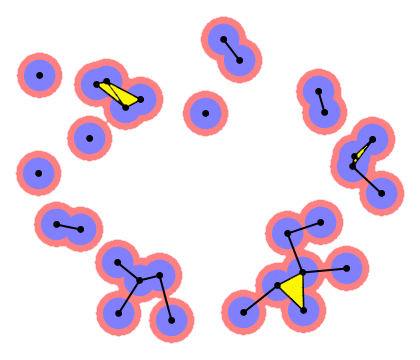} \quad 
\includegraphics[scale=0.3]{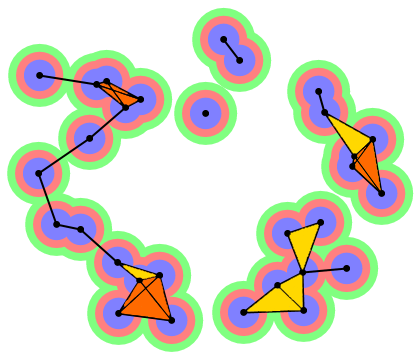}
\end{center}
\caption{The neighborhood space and $\cech$ complex of matching radius plotted at three different radii. Yellow indicates a triangle while orange indicates a tetrahedron. This family of simplicial complexes is the filtration utilized to compute and define persistent homology.} \label{grow_cech}
\end{figure}

A priori, it is unclear which choice of scale (radius), best describes the data; 
and oftentimes different scales reveal different information.
Thus, to investigate the topology of our data, we consider the appearance and disappearance of homological features at growing scale. 
This multiscale viewpoint, called persistent homology, is introduced in \citep{Persistence} and yields a topological summary of the data called a persistence diagram.
This is possible because we have a growing filtration of complexes, so each complex is included in the next (see Fig. \ref{grow_cech}). 
These inclusion maps induce inclusion maps at the chain group level and in turn induce maps (though not typically inclusions) at the level of homology groups.
These induced maps $f_{r_1, r_2}:H_k(\textrm{\v Cech}(\bm{x},r_1)) \goto H_k(\textrm{\v Cech}(\bm{x},r_2))$ are referred to here as the persistence maps, and take features to features (i.e., generators to generators) or to zero \citep{bottleneck}. 
Thus, each feature is tracked by how far the persistence maps preserve it. 
In turn, tracking features is boiled down to a very specific algorithm for obtaining the birth and death radii for each homological feature (e.g., see \citep{CompyTopo}). 
Features which persist over a large range of scale are typically considered more important, and their presence is stable under small perturbations of the underlying data \citep{tda_stability}.

Persistent homology yields a multiset of homological features, each born at a scale $b_i$, lasting until its death scale $d_i$, with degree of homology $k_i$;
in short, it yields a persistence diagram $\mathscr{D} = \LC \xi_i \RC_{i=1}^M = \LC (b_i,d_i,k_i) \RC_{i=1}^M$.
We interpret the birth-death values as coordinate points with degree of homology as labels.
For clarity and simplicity, we ignore any features with death value $d_i = \infty$, since these features are generally a characteristic of the ambient space.
In particular, one homological feature with $(b,d,k) = (0,\infty,0)$ is expected from any $\cech$ filtration.

Specifically, for data in $\R^{\dim}$, we consider each feature as an element of 
\begin{equation} \label{eqn_wedge}
\W = W \times \LC 0,...,\dim-1 \RC, \\
\end{equation}
where $W = \LC (b,d) \in \R^2 : d > b \geq 0 \RC$ is the infinite wedge. 
As a topological space, the ${\dim}$-fold multiwedge $\W$ is treated as ${\dim}$-disconnected copies of $W$, where $W$ has the Euclidean metric and topology.

It is desirable to define a metric between persistence diagrams with which to measure topological similarity. 
In TDA, Hausdorff distance is typically used to compare underlying datasets, while the bottleneck distance (Defn. \ref{defn_bottleneck}) is used to compare their associated persistence diagrams \citep{tda_confidence, tda_guide}.

\begin{defn} \label{defn_bottleneck}
The bottleneck distance between two persistence diagrams $D_1$ and $D_2$ is given by
\begin{equation} \label{eqn_bottleneck}
W_\infty(D_1,D_2) = \min_\gamma \max_{x \in D_1} \LN x - \gamma(x) \RN_\infty.
\end{equation}
where $\gamma$ ranges over all possible bijections between $D_1$ and $D_2$ which match in degree of homology.
The diagonal $\LC b = d \RC$ is included in both persistence diagrams with infinite multiplicity so that any feature may be matched to the diagonal. 
\end{defn}

\begin{remark}
Due to the unstable presence of features near the diagonal, typical metrics on persistence diagrams such as the bottleneck distance treat the diagonal as part of every persistence diagram \citep{Wass_Structure} in order to achieve stability with respect to Hausdorff perturbations of the underlying dataset \citep{bottleneck}.
Morally, one considers the diagonal as representing vacuous features which are born and die simultaneously. 
For convenient computation, the definition of bottleneck distance can be applied to each degree of homology separately.
\end{remark}

\section{Random Persistence Diagrams} \label{sect:RPDs}
In this section we establish background to make the notion of probability density for a random persistence diagram explicit and well-defined.
A persistence diagram changes its feature cardinality under small perturbation of the underlying dataset, and these features have no intrinsic order.
Consequently, we cannot treat persistence diagrams as elements of a vector space.
Instead, we consider a random persistence diagram $D$ as a random multiset of features $D = \LC \xi_i \RC \subset \W$ in the multiwedge defined in Eq. \eqref{eqn_wedge}. 
For underlying datasets sampled from $\R^{\dim}$ with bounded cardinality, the affiliated $\cech$ persistence diagrams also have bounded feature cardinality and degree of homology.
Thus, we assume that the cardinality of a random persistence diagram is bounded above by some value $\Ln D \Rn \leq M \in \N$ , and so consider the space $\mathcal{C}_{\leq M}(\W) = \LC D \textrm{ multiset in } \W : \Ln D \Rn \leq M \RC$.
We view $\mathcal{C}_{\leq M}(\W)$ through a list of functions $h_N$ which each map the appropriate dimension of Euclidean space into its corresponding cardinality component, $\mathcal{C}_{N}(\W)$.
This viewpoint facilitates the definition of probability densities. 

\begin{defn} \label{defn_hit-or-miss}
For each $N \in \LC 0 , ..., M \RC$, consider the space of $N$ topological features, denoted $\mathcal{C}_N(\W) = \LC D \textrm{ multiset in } \W : \Ln D \Rn = N \RC$, and the associated map $h_N: \W^N \goto \mathcal{C}_N(\W)$ defined by
\begin{equation} \label{eqn_E2HoM}
h_N(\xi_1,...,\xi_N) = \LC \xi_1,...,\xi_N \RC.
\end{equation}
The map $h_N$ creates equivalence classes on $\W^N$ according to the action of the permutations $\Pi_N$;
specifically, $\LB Z \RB = \LB \LP \xi_1,...,\xi_N \RP \RB_{h_N} = \LC \LP \xi_{\pi(1)},...,\xi_{\pi(N)} \RP : \pi \in \Pi_N \RC$ for each $Z = \LP \xi_1,...,\xi_N \RP \in \W^N$.
These equivalence classes yield the space
\begin{equation} \label{eqn_mod_space}
\W^N / \Pi_N = \LC \LB \bm{\xi} \RB_{h_N} : \bm{\xi} \in \W^N \RC,
\end{equation}
equipped with the quotient topology.
The topology on $\mathcal{C}_{\leq M}(\W)$ is defined so that each $h_N$ lifts to a homeomorphism between $\W^N / \Pi_N$ and $\mathcal{C}_N(\W)$.
\end{defn}

With a topology in hand, one can define probability measures on the associated Borel $\sigma$-algebra. 
Thus, we define a random persistence diagram $D$ to be a random element distributed according to some probability measure on $\mathcal{C}_{\leq M}(\W)$ for a fixed maximal cardinality $M \in \N$.
We denote associated probabilities by $\P \LB \cdot \RB$ and expected values by $\E \LB \cdot \RB$.
Since $\W^N / \Pi_N \cong \mathcal{C}_N(\W)$, we work toward defining probability densities on the collection of Euclidean spaces $\cup_{N=0}^M \W^N$. 

\begin{defn}[\citep{Matheron}] \label{defn_belief}
For a given random persistence diagram $D$ and any Borel subset $A$ of $\W$, the belief function $\beta_D$ is defined as
\begin{equation} \label{eqn_belief}
\beta_D(A) = \P \LB D \subset A \RB.
\end{equation}
\end{defn}
 
Since $A$ is a Borel subset of $\W$, the collection $O_A = \LC D \in \mathcal{C}_{\leq M}(\W) : D \subset A \RC$ is the quotient of $\cup_{N=0}^M A^N \subset \cup_{N=0}^M \W^N$ under $h_N$; moreover, $A^N$ is clearly Borel in the Euclidean topology of $\cup_{N=0}^M \W^N$. 
Therefore, since $h_N$ induces a homeomorphism (see Defn \ref{defn_hit-or-miss}), $O_A$ is a Borel subset of $\mathcal{C}_{\leq M}(\W)$.
The belief function of a random persistence diagram is similar to the joint cumulative distribution function for a random vector,
in particular by yielding a probability density function through Radon-Nikod\'{y}m type derivatives. 

\begin{defn} \label{defn_set_derivative}
\citep{Matheron} Fix $\phi$ defined on Borel subsets of $\mathcal{C}_{\leq M}(\W)$ into $\R$. 
For an element $\xi \in \W$ or a multiset $Z \subset \W$ with $Z = \LC \xi_1,...,\xi_N \RC$, the set derivative (evaluated at $\emptyset$) is respectively given by
\begin{equation} \label{eqn_set_derivative}
\begin{split}
\frac{\delta \phi}{\delta \xi}(\emptyset) &= \lim_{n \goto \infty} \frac{\phi(B(\xi,1/n))}{\lambda(B(\xi,1/n))}, \\
\frac{\delta \phi}{\delta Z}(\emptyset) &= \frac{\delta^N \phi}{\delta \xi_1 ... \delta \xi_N} = \LB \frac{\delta}{\delta \xi_1} \cdot \cdot \cdot \frac{\delta}{\delta \xi_N}\phi \RB (\emptyset),
\end{split}
\end{equation}
where $B(\xi,1/n)$ are Euclidean balls and $\lambda$ indicates Lebesgue measure on $\W$.
\end{defn}

\begin{remark}
Defn. \ref{defn_set_derivative} for set derivatives at the empty set closely mirrors the Radon-Nikod\'{y}m derivative with respect to Lebesgue measure.
The definition of a set derivative evaluated on a nonempty set is more involved, and is found in \citep{Matheron}. 
Here we are primarily concerned with evaluation at $\emptyset$, since this suffices for the definition of a probability density function.
Also, note that set derivatives satisfy the product rule. 
\end{remark}

\begin{remark}
Restricting to a particular cardinality $N$, consider $\phi_N = \phi \circ h_N$, a function on Euclidean space which is invariant under the action of $\Pi_N$.
The viewpoint of $\phi_N$ elucidates the relationship between set derivatives and Radon-Nikod\'{y}m derivatives with respect to Lebesgue measure. 
This viewpoint also shows that the iterated derivative given in Eq. \eqref{eqn_set_derivative} is independent of order and thus is well-defined for a multiset $Z$. 
\end{remark}

As with typical derivatives, there is a complementary set integration operation for set derivatives.
Set derivatives (at $\emptyset$) are essentially Radon-Nikod\'{y}m derivatives with order tied to cardinality, and so the corresponding set integral acts like Lebesgue integration summed over each cardinality.

\begin{defn} \label{defn_set_integral}
Consider a Borel subset $A$ of $\W$ and a Borel subset $O$ of $\mathcal{C}_{\leq M}(\W)$.
For a set function $f:\mathcal{C}_{\leq M}(\W) \goto \R$, its set integrals over $A$ and $O$ are respectively defined according to the following sums of Lebesgue integrals:
\begin{subequations} 
\begin{align}
\int_A f(Z) \delta Z &= \sum_{N=0}^M \frac{1}{N!} \int_{A^N} f(h_N(\xi_1,...,\xi_N)) d\xi_1 ... d\xi_N, \label{eqn_set_integral_a} \\
\int_O f(Z) \delta Z &= \sum_{N=0}^M \frac{1}{N!} \int_{h_N^{-1}(O)} f(h_N(\xi_1,...,\xi_N)) d\xi_1 ... d\xi_N, \label{eqn_set_integral_b}
\end{align}
\end{subequations}
where $Z = \LC \xi_1,...,\xi_N \RC \subset \W$ is a persistence diagram. 
\end{defn}

\noindent Dividing by $N!$ in Eqs. \eqref{eqn_set_integral_a} and \eqref{eqn_set_integral_b} accounts for integrating over $\W^N$ instead of $\W^N/\Pi_N \cong \mathcal{C}_N(\W)$.
It has been shown that set derivatives and integrals are inverse operations \citep{Matheron};
specifically, the set derivative of a belief function yields a probability density for a random diagram $D$ such that
\begin{equation} \label{eqn_belief_as_pdf}
\beta_D(A) = \int_A \frac{\delta \beta_D}{\delta Z}(\emptyset) \delta Z.
\end{equation}
Indeed, $A^N = h_N^{-1}(\LC D \subset A \RC)$ so that Eq. \eqref{eqn_set_integral_a} also holds as an integral over $O_A = \LC D \in \mathcal{C}_{\leq M} : D \subset A \RC$ in the sense of Eq. \eqref{eqn_set_integral_b}.

\begin{defn} \label{defn_global_pdf}
For a random persistence diagram $D$, a global probability density function (global pdf) $f_D: \cup_{N\in\N} \W^N \goto \R$ must satisfy
\begin{equation} \label{eqn_global_pdf}
\sum_{\pi \in \Pi_N} f_D(\xi_{\pi(1)},...,\xi_{\pi(N)}) = \frac{\delta^N \beta_D}{\delta \xi_1 \cdot ... \cdot \delta\xi_N} (\emptyset).
\end{equation}
and is described by its layered restrictions $f_N = f_D\big{|}_{\W^N}:\W^N \goto \R$ for each $N$.
\end{defn}

\begin{remark} \label{rmk_local_vs_global}
It is necessary to make a distinction between local and global densities because the global pdf is not defined on a single Euclidean space, and is instead expressed as a collection of densities over a range of dimensions. 
Specifically, while each local density $f_N$ (for input cardinality $N$) is defined on $\W^N$, the global pdf $f_D$ is defined on $\cup_{N=1}^M \W^N$ and restricts to a local density on each input dimension.
Each local density $f_N(Z) = f_D\big{|}_{\W^N}(Z)$ decomposes into the product of the conditional density $f_D( Z \big | \Ln Z \Rn = N)$ and the cardinality probability $\P[\Ln Z \Rn = N]$ (this follows from Prop. \ref{prop_belief_layers}).
Thus, each local density does not integrate to one, but instead to the associated probability $\P[\Ln Z \Rn = N]$. 
Also, the global pdf is not a set function and does not require division by $N!$, leading to the following relation:
$\int_{A^N} f_D(\xi_1,...,\xi_N) d\xi_1...d\xi_N = \frac{1}{N!} \int_{A^N} \frac{\delta^N \beta_D}{\delta^N Z} (\emptyset) d\xi_1...d\xi_N.$
\end{remark}
\begin{remark} \label{rmk_unique_symmetric}
While the global pdf and its local constituents need not be symmetric with respect to $\Pi_N$, 
there is a unique choice of global pdf (up to sets of Lebesgue measure 0) which satisfies Eq. \eqref{eqn_global_pdf} and is symmetric under the action of $\Pi_N$.
In this case, we safely abuse notation by denoting $f_D(\LC \xi_1, ..., \xi_N \RC) := N! f_D(\xi_1,...,\xi_N)$ and often write $f_D(Z)$ and allow context to determine whether $Z$ denotes a set or a vector. 
\end{remark}

The following proposition is critical to determine the global pdf for (i) the union of independent singleton diagrams (i.e., $\Ln D^j \Rn \leq 1$), (ii) a randomly chosen cardinality, $N$, followed by $N$ i.i.d. draws from a fixed distribution, and (iii) a random persistence diagram kernel density function. 
The proof of this proposition follows similar arguments to \citep{Nonparametric_Fusion} (Theorem 17, pp. 155--156).

\begin{prop} \label{prop_belief_layers}
Let $D$ be a random persistence diagram with cardinality bounded by $M$ and let $\beta_D(S) = \P(D \subset S)$ be the belief function for $D$. 
Then $\beta_D$ expands as
$$\beta_D(S) = a_0 + \sum_{m=1}^M a_mq_m(S),$$
where $a_m = \P(\Ln D \Rn = m)$ and $q_m(S)=\P[D \subset S \big{|} \Ln D \Rn = m]$.
\end{prop}

\begin{remark} \label{rmk_belief_layers}
The decomposition in Prop. \ref{prop_belief_layers} is often applied as a first step toward finding the local density constituents of the global pdf.
In particular, $f_N = f_D\big{|}_{\W^N} = 0$ for $N > M$. 
\end{remark}

Lastly, we encounter a computationally convenient summary for a random persistence diagram called the probability hypothesis density (PHD).
The integral of the PHD over a subset $U$ in $\W$ gives the expected number of points in the region $U$;
moreover, any other function on $\W$ with this property is a.e. equal to the PHD \citep{Mahler}.

\begin{defn} \label{defn_PHD}
\citep{Matheron} The probability hypothesis density (PHD) for a random persistence diagram $D$ is defined as the set function $F_D(a) = \frac{\delta \beta_D}{\delta Z} \LP\LC a \RC\RP$ and is expressed as a set integral as
\begin{equation} \label{eqn_PHD}
F_D(a) = \int_{\LC Z : \LC a \RC \subset Z \RC} \frac{\delta \beta}{\delta Z}(\emptyset) \delta Z.
\end{equation}
In particular, $\E(\Ln D \cap U \Rn) = \int_U F_D(u) \, du$ for any region $U$.
\end{defn}

\section{Kernel Density Estimation} \label{sect:KDE}

\subsection{Construction} \label{subsect:KDE_construction}

To estimate distributions of persistence diagrams, our goal is the creation of a kernel density function about a center persistence diagram $\mathscr{D}$ with a kernel bandwidth parameter $\sigma>0$, used for defining constituent Gaussians according to Definitions \ref{defn_upper_singletons} and \ref{defn_lower_cluster}.
Prop. \ref{prop_belief_layers} leads to the following lemma which is crucial for determining the kernel density. 
We refer to a random persistence diagram $D$ with $\Ln D \Rn \leq 1$ as a singleton diagram, and such
singletons are indexed by superscripts.

\begin{lemma} \label{lemma_combination}
Consider a multiset of independent singleton random persistence diagrams $\LC D^j \RC_{j=1}^M$. 
If each singleton $D^j$ is described by the value $q^{(j)} = \P[D^j \neq \emptyset]$ and the subsequent conditional pdf, $p^{(j)}(\xi)$, given $\Ln D^j \Rn = 1$, then the global pdf for $D = \cup_{j=1}^M D^j$ is given by
\begin{equation} \label{eqn_combination}
f_D(\xi_1,...,\xi_N) = \sum_{\gamma \in I(N, M)} \mathcal{Q}(\gamma) \prod_{k=1}^N p^{(\gamma(k))}(\xi_{k}),
\end{equation}
for each $N \in \LC 0,..., M \RC$ where
\begin{equation} \label{eqn_QQ}
\mathcal{Q}(\gamma) = \mathcal{Q}^*(\gamma) \prod_{k=1}^N q^{(\gamma(k))},
\end{equation}
$I(N,M)$ consists of all increasing injections $\gamma:\LC 1,...,N \RC \goto \LC 1,...,M \RC$, and
\begin{equation} \label{eqn_Qstar}
\mathcal{Q}^*(\gamma) = \frac{\prod_{j=1}^M (1-q^{(j)})}{\prod_{k=1}^N (1-q^{(\gamma(k))})}.
\end{equation}
\end{lemma}

\begin{proof}
Since the singleton events $D^j$ are independent, the belief function for $D = \cup_j D^j$ decomposes into $\beta_D(S) = \prod_{j=1}^M \beta_{D^j}(S)$.
Next, we employ the product rule for the set derivative (see Defn. \ref{defn_set_derivative}) to obtain the global pdf for $D$ in terms of the singleton belief functions and their first derivatives. 
Higher derivatives of $\beta_{D^j}$ are zero since $D^j$ are singletons (see Remark \ref{rmk_belief_layers}). 
Thus, the product rule yields first derivatives on all (ordered) subsets of the singleton belief functions:
\begin{align*}
\frac{\delta^N \beta_D}{\delta \xi_1 ... \delta \xi_N}(\emptyset) = \sum_{1 \leq j_1 \neq,...,\neq j_N \leq M} \frac{\beta_{D^1}(\emptyset) \cdot \cdot \cdot \beta_{D^M}(\emptyset)}{\beta_{D^{j_1}}(\emptyset) \cdot \cdot \cdot \beta_{D^{j_N}}(\emptyset)}
\LB \frac{\delta \beta_{D^{j_1}}}{\delta \xi_{1}}(\emptyset) \cdot \cdot \cdot \frac{\delta \beta_{D^{j_N}}}{\delta \xi_{N}}(\emptyset)\RB.
\end{align*}
By Prop. \ref{prop_belief_layers}, we have that $\beta_{D^j}(\emptyset) = (1-q^{(j)})$ and $\dfrac{\delta \beta_{D^{j_i}}}{\delta\xi_{i}}(\emptyset) = q_{j_i}p^{(j_i)}(\xi_{i})$ and so
\begin{align*}
\frac{\delta^N \beta_D}{\delta \xi_1 ... \delta \xi_N}(\emptyset) = \sum_{1 \leq j_1 \neq,...,\neq j_N \leq M} 
\LB \frac{\prod_{j=1}^M (1-q^{(j)})}{\prod_{j=1}^N (1-q^{(j_k)})} \prod_{k=1}^N q^{(j_k)} \RB \prod_{k=1}^N p^{(j_k)}(\xi_{k}),
\end{align*}
which nearly resembles Eq. \eqref{eqn_combination}.
To bridge the gap, we describe the choice of indices $j_i$ by an injective function from $\LC 1,...,N \RC$ into $\LC 1,...,M \RC$.
In turn, each such injective function is uniquely determined by the composition of an increasing injection $\gamma \in I(N,M)$ which decides the range of the function and permutations on the domain, $\Pi_N$.
These permutations take into account the order of the range. 
The value of $\mathcal{Q}$ is independent of order, and thus is determined by $\gamma$ as in Eq. \eqref{eqn_QQ}. 
We reorder the product in order to shift these permutations onto the input variables, obtaining
\begin{equation} \label{eqn_precomb}
\frac{\delta^N \beta_D}{\delta \xi_1 ... \delta \xi_N}(\emptyset) = \sum_{\pi \in \Pi_N} \sum_{\gamma \in I(N, M)} \mathcal{Q}(\gamma) \prod_{k=1}^N p^{(\gamma(k))}(\xi_{\pi(k)}).
\end{equation}
Finally, the global pdf in Eq. \eqref{eqn_combination} follows directly from applying Eq. \eqref{eqn_global_pdf} to Eq.\eqref{eqn_precomb}. 
\end{proof}

\begin{remark}
The global pdf in Eq. \eqref{eqn_combination}, and in particular the sum over $\gamma \in I(N,M)$, accounts for each possible combination of singleton presence. 
Moreover, summing over permutations as in Eq. \eqref{eqn_precomb} and dividing by $N!$ yields a symmetric pdf with terms for every possible assignment between singletons and inputs. 
The weights $\mathcal{Q}(\gamma)$ indicate the probability of each assignment occurring, and is the product of the appropriate probability for each singleton to be either present, $q^{(j)}$, or absent, $1-q^{(j)}$, for each $j$. 
\end{remark}

\begin{example} \label{ex_double_single} \rm
Consider two 1-dimensional singleton diagrams, $D^1$ and $D^2$, with probabilities of being nonempty $q^{(1)} = 0.6$ and $q^{(2)} = 0.8$, respectively. The corresponding local densities when nonempty are given by $p^{(1)}(x) = \frac{1}{\sqrt{2\pi}} e^{-(x+1)^2/2}$ and $p^{(2)}(x) = \frac{1}{\sqrt{2\pi}} e^{-(x-1)^2/2}$.
Lemma \ref{lemma_combination} yields the global pdf for $D = D^1 \cup D^2$ through a set of local densities $\LC f_0, f_1(x),f_2(x,y)\RC$ such that $f_0 = \P[\Ln D \Rn = 0] = (1-q^{(1)})(1-q^{(2)}) = 0.08$, $f_1 = f_D\big{|}_{\R}$, and $f_2 = f_D\big{|}_{\R^2}$. 
We sum over permutations and divide by $N!$ ($N=1,2$ is the input cardinality) to obtain a symmetric global pdf. 
\begin{subequations}
\begin{align}
\begin{split}
f_1(x) &= (1-q^{(2)})q^{(1)}p^{(1)}(x) + (1-q^{(1)})q^{(2)}p^{(2)}(x) \\
&= \frac{0.12}{\sqrt{2\pi}}e^{-(x+1)^2/2} +  \frac{0.32}{\sqrt{2\pi}} e^{-(x-1)^2/2} ,\label{eqn_012_A}
\end{split} \\[3mm]
\begin{split}
f_2(x,y) &= \frac{q^{(1)}q^{(2)}}{2} \LB p^{(1)}(x)p^{(2)}(y) + p^{(1)}(y)p^{(2)}(x) \RB \\
&= \frac{0.24}{2\pi} \LP e^{-((x-1)^2 + (y+1)^2)/2} + e^{-((x+1)^2 + (y-1)^2)/2} \RP. \label{eqn_012_B}
\end{split}
\end{align}
\end{subequations}
Accounting for each cardinality and following Eq. \eqref{eqn_012_A} and Eq. \eqref{eqn_012_B}, the total probability adds up to 
\begin{align*}
\P[\Ln D \Rn = 0] + \P[\Ln D \Rn = 1] + \P[\Ln D \Rn = 2] &= f_0 + \int_\R f_1(x) dx + \int_{\R^2} f_2(x,y) dx dy \\
&= (0.08) + (0.12 + 0.32) + (0.24 + 0.24) = 1,
\end{align*}
as desired.
The local densities in Eq. \eqref{eqn_012_A} and Eq. \eqref{eqn_012_B} are plotted in Fig. \ref{fig_012}.
Though $f_1(x)$ is the sum of two Gaussians, in Fig. \ref{fig_012} (Left) we see that the Gaussian centered at $x=1$ dominates, while the Gaussian centered at $x=-1$ is only indicated by a heavy left tail.
This behavior occurs because $q^{(2)} = 0.8$ is very close to $1$.

\begin{figure}
\begin{center}
\includegraphics[scale = 0.4]{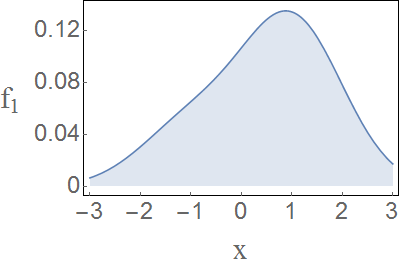} \quad \quad
\includegraphics[scale = 0.4]{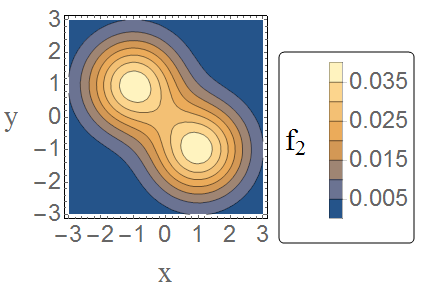}
\end{center}
\caption{\emph{Left:} Plot of the local density $f_1(x)$ in Eq. \eqref{eqn_012_A}. \emph{Right:} Contour plot of the local density $f_2(x,y)$ in Eq. \eqref{eqn_012_B}.
These pdfs cover the different possible input dimensions and are symmetric under permutations of the input.} \label{fig_012}
\end{figure}
\end{example}

Now we turn toward defining the kernel density.
We first define a random persistence diagram as a union of simpler constituents, and then determine its global pdf by combination in a fashion similar to Lemma \ref{lemma_combination}.
Indeed, we define the desired kernel density as the global pdf for this composite random diagram.
To start, we fix a degree of homology $k$ and consider a center diagram $\mathscr{D} \subset \mathcal{W}_k = W \times \LC k \RC$ (see Eq. \eqref{eqn_wedge}). 
Since $k$ is fixed, we treat $\mathscr{D} = \LC \xi_i \RC_{i=1}^M = \LC (b_i,d_i) \RC_{i=1}^M$ within $W = \LC (b,d) \in \R^2 : d > b \geq 0 \RC$.

Long persistence points in a persistence diagram represent prominent topological features which are stable under perturbation of underlying data, and so it is important to track each independently.
In contrast, we leverage the point of view that the small persistence features near the diagonal are considered together as a single geometric signature as opposed to individually important topological signatures.
Toward this end, features with short persistence are grouped together and interpreted through i.i.d. draws near the diagonal. 
Since features cluster near the diagonal in a typical persistence diagram (see, e.g., Fig. \ref{fig_PD_UD_wob} (Right) in the supplementary materials), treating short persistence features collectively simplifies our kernel density and thus speeds up its evaluation. 
It is imperative that these short persistence features are not ignored, because they still capture crucial geometric information for applications such as classification \citep{MaMa16,MaMa16-2, persissensor,fullerene,tda_phase,tda_geo_noise}.
Thus, we split $\mathscr{D}$ into upper and lower portions according to a bandwidth $\sigma$ as
\begin{equation} \label{eqn_split}
\mathscr{D}^u = \LC (b_i,d_i,k) \in \mathscr{D} : d_i-b_i \geq \sigma \RC \textrm{ and }
\mathscr{D}^\ell = \LC (b_i,d_i,k) \in \mathscr{D} : d_i-b_i < \sigma \RC.
\end{equation}

Now define random diagrams $D^u$ centered at $\mathscr{D}^u$ and $D^\ell$ centered at $\mathscr{D}^\ell$ such that $D = D^u \cup D^\ell$.
Ultimately, the global pdf of $D$ centered at $\mathscr{D}$ is our kernel density.

\begin{defn} \label{defn_upper_singletons}
Each feature $\xi_j = (b_j,d_j) \in \mathscr{D}^u$ yields an independent random singleton diagram $D^j$ defined by its chance to be nonempty $q^{(j)}$ (via Eq. \eqref{eqn_nonempty}) along with its potential position $(b,d)$ sampled according to a modified Gaussian distribution, denoted by $N^*((b_j,d_j), \sigma I)$.
The global pdf for $D^u$ is then determined by Lemma \ref{lemma_combination}, where each $p^{(j)}$ is given by the pdf associated with $N^*((b_j,d_j),\sigma I)$, which is given by
\begin{equation} \label{eqn_mod_normal}
p^{(j)}(b,d) = \dfrac{\varphi_j(b,d)}{\int_{W} \varphi_j(u,v) du \, dv} \one_{W}(b,d),
\end{equation}
where $\varphi_j$ is the pdf of the (unmodified) normal $N((b_j,d_j),\sigma I)$, and $\one_W(\cdot)$ is the indicator function for the wedge. 
\end{defn}

The global pdf for each $D^j$ is readily obtained by a pair of restrictions. 
First, we restrict the usual Gaussian distribution to the halfspace \hbox{$T = \LC (b,d) \in \R^2 : b < d \RC$}. 
Features sampled below the diagonal are considered to disappear from the diagram and thus we define the chance to be nonempty by 
\begin{equation} \label{eqn_nonempty}
q^{(j)} = \P(D^j \neq \emptyset) = \int_{\LC v > u \RC} \varphi_j(u,v) \, du \, dv.
\end{equation}
Afterward, the Gaussian restricted to $T$ is further restricted to $W$ and renormalized to obtain a probability measure as in Eq. \eqref{eqn_mod_normal}.
This double restriction to both $T$ and $W$ is necessary for proper restriction of the Gaussian pdf and definition of $q^{(j)}= \P(D^j \neq \emptyset)$. 
Indeed, restriction to $W$ alone causes points with small birth time to have an artificially high chance to disappear; while restriction to $T$ alone yields nonsensical features with negative radius (with $b < 0$). 
In kernel density estimation, the effects of this distinction become negligible as the bandwidth goes to zero. 
In practice, this distinction is important for features with small birth time relative to the bandwidth. 

\begin{remark}
In the $\cech$ construction of a persistence diagram, a feature lies on the line $b = 0$ if and only if it has degree of homology $k=0$. 
Consequently, for a feature $(0,d_j)$ with $k=0$, we instead take
$$ p^{(j)}(d) = \frac{\phi_j(d)}{\int_{\R^+} \phi_j(u) \, du} \one_{\R^+}(d) \textrm{ and } q^{(j)} = \int_{\R^+} \phi_j(u) \, du $$
where $\phi_j$ is the 1-dimensional Gaussian centered at $d_j$ with standard deviation $\sigma$. 
\end{remark}

Whereas the large persistence features in $D^u$ have small chance to fall below the diagonal and disappear, the existence of the small persistence features in $D^\ell$ is volatile: these features disappear and appear fluidly under small changes in the underlying data. 
The distribution of $D^\ell$ is described by a probability mass function (pmf) $\nu$ and lower density $p^\ell$. 

\begin{defn} \label{defn_lower_cluster}
The lower random diagram $D^\ell$ is defined by choosing a cardinality $N$ according to a pmf $\nu$ followed by $N$ i.i.d. draws according to a fixed density $p^\ell$.
First, take $N_\ell = \Ln \mathscr{D}^\ell \Rn$ and define $\nu(\cdot)$ with mean $N_\ell$ and so that $\nu(n) = 0$ for $n > mN_\ell$ for some $m >0$ independent of $N_\ell$. 
The subsequent density $p^\ell(b,d)$ is given by projecting the lower features $\mathscr{D}^\ell$ of the center diagram $\mathscr{D}$ onto the diagonal $b = d$, then creating a restricted Gaussian kernel density estimation for these features; specifically,
\begin{equation} \label{eqn_lower_density}
p^\ell(b,d) = \frac{1}{N_\ell} \sum_{(b_i,d_i) \in \mathscr{D}^\ell} \frac{1}{\pi\sigma^2} e^{-\LP\LP b - \frac{b_i+d_i}{2} \RP^2 + \LP d - \frac{b_i+d_i}{2} \RP^2\RP/2\sigma^2}.
\end{equation}
\end{defn}

Projecting the lower features $\mathscr{D}^\ell$ of the center diagram $\mathscr{D}$ onto the diagonal simplifies later analysis and evaluation of $p^\ell$;
without projecting, a unique normalization factor, similar to $q^{(j)}$ in Defn. \ref{defn_upper_singletons}, would be required for each Gaussian summand in Eq. \eqref{eqn_lower_density}.
By Prop. \ref{prop_belief_layers} and Eq. (\ref{eqn_global_pdf}), global pdfs of random persistence diagrams are described by a random vector pdf for each cardinality layer, resulting in the following global pdf for $D^\ell$:
\begin{equation} \label{eqn_lower_dist}
f_{D^\ell}(\xi_1,...,\xi_N) = \nu(N) \prod_{j=1}^N p^\ell(\xi_j).
\end{equation}
Combining the expressions for $D^\ell$ and $D^u$, we arrive at the following proposition.

\begin{theorem} \label{thm_construction}
Fix a center persistence diagram $\mathscr{D}$ and bandwidth $\sigma>0$.
Split $\mathscr{D}$ into $\mathscr{D}^\ell$ and $\mathscr{D}^u$ according to Eq. \eqref{eqn_split}.
Define $D^\ell$ with global pdf from Eq. \eqref{eqn_lower_dist}, and $D^u$ with global pdf from Eq. \eqref{eqn_combination}.
Treating the random persistence diagrams $D^u$ and $D^\ell$ as independent, the kernel density centered at $\mathscr{D}$ with bandwidth $\sigma$ is given by
\begin{equation} \label{eqn_construction}
K_\sigma(Z, \mathscr{D}) = \sum_{j=0}^{N_u} \nu(N-j) \sum_{\gamma \in I(j,N_u)} \mathcal{Q}(\gamma) \prod_{k=1}^j  p^{(\gamma(k))}(\xi_{k}) \prod_{k=j+1}^N  p^\ell(\xi_{k}),
\end{equation}
where $Z= (\xi_1,...,\xi_N)$ is the input, $\xi_i = (b_i,d_i)$ for $i = 1, ..., N$ are the features, and $N_u = \Ln \mathscr{D}^u \Rn$ depends on both $\mathscr{D}$ and $\sigma$. 
Here $\mathcal{Q}(\gamma)$ is given by Eq. \eqref{eqn_QQ}, each $p^{(j)}$ refers to the modified Gaussian pdf as shown in Eq. \eqref{eqn_mod_normal} for its matching feature $\xi_j$ in $D^u$, and $p^\ell$ is given by Eq. \eqref{eqn_lower_density}.
\end{theorem}
\begin{proof}
Since $D^u$ and $D^\ell$ are independent random persistence diagrams, the belief function decomposes into $\beta_D(S) = \beta_{D^u}(S) \beta_{D^\ell}(S)$. 
Moreover, since derivatives above order $N_u$ vanish for $\beta_{D^u}$ (see Remark \ref{rmk_belief_layers}), the product rule and binomial-type counting yield
\begin{equation} \label{eqn_RND_KDE}
\begin{split}
\frac{\delta^N\beta_D}{\delta \xi_1 ... \delta \xi_N}(\emptyset) &= \sum_{j=0}^{N_u} \sum_{1\leq i_1 \neq ... \neq i_j \leq N} \frac{\delta^j \beta_{D^u}}{\delta \xi_{i_1} ... \delta \xi_{i_j}}(\emptyset) \frac{\delta^{N-j} \beta_{D^\ell}}{\delta \xi_1 ... \hat{\delta \xi_{i_1}} ... \hat{\delta \xi_{i_j}} ... \delta \xi_N}(\emptyset) \\
 &= \sum_{\pi \in \Pi_N} \sum_{j=0}^{N_u} \frac{1}{j!(N-j)!} \frac{\delta^j \beta_{D^u}}{\delta \xi_{\pi(1)} ... \delta \xi_{\pi(j)}}(\emptyset) \frac{\delta^{N-j} \beta_{D^\ell}}{\delta \xi_{\pi(j+1)} ... \delta \xi_{\pi(N)}}(\emptyset)
\end{split}
\end{equation}
where $\delta\hat{\xi_i}$ indicates that the given index is skipped in the set derivative (having been allocated to the other factor). 
Similar to the proof of Lemma \ref{lemma_combination}, the choice of indices $i_j$ is replaced with a permutation $\pi \in \Pi_N$;
however, the ordering within each derivative is unrelated the choice of $i_j$, leading to $j!$-fold and $(N-j)!$-fold redundancy within each term. 

Taking Eq. \eqref{eqn_lower_dist} together with Eq. \eqref{eqn_global_pdf} yields
$$\frac{\delta \beta_{D^\ell}}{\delta \xi_{\pi(j+1)} ... \delta \xi_{\pi(N)}}(\emptyset) = (N-j)! \nu(N-j) \prod_{j=1}^{N-j} p^\ell(\xi_j). $$
Also, Eq. \eqref{eqn_combination} and Eq. $\eqref{eqn_global_pdf}$ yield
$$ \frac{\delta \beta_{D^u}}{\delta \xi_{\pi(1)} ... \delta \xi_{\pi(j)}}(\emptyset) = \sum_{\pi^* \in \Pi_j} \sum_{\gamma \in I(j,N_u)} \mathcal{Q}(\gamma) \prod_{k=1}^j  p^{(\gamma(k))}(\xi_{\pi^*(k)}). $$
We substitute these relations into the final expression of Eq. \eqref{eqn_RND_KDE}.
The first of these substitutions is straightforward, while the second has $j!$-fold redundant permutations overtop the existing permutations in $\Pi_N$.
These substitutions yield that $\frac{\delta^N\beta_D}{\delta \xi_1 ... \delta \xi_N}(\emptyset) = \sum_{\pi\in\Pi_N} K_\sigma(Z,\mathscr{D})$ as described in Eq. \eqref{eqn_construction} and shows that the kernel $K_\sigma(Z,\mathscr{D})$ satisfies the definition of a global pdf for $D$ (Defn. \ref{defn_global_pdf}). 
Finally, the sum over permutations is removed according to Eq. \eqref{eqn_global_pdf} to obtain the expression for $f_D(Z) = K_\sigma(Z,\mathscr{D})$.
\end{proof}

\begin{remark} \label{rmk_KDE_addition}
	A specific example of the component distributions provided for the kernel in Thm. \ref{thm_construction} is presented in Fig. \ref{heat}.
	Since the kernel density $K_\sigma$ of Eq. \eqref{eqn_construction} is a probability density according to Defn. \ref{defn_global_pdf}, it is a function on $\cup_{N=0}^M \W^N$, and so the sum of several such kernels is defined by adding each local pdf layer separately.
\end{remark}
\vspace{-5mm}
\begin{remark} \label{rmk_split_reason}
	Each feature in the upper random persistence diagram is described independently, while all the features in the lower random persistence diagram are described by a single density $p^\ell$.
	Evaluation of the kernel density in Eq. \eqref{eqn_construction} is made rapid by factoring the repeated (product) evaluation of $p^\ell$, a typical 2D Gaussian KDE, despite the kernel's definition in a very high-dimensional space.
	Indeed, while evaluation computation increases exponentially when upper features are added (of which there should be few), it only increases linearly for additional lower features.
	Furthermore, in datasets with too many points, one typically subsamples, e.g. by min-max sampling, to reduce the computational burden of calculating the persistence diagram itself (e.g., see \citep{tda_subsample}), yielding a persistence diagram with fewer features.
\end{remark}
\vspace{-5mm}
\begin{remark} \label{rmk_split}
	In the definition of our kernel, a single parameter $\sigma$ has been chosen for both the split of center diagrams, as well as the standard deviation used in the Gaussians which build our kernel.
	Without loss of generality, this choice simplifies the presentation of the kernel density and the proof of kernel density estimate (KDE) convergence (Theorem \ref{thm_KDE}).
	In general, the bandwidth parameter $\sigma_2$ which refers to the standard deviation used to define the Gaussians (as $\sigma$ appears in Defs. \ref{defn_upper_singletons} and \ref{defn_lower_cluster}) need not be equal to the splitting parameter $\sigma_1$ which determines which points are in $\mathscr{D}^u$ or $\mathscr{D}^\ell$ (as $\sigma$ appears in Eq. \eqref{eqn_split}). 
	Still, it is certainly desirable that $\sigma_1 = C \sigma_2$ when taking a limit of KDEs as the number of persistence diagrams grows to infinity (Theorem \ref{thm_KDE}).
	For a fixed kernel bandwidth $\sigma_2$, increasing $C$ (and thus $\sigma_1$) moves more features into the lower portion of the diagram.
	This choice may be useful in practice when underlying data are known to be noisy and more noise-related features are expected near the diagonal. 
	By the same token, for $\sigma_1 >> \sigma_2$, projecting the lower features onto the diagonal may lead to significant error in the approximation.
	On the other hand, taking $\sigma_1 << \sigma_2$ eliminates the computational benefit of splitting the diagram and is probably not useful in practice.
	For most cases, taking $\sigma_1 = \sigma_2$, is a reasonable balance between KDE accuracy and evaluation computation.
\end{remark}
\vspace{-5mm}
\begin{remark} \label{rmk_computation}
	Since the associations dictated by $\gamma \in I(j,N_i)$ in Thm. \ref{thm_construction} are known a priori, the calculation is embarrassingly parallelizable, and computation can be made rapid even for the evaluation of the global density function associated with a diagram $D$ with many features (see, e.g., Fig. \ref{fig_PD_UD_wob} (Right) in the supplementary materials).
	Nevertheless, the density described in Eq. \eqref{eqn_construction} is well organized for approximate evaluation.
	While Eq. \eqref{eqn_construction} is sufficient for set integration, it is not symmetric under permutations of the inputs $\xi_i$, and consequently does not represent the density at a set $\LC \xi_1,..., \xi_N \RC$.
	A symmetric version is desirable for methods such as maximum likelihood or mode estimation \citep{bayesbook_2014}. 
	Indeed, a symmetric pdf is available by summing over $\Pi_N$ as per Eq. \eqref{eqn_global_pdf} to obtain the set derivative of the belief function.
	At no loss of accuracy, the large sum over $\Pi_N$ need not range over all permutations and one may instead sort over compositions $\beta \circ \pi$ for $\beta \in I(j,N)$ and $\pi \in \Pi_j$ for each $j \in \LC 0,...,\Ln D^u \Rn \RC$.  
	Since one expects that $\Ln D^u \Rn = N_u << N$, this reorganization significantly diminishes the number of computations. 
\end{remark}

\begin{figure}[h]
	\centering{\includegraphics[width=0.3\textwidth]{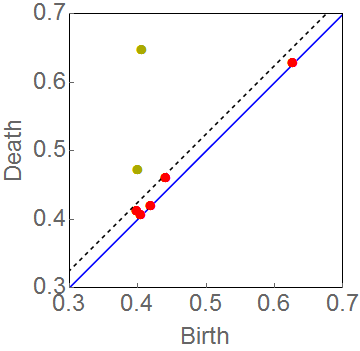} \hspace{3mm}
	\includegraphics[width=0.3\textwidth]{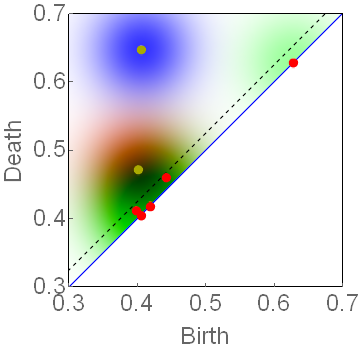}}
	\caption{\emph{Left:} A persistence diagram split according to Eq. \eqref{eqn_split}.
	The dashed black line, $d=b+\sigma$, separates the diagram into the red upper points of $\mathscr{D}^u$ and the yellow lower points of $\mathscr{D}^\ell$.
	\emph{Right:} The red and blue gradients represent the upper singleton densities $p^{(1)}$ and $p^{(2)}$ given by Eq. \eqref{eqn_mod_normal}. 
	The green gradient represents the lower density $p^\ell$ defined in Eq \ref{eqn_lower_density}. 
	While each of these densities is defined on the wedge $W \subset \R^2$, the global kernel in Eq. \eqref{eqn_construction} is defined on $\bigcup_N W^N$ for each input-cardinality $N$.}
	\label{heat}
\end{figure}

Since the kernel density is a probability density function for a random persistence diagram, it has an associated probability hypothesis density (See Defn. \ref{defn_PHD}).

\begin{cor} \label{cor_kernel_PHD}
	Fix a center persistence diagram $\mathscr{D}$ and bandwidth $\sigma>0$.
	Split $\mathscr{D}$ into $\mathscr{D}^\ell$ and $\mathscr{D}^u$ according to Eq. \eqref{eqn_split}.
	Define $D^\ell$ with global pdf from Eq. \eqref{eqn_lower_dist}, and $D^u$ with global pdf from Eq. \eqref{eqn_combination}.
	Treating the random persistence diagrams $D^u$ and $D^\ell$ as independent, the probability hypothesis density (PHD) associated with the kernel density centered at $\mathscr{D}$ with bandwidth $\sigma$ of Thm. \ref{thm_construction} is given by
	\begin{equation} \label{eqn_kernel_PHD}
	K_{\sigma,PHD}(\xi, \mathscr{D}) = N_\ell \, p^\ell(\xi) 
	+ \sum_{j=1}^{N_u} q^{(j)} p^{(j)}(\xi),
	\end{equation}
	where the feature $\xi$ is the input and $N_u = \Ln \mathscr{D}^u \Rn$ and $N_\ell = \Ln \mathscr{D}^\ell\Rn$ depend on both $\mathscr{D}$ and $\sigma$. 
	Here each $p^{(j)}$ refers to the modified Gaussian pdf as shown in Eq. \eqref{eqn_mod_normal} for its matching singleton feature $\xi_j$ in $D^u$, $q^{(j)}$ given by \eqref{eqn_nonempty} is the probability each singleton is present, and the lower density $p^\ell$ is given by Eq. \eqref{eqn_lower_density}.
\end{cor}
\begin{proof}
	The PHD is uniquely defined by its integral over a region $U$, which yields the expected number of points in the region.
	Consequently, the independent upper and lower random draws which build the kernel contribute additively to the PHD.
	Within the sum, each singleton density $p^{(j)}$ is weighted by the chance for $D^j$ to be present, $q^{(j)}$ and the lower density $p^\ell$ is weighted according to the mean draw cardinality, which was chosen to be $\Ln \mathscr{D}^\ell \Rn$. 
\end{proof}

Notice that in Cor. \ref{cor_kernel_PHD}, the input for the PHD is a single feature $\xi$ as opposed to a list of features $Z = \LC \xi_1, ..., \xi_N \RC$ for the global kernel in Thm. \ref{thm_construction}.
Furthermore, Thm. \ref{thm_construction} extends to the analogue result for a center persistence diagram with features of varied degree of homology.

\begin{cor} \label{cor_general_kernel}
Consider a persistence diagram $\mathscr{D} = \bigcup_{k=0}^{\dim-1} \mathscr{D}_k \times \LC k \RC$ split according to the degrees of homology with associated random persistence diagrams $D_k$ defined according to Eq. \eqref{eqn_construction} for each center diagram $\mathscr{D}_k$.
Treating each $D_k$ as independent, the full global pdf for $D = \bigcup D_k$ centered at $\mathscr{D}$ with bandwidth $\sigma$ is given by
\begin{equation} \label{eqn_general_kernel}
K_\sigma(Z,\mathscr{D}) = \Lambda(N) \prod_{k=0}^{\dim-1} K_\sigma(Z_k,\mathscr{D}_k),
\end{equation}
where $Z = \bigcup_{k=0}^{\dim-1} Z_k \times \LC k \RC \subset \W$ with each $Z_k \subset W$ of cardinality $\Ln Z_k \Rn = N_k$ within the multi-index $N = (N_0,...,N_{\dim-1})$ and 
$$ \Lambda(N) = \frac{N!}{\Ln N \Rn !} := \frac{\prod N_k!}{\LP \sum N_k\RP !}.$$
\end{cor}
\begin{proof}
The result follows immediately from taking set derivatives of the full belief function $\beta_D(S) = \prod_k \beta_{D_k}(S)$.
In particular, the set derivatives $\frac{\delta\beta_{D_k}}{\delta Z}(\emptyset)$ are zero unless $Z \subset \mathcal{W}_k$. 
Thus, the product rule leaves only the single term
$\frac{\delta\beta_{D}}{\delta Z}(\emptyset) = \prod_{k=0}^{\dim-1} \frac{\delta\beta_{D_k}}{\delta Z_k}(\emptyset)$.
In turn, each kernel global pdf $K_\sigma(Z_k,\mathscr{D}_k)$ is related to the associated belief function derivative by a sum over permutations $\Pi_{N_k}$ (see Eq. \eqref{eqn_global_pdf}). 
Compositions of these permutations are $N_k!$-fold redundant against the $\Ln N \Rn !$ permutations in $\Pi_{\Ln N \Rn}$, yielding the coefficient $\Lambda(N)$.
\end{proof}

Next, to prove the convergence (to the target distribution) of the kernel density estimate defined via the kernel established in Thm. \ref{thm_construction}, we consider persistence diagrams $\LC \mathscr{D}_i \RC_{i=1}^n$ which are i.i.d. sampled from a target distribution with global pdf $f$.
Toward this end, we require the following assumptions on $f$:
\begin{align*}
(A&1) \,\, f(Z) = 0 \textrm{ for } \Ln Z \Rn > M \in \N \textrm{ (bounded cardinality).}\\
(A&2) \,\, \textrm{The local density } f_N:\mathcal{W}_k^N \goto \R \textrm{ is bounded for each } N \in \LC 1,...,M \RC\!. \\
(A&3) \,\, \textrm{There exists } C_N >0 \textrm{ so that } f(\xi_1,...,\xi_N) \leq C_N \LN (\xi_1,...,\xi_N) \RN^{-2N} \textrm{ for each } N \in \LC 1,...,M \RC.
\end{align*}
	
The assumptions (A1), (A2), and (A3) describe conditions on the target random persistence diagram pdf.
It is important that these assumptions also hold for a random persistence diagram associated with typical (random) underlying datasets. 
For example, $(A1)$ trivially holds for underlying data in $\R^{\dim}$ of bounded cardinality. 
The conditions $(A2)$ and $(A3)$ hold for underlying data sampled from a compact set $E \subset \R^{\dim}$ perturbed by Gaussian noise.
The work \citep{tda_crackle}(see Cor. 2.3 and Thm 2.6 therein) describes the persistent homology of noise, and describes a `core' neighborhood.
Specifically for Gaussian noise, features are retained in the `core', but then extreme decay occurs for features of arbitrary degree outside the `core'.
Intuitively, by bounding death values by the diameter of the underlying dataset, one expects that the decay will be at worst a polynomial times Gaussian decay, which is sufficient for (A3).

The following theorem shows that the kernel density estimate converges to the true global pdf of a random persistence diagram as the number of persistence diagrams increases.
The pdf tracks not only the birth and death of features, but also their prevalence.
In particular, the persistence diagram pdf tied to a random dataset can determine which geometric features are stable regardless of their persistence.
The proof of this theorem is delegated to the supplementary materials.

\begin{theorem} \label{thm_KDE}
	Consider a random persistence diagram global pdf $f$ satisfying assumptions $(A1)$-$(A3)$. 
	Define the kernel $K_\sigma(Z,\mathscr{D})$ according to Thm. \ref{thm_construction} and consider the kernel density estimate $\hat{f}(Z) = \frac{1}{n} \sum_{i=1}^n K_\sigma(Z,\mathscr{D}_i)$, with centers $\mathscr{D}_i$ sampled i.i.d. according to global pdf $f$ and bandwidth $\sigma = O(n^{-\alpha})$ chosen with $0 < \alpha < \alpha_{2M}$.
	Then, as $n \goto \infty$, $\hat f \goto f$ uniformly on compact subsets of $W$.
\end{theorem}
\begin{remark} \label{rmk_bandwidth}
	The value of $\alpha_{2M}$ is inherited from bandwidth selection for $2M$-dimensional kernel density estimates \citep{KDE_book}.
	While the scaling of the bandwidth in the limit is determined by the maximum cardinality $M$ (and thus, the largest dimension of the local pdfs), choosing a bandwidth for a specific sample is an important step in applying kernel density estimation.
	If the bandwidth is too narrow, the estimate is overfitted and potentially biased; 
	if the bandwidth is too large, the estimate will be oversmoothed, resulting in accuracy loss.
	Several methods for bandwidth selection in multivariate kernel estimation are discussed in \citep{silverman}. 
	As a general rule of thumb, \citep{silverman} recommends choosing the bandwidth as $\sigma_{opt} = A(K) n^{-1/({\dim}+4)}$, where $n$ is the sample size (i.e., the number of persistence diagrams), ${\dim}$ is the dimension, and $A(K)$ is a constant depending on the kernel, $K$.
	In particular, one may choose $\alpha \approxeq 1/(2M+4)$ as an unbiased estimator for all local pdfs with cardinalities $m \leq M$ \citep{KDE_book}.
	Silverman's rule of thumb works best for distributions which are nearly Gaussian;
	for more general distributions, the bandwidth may be chosen empirically.
\end{remark}

\subsection{Examples} \label{subsect:Examples}
Here we provide detailed examples of the kernel density and kernel density estimation of an unknown pdf.
For simplicity, we restrict to a single degree of homology, say $k=1$. 
Due to the intrinsic high dimension of the kernel, we present contour plots for slices of the kernel density.
Specifically, for inputs $\LP (b_1,d_1),...,(b_N,d_N) \RP$, we consider the kernel density evaluated at $(b_1,d_1) \in W$ with $(b_i,d_i)$ fixed for $i \geq 2$.
For clarity, the unique symmetric pdf $f_{sym}(\xi_1,...,\xi_N) = \frac{1}{N!} \sum_{\pi \in \Pi_N} f(\xi_{\pi(1)},...,\xi_{\pi(N)})$ is used in the contour plots (see Remark \ref{rmk_unique_symmetric}). 
For explicit computation, we choose the probability mass function
\begin{equation} \label{eqn_pmf_card}
\nu(N) = \max \LC \frac{N_\ell + 1 - \Ln N_\ell - N \Rn}{(N_\ell + 1)^2}, 0 \RC
\end{equation}
when evaluating the lower density in Eq. \eqref{eqn_lower_dist}, where $N_\ell = \Ln \mathscr{D}^\ell \Rn$ is the lower cardinality of the center diagram. 
This probability mass function is chosen to satisfy the requirements of Defn. \ref{defn_lower_cluster}, and specifically has the property that $\nu(N) > 0$ for $0 \leq N \leq 2\Ln \mathscr{D}^\ell \Rn$. 

\begin{example} \label{ex1} \rm
	Consider the center persistence diagram $\mathscr{D} = \LC (1,3), (2,4), (1,1.3),(3,3.2) \RC \subset W$ and bandwidth $\sigma = 1/2$. 
	We construct the associated kernel density $K_\sigma(Z,\mathscr{D})$ according to Thm. \ref{thm_construction} and follow with some plots and analysis of the kernel density.
	The random persistence diagram $D$ associated with the kernel density $K_\sigma(Z,\mathscr{D})$ has a variable number of features $N = \Ln D \Rn$;
	consequently, the input diagram $Z = \LC\xi_1,...,\xi_N\RC$ must have variable length and therefore the kernel density has local definitions (see Rmk. \ref{rmk_local_vs_global}) on $W^N$ for each possible input cardinality $N$.
	
	Since each modified Gaussian $p^{(j)}$ (Defn. \ref{defn_upper_singletons}) and the lower density $p^\ell$ (Defn. \ref{defn_lower_cluster}) integrate to 1 over the wedge $W$, an expression for the probability mass function (pmf) $\P[\Ln D \Rn = N]$ can be expressed solely in terms of $\nu$ and $q^{(j)}$:
	\begin{equation}
	\begin{split}
	\P[\Ln D \Rn = N] &= \LB q^{(1)}q^{(2)}\RB \nu(N-2) \\
	&+ \LB q^{(1)}\LP 1-q^{(2)}\RP + q^{(2)}\LP 1-q^{(1)}\RP \RB \nu(N-1) \\
	&+ \LB\LP 1-q^{(1)}\RP \LP 1-q^{(2)}\RP\RB \nu(N)
	\end{split}
	\end{equation}
	The plot of this pmf is shown in Fig. \ref{pmf}.
	Recall that $D = D^u \cup D^\ell$, so that $\Ln D \Rn = \Ln D^u \Rn + \Ln D^\ell \Rn$;
	since $q^{(j)} \approx 1$ for $j=1,2$, $\Ln D^u \Rn = 2$ with high probability and the pmf $\P[\Ln D \Rn = N]$ is nearly the pmf for $\Ln D^\ell \Rn$, $\nu$, shifted up by 2 units.
	Fig. \ref{pmf} suggests that understanding the kernel density requires investigation into higher cardinality inputs.
	In general, it is important to consider input diagrams $Z$ with $\Ln Z \Rn \geq \Ln \mathscr{D}^u \Rn$.
	
	\begin{figure}[h!]
		\begin{center}
			\includegraphics[scale = 0.25]{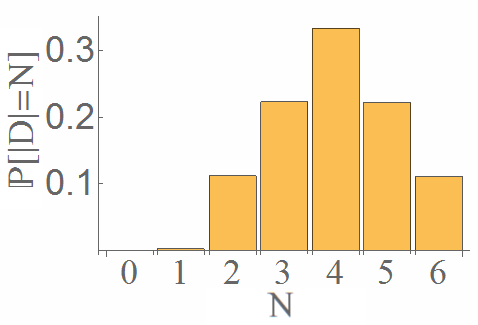}
		\end{center}
		\vspace{-4mm}
		\caption{Cardinality probabilities $\P[ \Ln D \Rn = N]$ for random diagram $D$ distributed according to global pdf $K_\sigma(\cdot,\mathscr{D})$ in Ex. \ref{ex1}.
		In general, we have that $0 \leq \Ln D^u \Rn \leq \Ln \mathscr{D}^u \Rn$ and according to Eq. \eqref{eqn_pmf_card}, $\nu(N) \neq 0$ for $0 \leq N \leq 2 \Ln \mathscr{D}^\ell \Rn$.
		Thus, the cardinality $\Ln D \Rn = \Ln D^u \Rn + \Ln D^\ell \Rn$ takes on values between $0$ and $6 = \Ln \mathscr{D}^u \Rn + 2 \Ln \mathscr{D}^\ell \Rn$.} \label{pmf} 
	\end{figure}

	First, we describe the random diagram associated to the lower features $\mathscr{D}^\ell = \LC (1,1.3), (3,3.2) \RC$ of the center diagram $\mathscr{D}$. 
	The lower random diagram $D^\ell$ is described in Defn. \ref{defn_lower_cluster} according to a probability mass function (pmf) $\nu$ for the cardinality of $D^\ell$ and a single probability density $p^\ell(b,d)$ for the subsequent features' locations in the wedge $W$. 
	The pmf $\nu$ is defined according to Eq. \eqref{eqn_pmf_card} with $N_\ell = 2$; that is, $\nu(\LC 0,1,2,3,4 \RC) = \LC 1/9, 2/9, 3/9, 2/9, 1/9 \RC$ respectively, and zero otherwise. 
	Following Defn. \ref{defn_lower_cluster}, we project the features of $\mathscr{D}^\ell$ onto the diagonal to obtain $\LC (1.15,1.15), (3.1, 3.1) \RC$.
	Relying on Eq. \eqref{eqn_lower_density}, the resulting lower density is given by 
	\begin{equation} \label{eqn_ex_pl}
		p^\ell(b,d) = \frac{2}{\pi} \LB e^{- \LP (b-1.15)^2 + (d-1.15)^2 \RP} + e^{-\LP (b-3.1)^2 + (d-3.1)^2 \RP} \RB.
	\end{equation}
	restricted to the wedge $W$.
	The coefficient $\frac{2}{\pi}$ is obtained by a direct substitution into Eq. \eqref{eqn_lower_density}.

	Due to the flexible input cardinality, the kernel will be expressed and plotted separately for different input cardinalities.
	For brevity, we present the local kernels on $W^N \subset \R^{2N}$ for cardinalities $N = 1, 2, 3$.
	First, we consider the probability hypothesis density (or PHD, as defined in Eq. \eqref{eqn_PHD}) along with the kernel density evaluated at a single input feature in Fig. \ref{PHD_vs_1F}.
	Recall that the integral of the PHD over a region $U$ yields the expected number of features in $U$ (see Defn \ref{defn_PHD}).
	The kernel's corresponding PHD is a sum of Gaussians as described in Cor. \ref{cor_kernel_PHD}.
	\begin{equation} \label{eqn_phd_exp}
		\begin{split}
		K_{\sigma,PHD}((b,d),\mathscr{D}) &= 2 p^\ell(b,d) + q^{(1)}p^{(1)}(b,d)+ q^{(2)}p^{(2)}(b,d) \\
		&= 1.273 \LP e^{-2 \LP (b-3.1)^2 + (d-3.1)^2 \RP}+e^{-2 \LP (b-1.15)^2 + (d-1.15)^2 \RP} \RP \\
		&\hspace{5mm} + 0.635 e^{-2 \LP (b-2)^2 + (d-4)^2 \RP} + 0.635 e^{-2 \LP (b-1)^2 + (d-3)^2 \RP}.
		\end{split}
	\end{equation}

	Next, for input of cardinality $\Ln Z \Rn = 1$, we obtain an easily viewable 2-dimensional distribution. 
	Thm. \ref{thm_construction} yields the following expression:
	\begin{equation} \label{eqn_pdf_f1}
		\begin{split}
		K_\sigma((b_1,d_1),\mathscr{D}) &= \nu(0)\LB(1 - q^{(2)})q^{(1)} p^{(1)}(b_1,d_1)
		+ (1 - q^{(1)})q^{(2)} p^{(2)}(b_1,d_1) \RB \\ 
		&\hspace{5mm} + \nu(1)\LB(1 - q^{(1)})(1 - q^{(2)})p^\ell(b_1,d_1)\RB. \\
		&= 7.74 \times 10^{-2} \LP e^{-2 \LP (b_1-2)^2 + (d_1-4)^2 \RP} + e^{-2 \LP (b_1-1)^2 + (d_1-3)^2 \RP} \RP. \\
		&\hspace{5mm} + 1.65 \times 10^{-4} p^\ell(b_1,d_1).
		\end{split}
	\end{equation}
	The kernel is treated as a global pdf as in Prop. \ref{defn_global_pdf} and Rmk. \ref{rmk_local_vs_global}; thus, this 2-D density is only a local density for the whole kernel.
	Each term is a weighted product of the combination of upper features considered (In order: $(2,4)$, $(1,3)$, or none.).
	Since the values of $q^{(j)}$ are very close to 1, terms which include the upper pdfs $p^{(j)}$ have much larger total mass. 

	Contour plots of the densities expressed in Eqs. \eqref{eqn_phd_exp} and \eqref{eqn_pdf_f1} (restricted to $W$) are respectively shown in Figs. \ref{PHD_vs_1F}(a) and \ref{PHD_vs_1F}(b).
	In Fig. \ref{PHD_vs_1F}(a), the PHD indicates that in general, as many features will appear near the diagonal as will appear near the upper features.
	According to the local kernel shown in Fig. \ref{PHD_vs_1F}(b), if only a single feature is present, this feature is far more likely to have long persistence. 
	Indeed, the kernel density is defined (see Eq. \eqref{eqn_construction}) so that the number of points near the diagonal is fluid (by our choice of $\nu$), whereas the probability of each feature in the upper diagram is nearly 1.
	In essence, this demonstrates that the kernel density naturally considers features with long persistence to be stable or prominent in density estimation.

	\begin{figure}[h!]
		\begin{center}
			\begin{multicols}{2}
				\includegraphics[scale = 0.24]{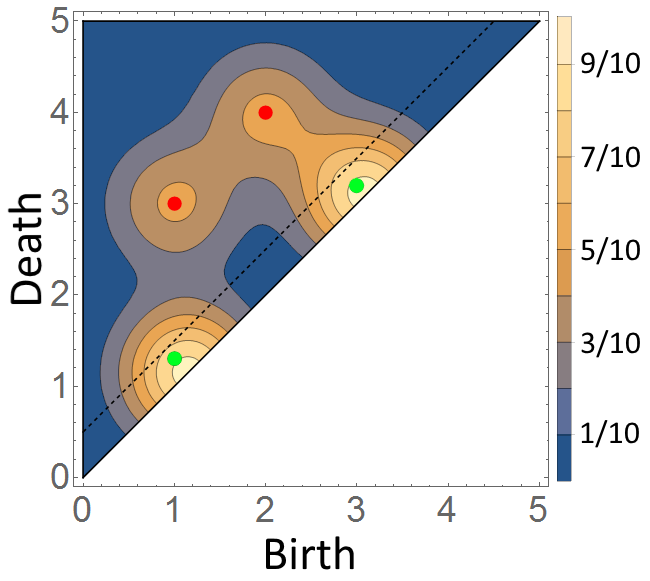} \\ (a) \\
				\includegraphics[scale = 0.24]{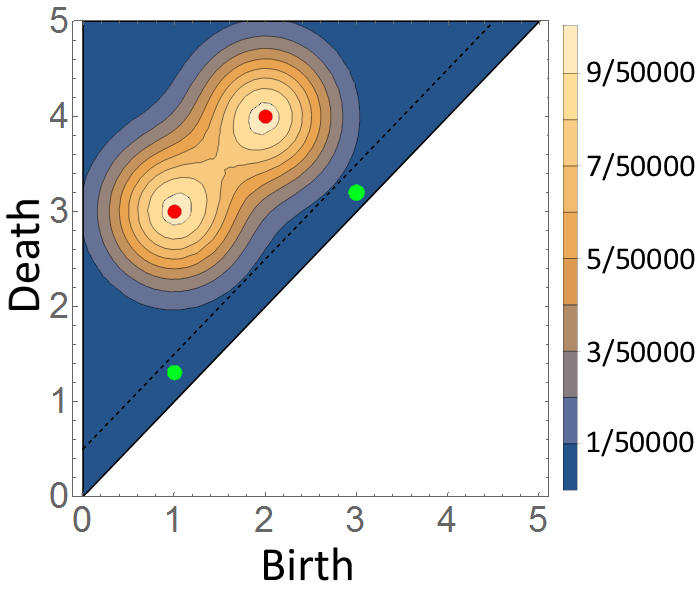} \\ (b)
			\end{multicols}
		\end{center}
		\vspace{-3mm}
		\caption{Contour maps for (a) the probability hypothesis density associated to the kernel density (Eq. \eqref{eqn_phd_exp}) and (b) the kernel density restricted to a single input feature (Eq. \eqref{eqn_pdf_f1}). 
		The center diagram is indicated by red (upper) and green (lower) points.
		Scale bars at the right of each plot indicate the range of probability density in each shaded region.} \label{PHD_vs_1F}
	\end{figure}

	Taking $Z = (\xi_1,\xi_2) = ((b_1, d_2), (b_2,d_2))$, we arrive at a more complex expression for the kernel density when considering 2 input features. 
	From Eq. \eqref{eqn_construction}, we obtain:
	\begin{equation} \label{eqn_pdf_f2}
		\begin{split}
		K_\sigma((\xi_1,\xi_2), \mathscr{D}) &= \nu(0)q^{(1)}q^{(2)}p^{(1)}(b_1,d_1)p^{(2)}(b_2,d_2) \\
		&\hspace{5mm} + \nu(1)\LB (1-q^{(2)})q^{(1)}p^{(1)}(b_1,d_1) + (1-q^{(1)})q^{(2)}p^{(2)}(b_1,d_1) \RB p^\ell(b_2,d_2) \\
		&\hspace{5mm} + \nu(2) (1-q^{(1)})(1-q^{(2)}) p^\ell(b_1,d_1)p^\ell(b_2,d_2) \\
		&= 4.5 \times 10^{-2} e^{-2 \LP (b_1-2)^2 + (d_1-4)^2 \RP}e^{-2 \LP (b_1-1)^2 + (d_1-3)^2 \RP} \\
		&\hspace{5mm} + 2.11 \times 10^{-4} \LB e^{-2 \LP (b_1-2)^2 + (d_1-4)^2 \RP} + e^{-2 \LP (b_1-1)^2 + (d_1-3)^2 \RP} \RB p^\ell(b_2,d_2) \\
		&\hspace{5mm} + 7.39 \times 10^{-7} p^\ell(b_1,d_1) p^\ell(b_2,d_2).
		\end{split}
	\end{equation}
	Notice that this local kernel also decomposes into terms which describe presence of upper features: one term for both, one term for each of the two upper features, and the last term has no upper features.
	Contour plots of slices of this local kernel are shown in Fig. \ref{2F};
	a general description of slicing is given in Rmk. \ref{rmk_slices}.
	
	\begin{remark} \label{rmk_slices}
		Slices are used to view local pdfs defined on a high dimensional space $W^N \subset \R^{2N}$ for $N > 1$. 
		To obtain these slices, one fixes features $(b_j,d_j) = (b_j',d_j')$ for $j = 2,...,N$, and views the density on the corresponding hyperplane $W \times \LC (b_2',d_2') \RC \times ... \times \LC (b_N',d_N') \RC \subset W^N$.
		In practice, the fixed features are chosen as modes of earlier (smaller $N$) slices in order to view important parts of the distribution.
		We also sum over possible permutations in order to view a slice of the symmetric pdf, as was done for Ex. \ref{ex_double_single}.
	\end{remark}
	
	If we consider the density evaluated along slices as $K_\sigma\LP \LP(b,d),(1,3)\RP, \mathscr{D} \RP$ or $K_\sigma\LP \LP(b,d),(2,4)\RP, \mathscr{D} \RP$ (Fig. \ref{2F} (a) or (b), respectively), the restricted plot is a Gaussian centered at the other upper feature. 
	If the fixed feature is instead close to the diagonal, as in Fig. \ref{2F} (c), the density slice is close to a mixture between the two upper Gaussians $p^{(1)}$ and $p^{(2)}$.

	\begin{figure}[h!]
		\begin{center}
			\begin{multicols}{3}
				\includegraphics[scale = 0.22]{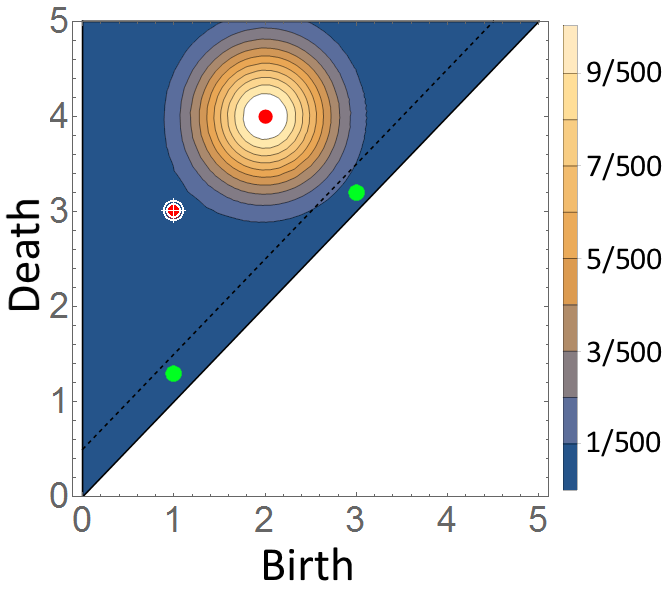} \\ (a) \\
				\includegraphics[scale = 0.22]{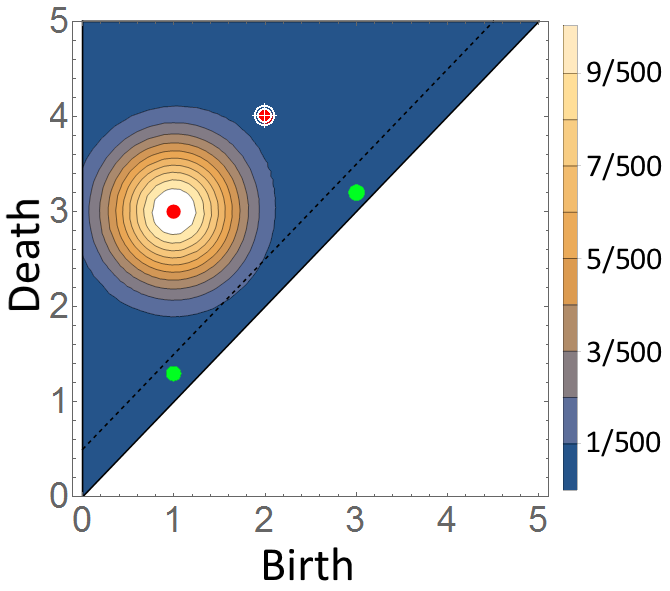} \\ (b) \\
				\includegraphics[scale = 0.22]{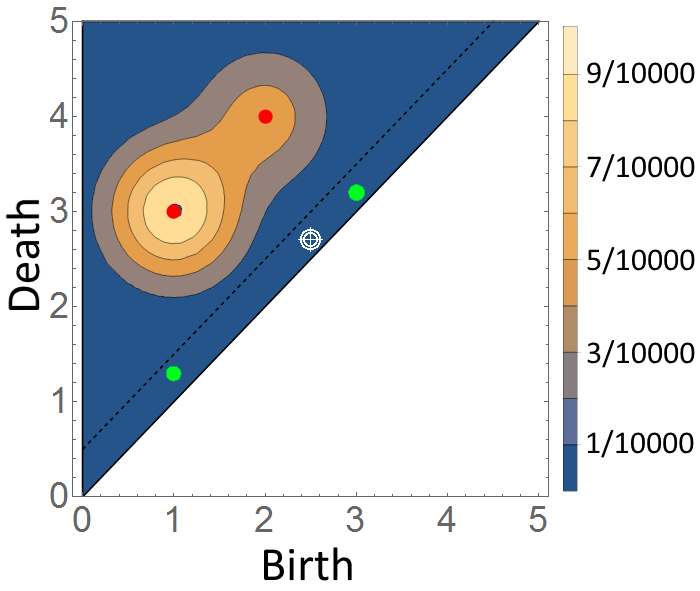} \\ (c)
			\end{multicols}
		\end{center}
		\vspace{-3mm}
		\caption{Contour maps for slices of the kernel density $K_\sigma((\xi,\xi_2'),\mathscr{D})$ with input cardinality 2.
		A single feature $\xi_2'$, indicated by white crosshairs, is fixed to restrict to a 2D subspace as follows: (a) $\xi_2' = (1,3)$ (b) $\xi_2' = (2,4)$ and (c) $\xi_2' = (2.5,2.7)$.
		The center diagram is indicated by red (upper) and green (lower) points.
		Scale bars at the right of each plot indicate the range of probability density in each shaded region.} \label{2F}
	\end{figure}

	In a similar fashion, we also express the kernel density with input cardinality $\Ln Z \Rn = 3$.
	Since there are only 2 upper features in $\mathscr{D}$, this and further expressions are not markedly more complicated than Eq. \eqref{eqn_pdf_f2}. 
	From Eq. \eqref{eqn_construction}, we obtain:
	\begin{equation} \label{eqn_pdf_f3}
		\begin{split}
		K_\sigma((\xi_1,\xi_2,\xi_3),\mathscr{D}) &= \nu(1)\LB q^{(1)}q^{(2)}p^{(1)}(b_1,d_1)p^{(2)}(b_2,d_2) \RB p^\ell(b_3,d_3) \\
		&\hspace{5mm} + \nu(2)(1-q^{(2)})q^{(1)}p^{(1)}(b_1,d_1) p^\ell(b_2,d_2) p^\ell(b_3,d_3)\\
		&\hspace{5mm} + \nu(2)(1-q^{(1)})q^{(2)}p^{(2)}(b_1,d_1) p^\ell(b_2,d_2) p^\ell(b_3,d_3) \\
		&\hspace{5mm} + \nu(3)(1-q^{(1)})(1-q^{(2)}) p^\ell(b_1,d_1) p^\ell(b_2,d_2) p^\ell(b_3,d_3). \\
		&= \, \, 9.01 \times 10^{-2} p^\ell(b_3,d_3) e^{ -2\LP (b_1-1)^2 +  (d_1-3)^2 \RP} e^{ -2\LP (b_2-2)^2 +  (d_2-4)^2 \RP } \\
		&\hspace{5mm} + 4.96 \times 10^{-4} p^\ell(b_2,d_2) p^\ell(b_3,d_3) e^{-2 \LP (b_1-2)^2 +  (d_1-4)^2 \RP} \\
		&\hspace{5mm} + 4.96 \times 10^{-4} p^\ell(b_2,d_2) p^\ell(b_3,d_3) e^{-2 \LP (b_1-1)^2 +  (d_1-3)^2 \RP} \\
		&\hspace{5mm} + 1.22 \times 10^{-6} p^\ell(b_1,d_1) p^\ell(b_2,d_2) p^\ell(b_3,d_3).
		\end{split}
	\end{equation}
	
	One may notice that Eq. \eqref{eqn_pdf_f3} has the same 4 terms as Eq. \eqref{eqn_pdf_f2}, but with another factor of $p^\ell$ in each term.
	Indeed, the local kernels for input cardinality $N = 4, 5, 6$ appear very similar as well, and with progressively more factors of $p^\ell$.
	Contour plot slices of this local kernel are shown in Fig. \ref{3F}, following Rmk. \ref{rmk_slices}.
	In this case, since the local pdf is defined in $W^3$, we must fix a pair of features in order to view a slice in $W \times \LC (b_2',d_2') \RC \times \LC (b_3',d_3') \RC$.
	In Eq. \eqref{eqn_pdf_f3}, the heaviest weighted term consists of both upper features' densities as well as the lower density $p^\ell(b_3,d_3)$.
	Indeed, Fig. \ref{3F}(a) shows the slice $K_\sigma(((b,d),(1,3),(2,4)),\mathscr{D})$, which leaves both upper features fixed, and the resulting slice is nearly proportional to the lower density $p^\ell$. 
	Fig. \ref{3F} (b) shows the slice $K_\sigma(((b,d),(1,3),(2.5,3.5)),\mathscr{D})$, which fixes one of the upper features of $\mathscr{D}$ as well as a feature of moderate persistence.
	This slice does not go through a mode of the local kernel, and so the geometry of the dataspace $W^3/\Pi_3$ makes the slice look multi-modal, depending on whether $(2.5,3.5)$ is assigned to $p^{(2)}$ or $p^\ell$.
	Other assignments have negligible mass. 
	Thus, Fig. \ref{3F} (b) resembles a mixture of these two densities.

	\begin{figure}[h!]
		\begin{center}
			\begin{multicols}{2}
				\includegraphics[scale = 0.22]{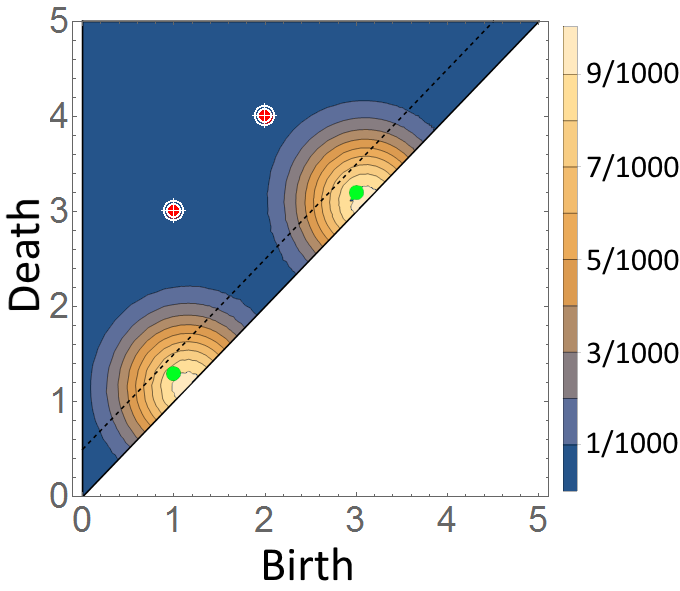} \\ (a) \\
				\includegraphics[scale = 0.22]{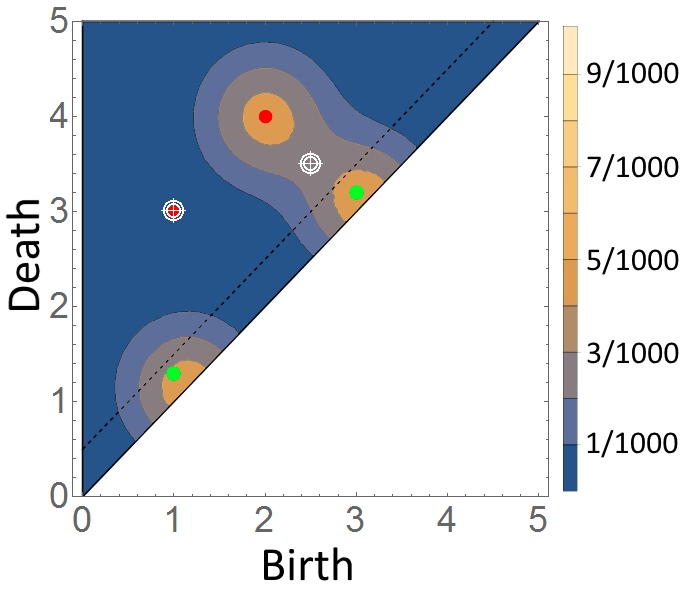} \\ (b)
			\end{multicols}
		\end{center}
		\vspace{-3mm}
		\caption{Contour maps for slices of the kernel density $K_\sigma((\xi,\xi_2',\xi_3'),\mathscr{D})$ with input cardinality 3.
		A pair of features $\xi_2'$ and $\xi_3'$, indicated by white crosshairs, are fixed to restrict to a 2D subspace as follows: (a) $(\xi_2', \xi_3') = ((1,3),(2,4))$ and (b) $(\xi_2', \xi_3') = ((1,3),(2.5,3.5))$.}
		Since the symmetric version of the density is used, the order of these features is irrelevant. 
		The center diagram is indicated by red (upper) and green (lower) points.
		Scale bars at the right of each plot indicate the range of probability density in each shaded region.
		\label{3F}
	\end{figure}

The terms $(1-q^{(k)})$ within the $\mathcal{Q}^*$ expression (see Eq. \eqref{eqn_Qstar}) are very small and appear in terms for which the corresponding upper feature is unassigned. 
These terms are so small because both upper features have very long persistence in this example (four times the bandwidth), and so the terms in Eqs. \eqref{eqn_pdf_f1}, \eqref{eqn_pdf_f2}, and \eqref{eqn_pdf_f3} which do not include one or both upper Guassians $p^{(1)}$ and $p^{(2)}$ have progressively smaller contribution to the overall local kernel. 
Consequently, the kernel places much higher probability density near input diagrams with features nearby each upper feature in the center diagram.
This behavior is seen in Fig. \ref{PHD_vs_1F}, \ref{2F}, \ref{3F}, and their respective analyses, and is directly correlated to the ratio of persistence to bandwidth for each feature.

\end{example}

\begin{example} \label{ex2} \rm
Here we consider the random persistence diagram generated from a specific random dataset in $\R^2$. 
Our goal in this example is to build and demonstrate convergence of the kernel density estimate for the pdf of the associated random persistence diagram.
Specifically, we generate sample datasets which each consist of 10 points sampled uniformly from the unit circle with additive Gaussian noise, $N((0,0),\LP\frac{1}{50}\RP^2I_2)$.
This toy dataset is prototypical for signal analysis (corresponding to the circular dynamics of a noisy sine curve), wherein the high dimensional point cloud is obtained through delay-embedding of the signal.
An in-depth analysis of using delay embedding alongside persistent homology is found in \citep{tda_windows}. 

These datasets each yield a $\cech$ persistence diagram as described in Section \ref{sect:TDA} for degree of homology $k=1$. 
A sample dataset and its associated $k=1$ persistence diagram are shown in Fig. \ref{ex2_sample_data}.
Since these datasets are sampled from the unit circle perturbed by relatively small noise, one expects the associated 1-homology to have a single persistent feature with $d \approx 1$ with possible brief features caused by noise. 

\begin{figure}
\begin{center}
\begin{multicols}{2}
\includegraphics[scale = 0.3]{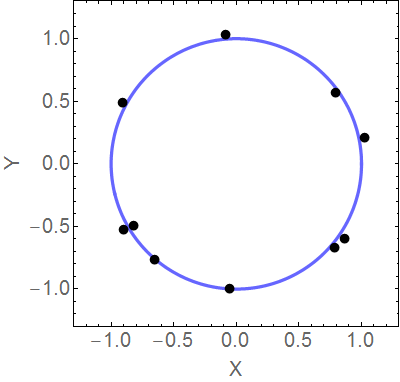} \\ (a) \\
\includegraphics[scale = 0.3]{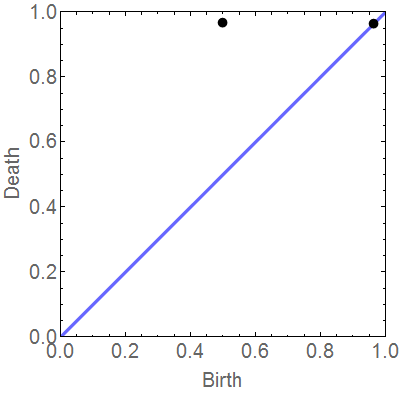} \\ (b)
\end{multicols}
\end{center}
\vspace{-3mm}
\caption{An example underlying dataset and its associated persistence diagram. 
The persistence diagrams are used as the centers for the kernel density estimate.
For this example, persistence diagrams with more than one feature are relatively rare.}
\label{ex2_sample_data}
\end{figure}

\begin{table}
	\begin{tabular}{|c | c | c | c | c |} \hline
			KDE	 & (1)	& (2)	& (3)	& (4)	\\ \hline
			n	 & 100	& 300	& 1000	& 5000 	\\ \hline
		$\sigma$ & 0.03	& 0.025 & 0.020	& 0.015 \\ \hline
	\end{tabular} \\[2mm]
	\caption{Choices of sample size $n$ (number of persistence diagrams) and bandwidth $\sigma$ for each kernel density estimate $\hat f_{n,\sigma}(Z)$ shown in Fig. \ref{fig_KDE_simp}.}
	\label{table_Nsigma_simp}
\end{table}

We consider several KDEs as we simultaneously increase the number of persistence diagrams ($n$) and narrow the bandwidth ($\sigma$) as shown in Table \ref{table_Nsigma_simp}).
The bandwidth was chosen to scale according to Silverman's rule of thumb \citep{silverman} (see Rmk. \ref{rmk_bandwidth}).

Since the KDEs $\hat f_{n, \sigma}(Z)$ are defined on $\bigcup_N W^N$ for several input cardinalities $N$, we present them in multiple slices by fixing a cardinality and then fixing all but one input feature as described in Rmk \ref{rmk_slices}.
For example, $g(\xi) = \hat f_{n, \sigma}(\xi, \xi_2',...,\xi_N')$ for fixed $\xi_j'$ ($j = 2,..., N$) is a function on $W$ and represents a slice of the local KDE on $W^N$. 
The progression of KDE slices can be seen in Fig. \ref{fig_KDE_simp}, wherein the same slices (i.e., the same features are fixed) are viewed for each choice of $(n,\sigma)$. 
These plots demonstrate in practice the convergence of the kernel density estimator shown in Theorem 1.
Because the sample points for the underlying dataset lie so close to the unit circle, one expects the topological feature to die near scale $d = 1$, as is reflected in the KDEs shown in Fig. \ref{fig_KDE_simp} (left);
however, the distribution of points along the circle allows its birth scale to vary quite a lot. 
Additional features with brief persistence are concentrated very close to the diagonal due to small noise. 
These features tend to be either spurious holes near the edge (smaller $b$ and $d$) or a short split of the main topological loop in two (larger $b$ and $d$); this behavior is reflected in the two peaks for slices of the KDEs shown in Fig. \ref{fig_KDE_simp} (right). 
Indeed, the persistence diagram shown in Fig. \ref{ex2_sample_data} is typical for this example.
Overall, by scanning from top to bottom, Fig. \ref{fig_KDE_simp} demonstrates the convergence of the KDEs as $n$ increases and $\sigma$ decreases.
The location and mass of each mode is as expected from underlying data sampled from the unit circle.
Moreover, very small spread in the limiting density arises from the small noise in the underlying data.
The shape and spread of each mode converges, and the densities for $n = 1000$ and $n = 5000$ are nearly the same.

\begin{figure}
	\begin{multicols}{2}
		\raisebox{29mm}{{\Large(1)}}\includegraphics[scale = 0.37]{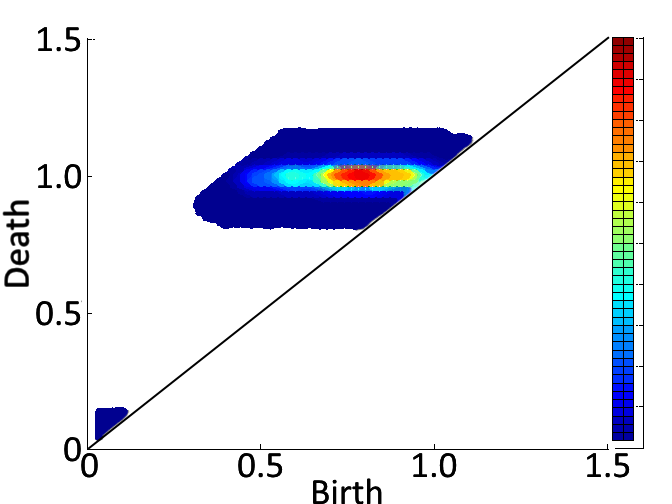}     \\
		\raisebox{29mm}{{\Large(2)}}\includegraphics[scale = 0.37]{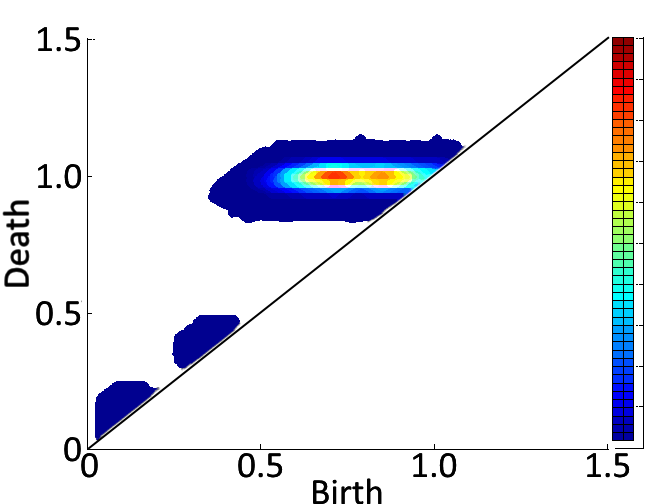}    \\
		\raisebox{29mm}{{\Large(3)}}\includegraphics[scale = 0.37]{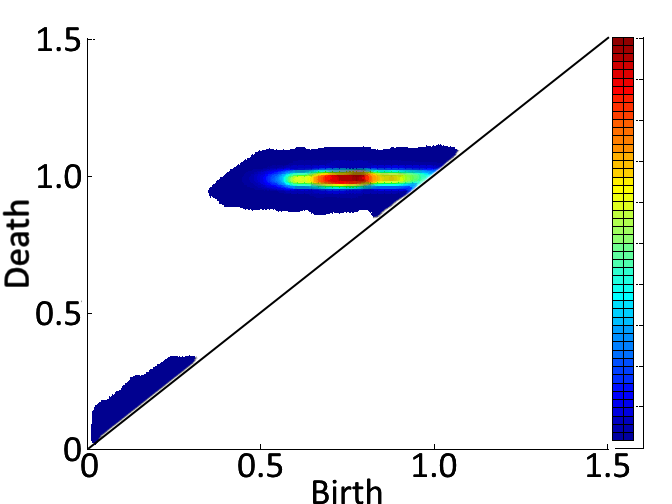}    \\
		\raisebox{29mm}{{\Large(4)}}\includegraphics[scale = 0.37]{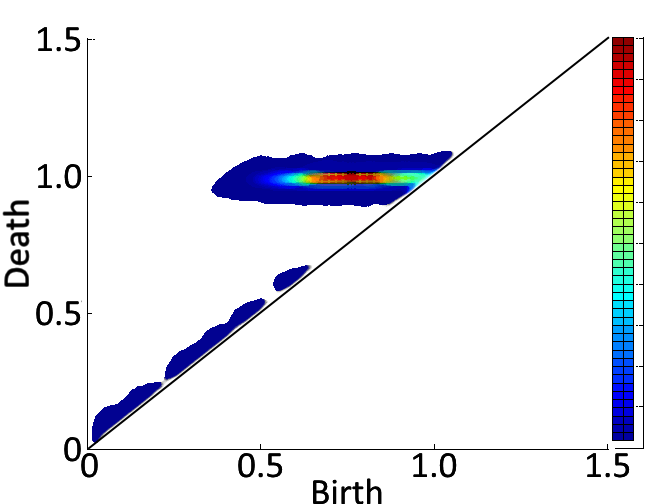}   \\
		\includegraphics[scale = 0.37]{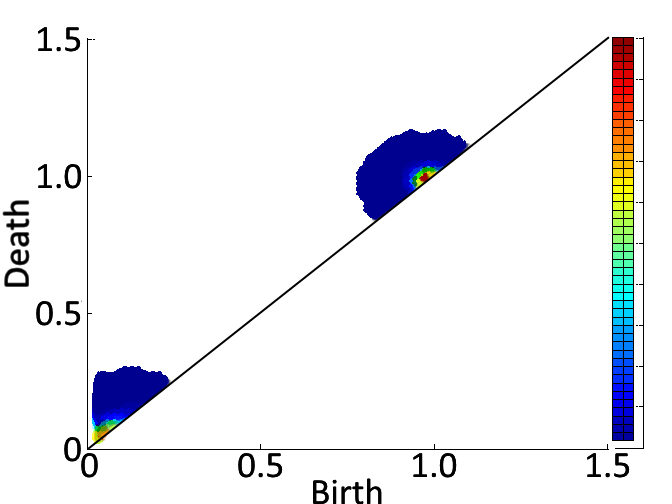}    \\
		\includegraphics[scale = 0.37]{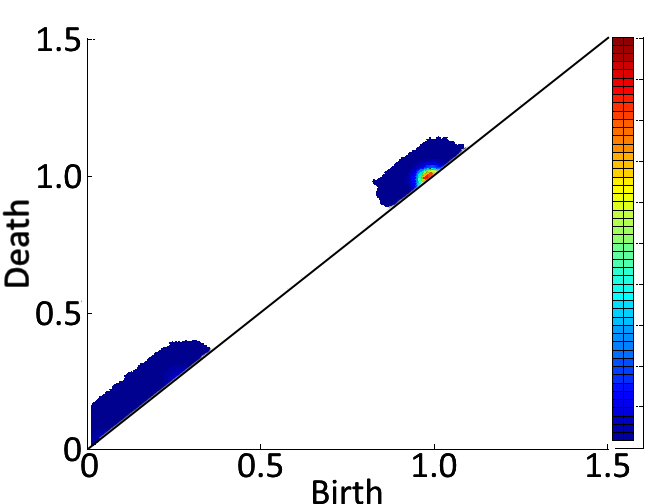}   \\
		\includegraphics[scale = 0.37]{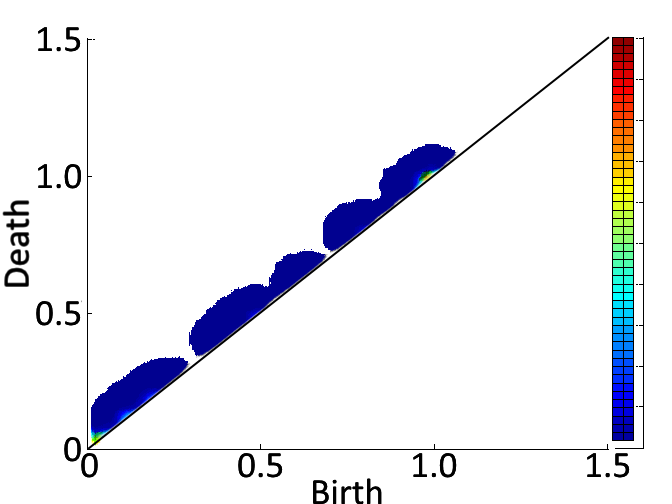}   \\
		\includegraphics[scale = 0.37]{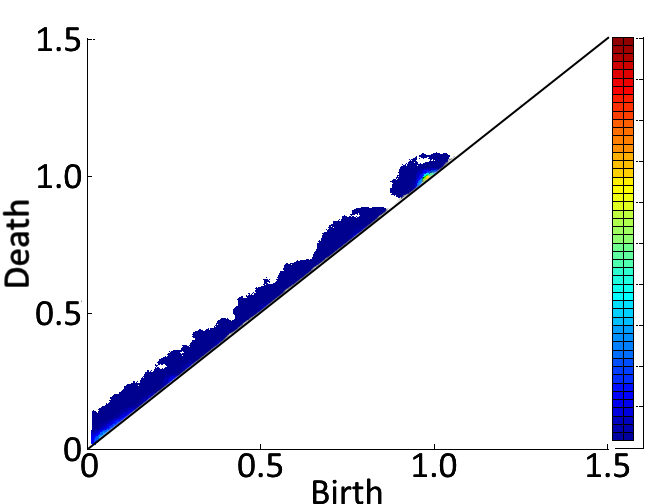} \\
	\end{multicols}
	\caption{Plots of persistence diagram KDEs for Ex. \ref{ex2}.
		Each plot is presented as a heat map where color indicates the probability density.
		White regions above the diagonal indicate portions of very low probability density. 
		Each column is a particular slice, while each row is a particular global KDE with $n$ and $\sigma$ as indicated in Table \ref{table_Nsigma_simp}.
		The left column are the local KDEs $\hat f_{n,\sigma}((b,d))$ evaluated at a diagram with only one feature.
		The mode of the converged density is approximately $(b_2',d_2') = (0.77,0.98)$.
		The right column are the local KDEs $\hat f_{n,\sigma}((b,d), (0.77,0.98))$ evaluated at a diagram with two features and one feature fixed.
		These slices have two modes which are very close to the diagonal at $(0,0)$ and $(1,1)$.
		Overall, this figure demonstrates KDE convergence.}
	\label{fig_KDE_simp}
\end{figure}

Two more examples of persistence diagram KDEs at increasing $n$ and decreasing $\sigma$ are given in the supplementary materials, but which involve more complex underlying data.
\end{example}

\subsection{A Measure of Dispersion} \label{subsect:KDE_MoD}
Theorem \ref{thm_KDE} has established the convergence of a kernel density estimator.
Along with density function estimation, one would like to verify the convergence of properties such as spread. 
In the absence of vector space structure on the space of persistence diagrams, we turn to the bottleneck metric (Defn. \ref{defn_bottleneck}) to define a notion of spread.

Specifically, we measure dispersion with respect to a distribution of persistence diagrams through its mean absolute deviation in this metric.
\begin{defn} \label{defn_MAD}
The mean absolute bottleneck deviation (MAD) from origin diagram $\mathscr{D}$ with respect to a global pdf $f$ is given by
\begin{equation} \label{eqn_bottle_moment}
\textrm{MAD}_f(\mathscr{D}) = \int_{\W} W_\infty(\mathscr{D},Z) f(Z) \delta Z
\end{equation}
\end{defn}

The following proposition aids in proving convergence of MAD kernel estimates. 
Proofs for this section are delegated to the supplementary materials.

\begin{prop} \label{prop_large_dev}
Consider $D$ distributed according to the kernel density $K_\sigma(\cdot,\mathscr{D})$ with center diagram $\mathscr{D}$ and bandwidth $\sigma$. 
Fix $\delta \geq 1$. Then,
\begin{equation} \label{eqn_large_dev}
\P\LB W_\infty(D,\mathscr{D}) < \delta \sigma\RB \geq \LP \int_{B(\bm 0,\delta)} \frac{1}{2 \pi} e^{-(x^2+y^2)/2} \, dx \, dy \RP^M
\end{equation}
where $M$ is the maximal cardinality of $D$ (a multiple of $\Ln \mathscr{D} \Rn$).
Here $B(x,r)$ refers to a ball with respect to the infinity metric (as is used in bottleneck distance). 
\end{prop}

Next, we relax assumption $(A2)$ by considering the entire multi-wedge $\W$ and tighten the decay control from assumption $(A3)$. 
Formally,
\begin{align*}
(A&2)^* \,\, \textrm{The local density } f_N:\W^N \goto \R \textrm{ is bounded for each } N \in \LC 1,...,M \RC\!. \\
(A&3)^* \,\, \textrm{There exists } C >0 \textrm{ so that } f(\xi_1,...,\xi_N) \leq C \LN (\xi_1,...,\xi_N) \RN^{-2N-2} \textrm{ for } N \in \LC 1,...,M \RC.
\end{align*}

These assumptions (and $(A1)$) are required for the subsequent lemma, which ensures that the mean absolute bottleneck deviation (MAD) is finite. 

\begin{lemma} \label{lemma_bottleneck_control}
Consider a random persistence diagram $D$ distributed according to a global pdf $f$ satisfying assumptions $(A1)$, $(A2)^*$, and $(A3)^*$.
Then $D$ has finite MAD for any choice of origin diagram $\mathscr{D}$. 
\end{lemma}

Similar to assumption $(A3)$ (given prior to Thm. \ref{thm_KDE}), $(A3)^*$ holds for a random persistence diagram associated with underlying data sampled from a compact set perturbed by Gaussian noise.
One may also replace Lemma \ref{lemma_bottleneck_control} and its assumptions by directly assuming that the maximal persistence moment is bounded; with this, the results of the lemma follow immediately from Eq. \eqref{eqn_minkowski} in the supplementary.
This direct assumption is weaker (implied by $(A1)$, $(A2)^*$, and $(A3)^*$), but may be difficult to show directly in practice. 

\begin{theorem} \label{thm_moment}
Consider a distribution of persistence diagrams with bounded global pdf, $f$, satisfying assumptions $(A1)$, $(A2)^*$, and $(A3)^*$.
Let $\hat f(Z) = \frac{1}{n} \sum_{i=1}^n K_\sigma(Z,\mathscr{D}_i)$ be a kernel density estimate with centers $\mathscr{D}_i$ sampled i.i.d. according to global pdf $f$ and bandwidth $\sigma = O(n^{-\alpha})$ chosen with $0 < \alpha < \alpha_{2M}$. 
Then, the mean absolute bottleneck deviation estimate converges; in other words,
\begin{equation} \label{eqn_moment}
\int_{\W} W_\infty(\mathscr{D}_0,Z) \hat{f}(Z) \delta Z \goto \int_{\W} W_\infty(\mathscr{D}_0,Z) f(Z) \delta Z
\end{equation}
as $n \goto \infty$ for any origin diagram $\mathscr{D}_0$. 
\end{theorem}

\section{Discussion and Conclusions} \label{sect:discussion}
A nonparametric approach to approximating density functions of finite random persistence diagrams has been presented.
This includes the introduction of a kernel density function, as well as proof that the kernel density itself and its mean absolute deviation converge to those of the target distribution. 
Future work will investigate the convergence of powers of the absolute deviation (e.g., bottleneck variance) and deviations involving the Wasserstein metric (an $L^p$ generalization of bottleneck metric, see \citep{CompyTopo}).
Our framework is presented through the lens of geometric simplicial complexes, and in particular $\cech$ complexes.
The resulting persistence diagrams are based on underlying datasets in a metric space.
In general, one may define persistent homology for a function $f$ defined on a topological space \citep{CompyTopo}, and therefore random functions may also give rise to random persistence diagrams, see \citep{tda_fields} for an example.
A similar kernel density estimate approach can be formulated in this case, but perhaps different assumptions may be needed on the target pdf.

Our approach is fully data-driven, a necessary step since distributions of persistence diagrams were previously poorly understood.
The assumptions $(A1)$-$(A3)$, $(A2)^*$, and $(A3)^*$ are typical for kernel density estimators \citep{KDE_book}. 
Similar assumptions on the underlying data are inherited by the random persistence diagram, because variation in \v Cech persistent homology is controlled by interpoint distances.
In particular, probability density decay follows the same trends as noise in the underlying data; 
this is seen in Fig. \ref{fig_KDE_simp} (a) for Gaussian noise.
Thus, the kernel density estimates defined here can be reliably used for data analysis, adding a detailed tool to the methods used in topological data analysis. 
In particular, this is the first result yielding probability density functions which directly analyze the full distribution information of a random persistence diagram.
For applications in machine learning such as classification, the kernel density estimates carry information for generating more sophisticated features than previously available;
 e.g., the value of the global pdf at a specific input or list of inputs or the integral of the global pdf over a specified region. 
Access to a pdf also provides a tool with which one can check for classification robustness in terms of likelihood or Bayes factors, providing a measure of the confidence in a particular outcome. 

Lending credence to applicability in data analysis, an example of kernel density estimation is presented in Subsection \ref{subsect:Examples}. 
In this example, underlying datasets are generated to lie on the unit circle with additive noise, a prototypical example for topological data analysis. 
Our analysis yields detailed information about the distribution of diagrams, even though only two 2-dimensional slices of the kernel density estimate are shown. 
This example demonstrates the convergence of the kernel density estimator in practice for large enough sample size (number of persistence diagrams).
This example along with the supplementary examples also demonstrate the detailed information contained in a persistence diagram KDE. 

In the context of Fig. \ref{heat}, it is clear that sampling from the kernel density is straightforward, and in fact computation time scales linearly in the number of features in the center diagram $\mathscr{D}$. 
In contrast, precise evaluation of the kernel global pdf at a diagram requires the more thorough computations shown in Eq. \eqref{eqn_construction}. 
This evaluation is made tractable due to the separation of the center diagram into upper and lower portions: $\mathscr{D} = \mathscr{D}^u \cup \mathscr{D}^\ell$ as described in Eq. \eqref{eqn_split}.
In practice, diagrams should split so that $\Ln \mathscr{D}^u \Rn$ is small while $\Ln \mathscr{D}^\ell \Rn$ is large.
Evaluation of individual feature pdfs on the multi-wedge $\W$ only scales quadratically on the cardinality $\Ln \mathscr{D} \Rn$ and higher degree calculations are required only for combinatorics on the large persistence features in the upper diagram $\mathscr{D}^u$. 
Consequently, these calculations are tractable so long as $\mathscr{D}^u$ does not grow too much in cardinality, while an increased cardinality for $\mathscr{D}^\ell$ has a lesser effect on computation time. 

The kernel density presented here treats the small persistent features in $D^\ell$ as a single group.
Since convergence (Thm. \ref{thm_KDE}) requires very little structure in the lower random diagram, it may be helpful in practice to cluster the lower portion of the center diagram, followed by defining a random diagram centered at each cluster.
This approach somewhat complicates the expression and evaluation of the kernel density, but does not complicate sampling from the kernel density.
The goal of this approach is to more carefully capture the geometric features of the underlying random dataset, since such geometric features often correspond to briefly persistent homological features. 
For example, geometric features are of paramount importance for classifying periodic signals through their delay embeddings, wherein the large persistent feature indicates periodicity and thus is expected to appear in every class.


\pagebreak\appendix
\section{Proof of Theorem \ref{thm_KDE}} \label{sect:proof_of_thm_KDE} 
The proof presented in this section describes the case for degree of homology $k>0$. 
The case for $k=0$ is obtained by a slight modification and the full result follows by an application of Corollary \ref{cor_general_kernel}.

Recall Thm. \ref{thm_construction}, which defines the pertinent kernel density $K_\sigma(Z,\mathscr{D})$ evaluated at $Z= (\xi_1,...,\xi_N)$ according to center diagram $\mathscr{D}$ and bandwidth $\sigma$ by
$$ K_\sigma(Z, \mathscr{D}) = \sum_{j=0}^{N_u} \nu(N-j) \sum_{\gamma \in I(j,N_u)} \mathcal{Q}(\gamma) \prod_{k=1}^j  p^{(\gamma(k))}(\xi_{k}) \prod_{k=j+1}^N  p^\ell(\xi_{k}) $$
where $\mathcal{Q}(\gamma)$ is given by Eq. \eqref{eqn_QQ}, each $p^{(j)}$ refers to the modified Gaussian pdf shown in Eq. \eqref{eqn_mod_normal} for its matching feature $\xi_j$ in $D^u$, $N_u = \Ln \mathscr{D}^u \Rn$, and $p^\ell$ is given by Eq. \eqref{eqn_lower_density}.
Also recall that $\mathscr{D}$ is split into $\mathscr{D}^\ell$ and $\mathscr{D}^u$ according to Eq. \eqref{eqn_split}, $D^\ell$ is defined with global pdf from Eq. \eqref{eqn_lower_dist}, and $D^u$ is defined with global pdf from Eq. \eqref{eqn_combination}.

Throughout the proof we use $\xi_i$ to denote input features and $Z = \LC \xi_1,...,\xi_N \RC$ or $Z = \LP \xi_1,...,\xi_N \RP$ to denote an input persistence diagram as a set or vector of features. 
Several preliminary lemmas are presented before the main body of the proof. 
We begin with a critical lemma which controls the number of features sampled in the band diagonal $\Delta_{\alpha}^{\beta} = \LC (b,d) \in W : \alpha < d-b < \beta \RC$.

\begin{lemma} \label{band_diag}
	Consider a random persistence diagram $D$ distributed according to $f$ satisfying assumptions $(A1)$-$(A3)$.
	Then there exists $C > 0$ so that $\E^f\LP\Ln \Delta_0^{\sigma} \cap D\Rn\RP \leq C \sigma$. 
\end{lemma}
\begin{proof}
	Consider a region $A \subset W$ and a counting function $\kappa_A(Z) = \Ln Z \cap A \Rn$ such that 
	$\kappa_A(\LC \xi_1,...,\xi_N \RC) = \sum_{i=1}^N \one_A(\xi_i)$. 
	It is clear that this set function is well defined and measurable if $A$ is measurable. 
	Using set integration (Defn. \ref{defn_set_integral}),
	\begin{equation} \label{eqn_counting}
	\E(\Ln \Delta_0^{\sigma} \cap D\Rn) = \int_{W} \kappa_{\Delta_0^{\sigma}}(Z) f(Z) \delta Z = \sum_{N=0}^M \frac{N}{N!} \int_W \one_{\Delta_0^{\sigma}}(\xi_1) \LB \int f(\xi_1,...\xi_N) d\xi_2 ... d\xi_N \RB d\xi_1
	\end{equation}
	The expressions in Eq. \eqref{eqn_counting} can be phrased in terms of the probability hypothesis density from Eq. \eqref{eqn_PHD}, and for any choice of $L > 0$ are bounded by
	\begin{align*}
	\int_{\Delta_0^{\sigma}} F_D(\xi)  d\xi &\leq \int_{0}^{L} \int_{y-\sigma}^{y} F_D(x,y) \, dx\,dy 
	+ \int_L^\infty \int_{y-\sigma}^{y} C_3 y^{-2} \, dx\, dy \\[2mm]
	&\leq L C_2 \sigma + 3C_3 \sigma/L = (LC_2 + C_3/L) \sigma
	\end{align*}
	where assumptions (A2) and (A3) respectively yield the bounds $C_2$ and $C_3 y^{-2}$ on the probability hypothesis density, $F_D$. 
\end{proof}

Lemma \ref{band_diag} yields control over the counting measure $\nu_i$ defined in Defn. \ref{defn_lower_cluster} and the coefficients $\mathcal{Q}^*_i(\cdot)$ of Eq. \eqref{eqn_Qstar} which respectively determine the distribution of lower and upper cardinalities for a persistence diagram sampled according to the kernel density $K_\sigma(Z,\mathscr{D}_i)$.

\begin{cor} \label{nu}
	Consider a random persistence diagram $D$ distributed according to $f$ satisfying assumptions $(A1)$-$(A3)$.
	Take $\nu$ to be the lower cardinality probability mass function for the kernel density $K_\sigma(Z,D)$ shown in Eq. \eqref{eqn_construction}. 
	Then, there exists $C > 0$ so that $\E^f \nu(j_0) \leq C\sigma$ whenever $j_0 \neq 0$. 
\end{cor}
\begin{proof}
	Since $D$ is random with respect to $f$, $\nu$ is random with respect to $f$ as well.
	Recall that $\nu$ is defined so that $\E^{\nu}(\bm a) = \Ln D^\ell \Rn$ for $\bm a$ distributed according to $\nu$ and thus $\E^f[\E^\nu(\bm a)] \leq C \sigma$ for some $C>0$ by Lemma \ref{band_diag}. 
	Subsequently, the value $\E^f \nu(j_0)$ is controlled by this double expectation so long as $j_0 \neq 0$. 
	Indeed, 
	$$\E(\bm a) = \sum_{j = 0}^\infty j \nu(j) = \sum_{j=1}^\infty j \nu(j) \geq \sum_{j=1}^\infty \nu(j) \geq  \nu(j_0)$$
	for any $j_0 > 0$ and $\nu_i(j_0) = 0$ for $j_0 < 0$ since it represents a cardinality distribution.
\end{proof}

In the following lemma, the result of Lemma \ref{band_diag} is used to control the expressions $\mathcal{Q}(\gamma)$ or $\mathcal{Q}^*(\gamma)$, of Eq. \eqref{eqn_QQ} and Eq. \eqref{eqn_Qstar} respectively, in the kernel density estimate.

\begin{lemma} \label{lemma_QQ}
	Consider a random persistence diagram $D$ distributed according to $f$ satisfying assumptions $(A1)$-$(A3)$.
	Take $\mathcal{Q}$ of Eq. \eqref{eqn_QQ} and $\mathcal{Q}^*$ of Eq. \eqref{eqn_Qstar} to be the upper singleton probabilities for the kernel density $K_\sigma(Z,D)$ shown in Eq. \eqref{eqn_construction}. 
	Then, there exists $C > 0$ so that $\E^f \LB \mathcal{Q}(\gamma) \RB \leq \E^f \LB \mathcal{Q}^*(\gamma) \RB \leq C \sigma$ for any $\gamma \in I(j,N)$ with $j < N$.
\end{lemma}

\begin{proof}
	Since every $q^{(k)} \in (0,1)$, we have that $\mathcal{Q}(\gamma) \leq \mathcal{Q}^*(\gamma)$;
	and furthermore, since $\gamma \in I(j,N)$ are not onto when $j < N$, each product $\mathcal{Q}^*$ is bounded by one of the terms of the $(1-q_i^{(k)})$ type. 
	By construction, these terms depend monotonically upon a feature's persistence, and the maximum (over all indices $j<N$ and functions $\gamma$) is tied to the least persistent feature of $\mathscr{D}_i^u$.
	
	For a feature $(b,d)$ of persistence $p=d-b$, we define $q(p) := \int_{-p/(\sqrt{2}\sigma)}^\infty \frac{1}{\sqrt{2 \pi}} e^{-x^2/2} dx$ in concordance with Eq. \eqref{eqn_nonempty};
	or in terms of the error function $\Phi$, $q(p) = \frac{1}{2}\LP1+\Phi\LP \frac{p}{2\sigma}\RP \RP$. 
	Define the minimal persistence as $p_{\min}(Z) = \sup \LC p : \Ln \Delta_0^p \cap Z \Rn = \emptyset \RC$ which satisfies $p_{\min}(Z) \geq p$ if and only if $\Ln \Delta_0^p \cap Z \Rn = \emptyset$.  
	In turn, we may bound $\mathcal{Q}^*(\gamma) \leq (1-q(p_{\min}(D))$ independently of $\gamma$. 
	By Lemma \ref{band_diag}, there is $C>0$ such that $\P^f\LB\Ln \Delta_0^{\sigma} \cap D\Rn \neq \emptyset \RB \leq \E^f\LB\Ln \Delta_0^{\sigma} \cap D\Rn \RB \leq C\sigma$, which controls the distribution of the minimal persistence.
	
	In particular, $q'(p) = \frac{1}{2\sigma\sqrt{\pi}} e^{-p^2/4\sigma^2}$ by the fundamental theorem of calculus. 
	The control of Lemma \ref{band_diag} and the fact that $p_{\min}(Z) \geq 0$ also allows us to utilize integration via the probability of sublevel sets.
	Take $g(p) = 1 - q(p)$ so that $\lim_{p\goto \infty} g(p) = 0$. 
	Specifically, since  $\mathcal{Q}^*(\gamma) \leq (1-q(p_{\min}(D))$, and using the fundamental theorem of calculus then Fubini's theorem, we have:
	\begin{equation} \label{eqn_Qlayers}
	\begin{split}
		\E^f[\mathcal{Q}^*(\gamma)] \leq \int_{\W} g(p_{\min}(Z))f(Z) \delta Z 
		&= \int_{\W} \LP \int_{\infty}^{p_{\min}(Z)} g'(p) dp \RP f(Z) \delta Z \\
		= \int_\infty^0 \LP \int_{\LC Z : p_{\min}(Z) < p \RC} f(Z) \delta Z \RP g'(p) dp 
		&= \int_0^{\infty} \LP \P^f [p_{\min} < p] \RP q'(p) dp .
	\end{split}
	\end{equation}
	
	We now further bound the expectation in Eq. \eqref{eqn_Qlayers}.
	Replacing terms with their definitions and using the bound control from Lemma \ref{band_diag} we obtain:
	
	\begin{align*}
	\E^f \LB \mathcal{Q}^*(\gamma) \RB  &\leq \int_0^\infty \P^f(\Delta_0^p \cap D \neq \emptyset) \frac{1}{2\sigma\sqrt{\pi}} e^{-p^2/4\sigma^2} dp \\
	&\leq \frac{C}{2\sigma\sqrt{\pi}} \int_0^\infty p e^{-(p/2\sigma)^2} dp 
	= \frac{C}{2\sigma\sqrt{\pi}} \LB -2\sigma^2e^{-p^2/4\sigma^2} \RB_{p=0}^\infty = \frac{C}{\sqrt{\pi}} \sigma.
	\end{align*}
\end{proof}

\textbf{Proof of Theorem \ref{thm_KDE}.} \\
For convenience, we denote the upper cardinalities by $N_i = \Ln \mathscr{D}_i^u \Rn$ and total cardinalities by $M_i = \Ln \mathscr{D}_i \Rn$ for the sample persistence diagrams. 
Denote the set of strictly increasing functions from $\LC 1, ..., j\RC$ into $\LC 1, ..., N_i \RC$ by $I(j,N_i)$. 
Here we use `$\textrm{id}$' to denote the identity map, where $I(N_i,N_i) = \LC \textrm{id} \RC$.
The proof is organized by splitting the kernel densities into several pieces and then controlling each piece separately. 

First, we separate the kernel $K_\sigma(Z,\mathscr{D}_i)$, defined in Eq. \eqref{eqn_construction}, into three portions, $A_i$, $B_i$, and $C_i$, according to the upper cardinality $j$:
\begin{align}\label{ABC}
\begin{split}
K_\sigma(Z,\mathscr{D}_i) &= \sum_{j=0}^{N_i} \nu_i(N-j)  \sum_{\gamma \in I(j,N_i)} \mathcal{Q}_i(\gamma) \prod_{k=1}^j  p_i^{(\gamma(k))}(\xi_{k}) \prod_{k=j+1}^N  p_i^\ell(\xi_{k}) \\
&=  \nu_i(N-N_i) \mathcal{Q}_i(\id) \prod_{k=1}^{N_i}  p_i^{(k)}(\xi_{k}) \prod_{k=N_i+1}^N  p_i^\ell(\xi_{k}) \\
&+ \sum_{j=0, j \neq N}^{N_i-1} \nu_i(N-j) \sum_{\gamma \in I(j,N_i)} \mathcal{Q}_i(\gamma) \prod_{k=1}^j  p_i^{(\gamma(k))}(\xi_{k}) \prod_{k=j+1}^N  p_i^\ell(\xi_{k}) \\
&+ \one_{\LC n \in \N : n < N_i \RC}(N) \nu_i(0) \sum_{\gamma \in I(N,N_i)} \mathcal{Q}_i(\gamma) \prod_{k=1}^N  p_i^{(\gamma(k))}(\xi_{k}) \\
&= A_i + B_i + C_i,
\end{split}
\end{align}
where $A_i$ follows from $j = N_i$, $C_i$ follows from $j = N$ ($C_i = 0$ if $N_i \leq N$), and $B_i$ consists of all remaining terms. 

The terms $B_i$ in Eq. \eqref{ABC} are controlled by the lower product $\LB \prod_{k=j+1}^N p_i^\ell(\xi_{k})\RB$.
Since $(1-q_i^{(j)}) \leq 1$ and $\nu_i(N-j) \leq 1$ for any choice of $\gamma$ and $j$, we have that $B_i$ is bounded above by
\begin{equation} \label{eqn_Bbound}
\sum_{j=0,j \neq N}^{N_i-1} \sum_{\gamma \in I(j,N_i)} \LB \prod_{k=1}^j q_i^{(\gamma(k))} p_i^{(\gamma(k))}(\xi_{k}) \prod_{k=j+1}^N  p_i^\ell(\xi_{k}) \RB.
\end{equation}
The bounding sum of Eq. \eqref{eqn_Bbound} consists of restricted $2N$-dimensional Gaussians, with the weights $q_i^{(j)}$ dominating the restriction rescaling in Eq. \eqref{eqn_mod_normal}. 
Fix $\pi \in \Pi_N$ and $j \in \LC 0,..., M - 1\RC \setminus \LC N \RC$. 
Without loss of generality, we treat the case when the permutation $\pi$ is the identity. 
Since our ultimate goal is to control the kernel density estimate $\hat f$, consider the portion of $\sum_{i=1}^n \frac{1}{n} B_i$ for which the cardinalities $M_i = \Ln \mathscr{D}_i \Rn$ are fixed at level $M_i = m \in \LC 0,...,M \RC$.
Now, $m = \Ln \mathscr{D}_i \Rn \geq N_i > j$, so there is some extension for every $\gamma$ within the sum, $\gamma^* \in \Pi_m$. 
Recall that this collection is random because each $\mathscr{D}_i$ is randomly distributed according to $f$, therefore we consider the expectation with respect to this randomness:
$$\E^f\LB\sum_{\LC i: M_i=m \RC} \frac{1}{\Ln \LC i: M_i=m \RC \Rn} \prod_{k=1}^{M_i} q_i^{(\gamma^*(k))}p_i^{(\gamma^*(k))} (\xi_k)\RB \goto f(\xi_1,...,\xi_m),$$
for any point $(\xi_1,...,\xi_m)$ as a $2m$-dimensional Gaussian kernel density estimate with a proper choice of $\sigma = O(n^{-\alpha})$ appropriate for $2M$ (and hence $2m$) dimensions \citep{KDE_book}.
Integrating both sides against the extra coordinates, Assumptions (A2) and (A3) along with the dominated convergence theorem yield 
\begin{equation} \label{B1}
\E^f\LB\sum_{\LC i: M_i=m \RC} \frac{1}{\Ln \LC i: M_i=m \RC \Rn} \prod_{k=1}^{j} q_i^{(\gamma(k))}p_i^{(\gamma(k))} (\xi_k)\RB \goto \int_{W^{m-j}}f(\xi_1,...,\xi_m) d\xi_{j+1} ... d\xi_m,
\end{equation}
which is again bounded via (A2) and (A3). 
Of course, $\Ln \LC i: M_i=m \RC \Rn \leq n$, so taking Eq. \eqref{B1} into account for every $m$ bounds the averaging sum of the upper product: $\frac{1}{n} \sum_{i=1}^n \prod_{k=1}^j  q_i^{(\gamma(k))} p_i^{(\gamma(k))}(\xi_{k})$. \\

Relying on Eq. \eqref{eqn_Bbound}, we must also consider the lower product $\prod_{k=j+1}^N p_i^\ell(\xi_{k})$. 
Since the points $\xi_i$ are fixed, we focus on their minimal persistence $p_{\min} = \min_i(d_i-b_i)$.
Thus, 
$$p_i^\ell(\xi_i) \leq \frac{1}{2\pi\sigma^2} e^{-(b-d)^2/4\sigma^2} \leq \frac{1}{2\pi\sigma^2} e^{-p_{\min}^2/4\sigma^2},$$
and subsequently,
\begin{equation} \label{B2}
\LB \prod_{k=j+1}^N p_i^\ell(\xi_{k}) \RB \leq \frac{1}{(2\pi\sigma^2)^N} e^{-N p_{\min}^2/4\sigma^2} \goto 0,
\end{equation}
as $\sigma \goto 0$, uniformly on any compact subset of $W$ (or $\W$). 
Altogether, Eqs. \eqref{B1} and \eqref{B2} guarantee that the term $\sum_{i=1}^n \frac{1}{n} B_i \goto 0$ as $n \goto \infty$ in the kernel density estimation. \\

Next we focus on the terms $A_i$ in Eq. \eqref{ABC}.
We split the sum $\frac{1}{n} \sum_{i=1}^n A_i$ according to the cardinality of $\mathscr{D}_i$.
Specifically, separate $A_i$ into the cases where $M_i \neq N_i$ or $M_i = N_i$. 
First consider the associated set of indices $\LC i : M_i \neq N_i \RC$ and define the mismatch number $\textrm{MM}(n)$ to be its cardinality. 
Critical to our argument, the mismatch number is random with respect to $f$ because it is defined according to the features in $\mathscr{D}_i$. 
We obtain the following mismatched term:
\begin{equation} \label{eqn_mismatch}
\frac{1}{n} \sum_{\LC i:N_i \neq M_i \RC} A_i \leq \LP \frac{MM(n)}{n} \RP \frac{1}{MM(n)} \sum_{\LC i:N_i \neq M_i \RC} \LB \mathcal{Q}_i(\id) \prod_{k=1}^{N_i} p_i^{(k)}(\xi_{k}) \prod_{k=N_i+1}^N p_i^\ell(\xi_{k}) \RB
\end{equation}
The bounding sum in Eq. \eqref{eqn_mismatch} is split into pieces where $M_i = m$ for each $m$ between $0$ and $M$. 
Using the same strategy yielding Eq. \eqref{B1}, with $MM(n)$ in place of $n$, the sum of the upper product converges to layered integrals of $f$ for each level $m$ and each $N_i < m$ by extending $\gamma = \id$. 
Using the same approach leading to Eq. \eqref{B2}, the lower product vanishes in the limit if $N_i \neq N$, or is an empty product if $N_i = N$; in either case, this factor is bounded. 
Now, according to Lemma \ref{band_diag}, $\P^f(M_i \neq N_i) = \P^f(\mathscr{D}_i \cap \Delta_0^{\epsilon\sigma} \neq \emptyset) \leq C_5 \sigma$;
consequently, $\E^f[MM(n)/n] \goto 0$ and the mismatch terms on left hand side of Eq. \eqref{eqn_mismatch} follow. 

Now consider the indices for which $N_i = M_i$. 
In this case, since $\mathscr{D}_i^\ell$ are empty, $\nu_i = \delta_0$, and the only values which contribute to the sum are for $N_i = N$. 
The remaining portion of the kernel density estimate is given by
\begin{equation} \label{eqn_match}
\begin{split}
\frac{1}{n} \E^f \hspace{-3mm} \sum_{\LC i: N_i= M_i \RC} \hspace{-4mm} A_i 
&= \frac{1}{n} \E^f \LB \sum_{\LC i: N_i= M_i \RC} \hspace{-2mm} \LP \mathcal{Q}_i(\id) \prod_{k=1}^{N}  p_i^{(k)}(\xi_{k}) \RP \RB \\
&= \frac{1}{n} \E^f \LB \sum_{\LC i: N_i= M_i \RC} \hspace{-2mm} \LP \prod_{k=1}^{N}  q_i^{(k)} p_i^{(k)}(\xi_{k}) \RP \RB.
\end{split}
\end{equation}
As shown, the terms in Eq. \eqref{eqn_match} are restricted $2N$ dimensional Gaussians. 
It is known \citep{KDE_book} that restricted Gaussian kernel density estimates like $\LB \prod_{k=1}^{N}  q_i^{(k)} p_i^{(k)}(\xi_{k})\RB$ converge (uniformly on compactly contained sets) to the true value of the chosen draws $\mathscr{D}_i$  for a suitable choice of $\alpha$ in $\sigma = O(n^{-\alpha})$ as restricted by $N \leq M$. 
After correcting for the samples with $N_i < M_i = N$, the samples $\mathscr{D}_i$ are treated as random draws from $f(D | \Ln D \Rn = N)$.
Consequently, we may conclude that the target distribution associated with $\LB \prod_{k=1}^{N}  q_i^{(k)} p_i^{(k)}(\xi_{k})\RB$ is the rescaled $\frac{1}{f(N)}f(\xi_1,...,\xi_N)$, where $f(N) := \P^f(\Ln D \Rn = N)$. 
This rescaling for the conditional pdf $f(D | \Ln D \Rn = N)$ is necessary to reweight according to Prop. \ref{prop_belief_layers}.

Application of classical kernel density estimate results require division by the cardinality of the draw, when in context $n$ is generally larger than this cardinality.
Thus, we must again consider the cases wherein $N_i \neq M_i$.
Consequently, we find that the expectation for the ratio between the true draw cardinality and $n$ is given by $\P^f(\Ln D \Rn = N) + O(\sigma)$ according to Lemma \ref{band_diag}. 
Indeed, this ratio converges to $f(N) := \P^f(\Ln D \Rn = N)$. 
After this final correction, we have shown that $\frac{1}{n} \sum_{i=1}^n A_i$ approach the true pdf $f(\xi_1,...,\xi_N)$. 

Lastly, we need only to control the terms $C_i$ from Eq. \eqref{ABC}. 
We begin by bounding the probability mass functions $\nu_i$ by 1 and considering only terms for which the characteristic function is nonzero:
\begin{equation} \label{Cbound}
\frac{1}{n} \sum_{i=1}^n  C_i = \frac{1}{n} \sum_{\LC i: N < N\RC} \nu_i(0) \sum_{\gamma \in I(N,N_i)} \mathcal{Q}_i(\gamma) \prod_{k=1}^N  p_i^{(\gamma(k))}(\xi_{k})
\leq \frac{1}{n} \sum_{\LC i: N < N_i \RC} \sum_{\gamma \in I(N,N_i)} \mathcal{Q}_i(\gamma) \prod_{k=1}^N  p_i^{(\gamma(k))}(\xi_{k}). \\
\end{equation}

Next, we split the term $\mathcal{Q}(\gamma)$ according to Eq. \eqref{eqn_QQ} and apply Lemma \ref{lemma_QQ} to the upper bound in Eq. \eqref{Cbound} to obtain the larger upper bound

\begin{equation}
\begin{split} \label{CC} 
\frac{1}{n} \sum_{\LC i: N < N_i \RC} \sum_{\gamma \in I(N,N_i)} \mathcal{Q}^*(\gamma) \prod_{k=1}^N q_i^{(\gamma(k))} p_i^{(\gamma(k))}(\xi_{k}) 
\leq C \LB \frac{1}{n} \sum_{\LC i: N < N_i \RC} \sum_{\gamma \in I(N,N_i)} \prod_{k=1}^N q_i^{(\gamma(k))} p_i^{(\gamma(k))}(\xi_{k}) \RB \sigma.
\end{split}
\end{equation}

The expectation of the bracketed terms in Eq. \eqref{CC} converges in a fashion identical to the terms $\frac{1}{n} \sum_{i=1}^n A_i$.
Since these terms are multiplied by $\sigma$, altogether $\LB \frac{1}{n} \sum_{i=1}^n C_i \RB$ vanishes in the limit as $n \goto \infty$. 
Putting together the limits of each portion built from $K_\sigma(Z,\mathscr{D}_i) = A_i + B_i + C_i$, the theorem follows. \hfill $\blacksquare$

\section{Proofs from Section 4.3}

\subsection{Proof of Proposition \ref{prop_large_dev}}
Note that the lower bound integral is the probability for a pair $z = (x,y)$ of independent standard normal variables to lie in $B((0,0), \delta)$. 
In order to bound the bottleneck distance $W_\infty(D,\mathscr{D}) < \delta \sigma$, it is sufficient that each constituent feature does not stray too far from either its corresponding center or the diagonal (see Fig. \ref{heat} for reference).
Specifically, we follow Defn. \ref{defn_bottleneck} to build a correspondence between $D$ and $\mathscr{D}$ so that the maximal distance undercuts $\delta\sigma$, and thus the (potentially smaller) bottleneck distance is also bounded by $\delta\sigma$. 
For clarity, features in $D$ are denoted using $\zeta$ while features in $\mathscr{D}$ are denoted using $\xi$. 

Consider each feature $\xi^j \in \mathscr{D}^u = \mathscr{D} \cap \LC d - b \geq \sigma \RC$ and its associated random singleton diagram $D^j = \LC \zeta^j \RC$ or $\emptyset$ as in Defn. \ref{defn_upper_singletons}. 
Assuming the disc neighborhood $B_j = B(\xi^j,\delta\sigma)$ is contained in the wedge $W = \LC (b,d) \in \R^2 : d > b \geq 0 \RC$, 
the density of $z^j = \frac{\zeta^j - \xi^j}{\sigma}$ is a multiple ($>1$) of the density of the Gaussian random variable $z \sim N((0,0),I_2)$ in the region where $\zeta^j \in B_j$ (or equivalently $z^j \in B((0,0),\delta)$).
Thus, we obtain $\P \LB \zeta^j \in B(\xi^j,\delta\sigma) \RB \geq \P \LB\Ln z \Rn \leq \delta\RB$ for the probability that $\zeta^j$ can be mapped to $\xi^j$ in a bounding correspondence. 
If $B_j \nsubseteq W$, this probability is even higher because $\xi^j$ can be mapped to the diagonal and thus the case $D^j = \emptyset$ is included. 

Now take into account the features in $\mathscr{D}^\ell$ and the associated random features $D^\ell$ as in Defn. \ref{defn_lower_cluster}. 
Although the features in $D^\ell$ are not necessarily independent, we may assume without loss of generality the worst case, in which the maximal cardinality is drawn.
Given a fixed cardinality, the draws of $D^\ell$ are independent. 
Since any feature may be mapped to the diagonal in the bottleneck distance, a bounding correspondence can be obtained whenever the draws in $D^\ell$ and features in $\mathscr{D}^\ell$ are close enough to the diagonal (within $\delta\sigma$).
Indeed, the features in $\mathscr{D}^\ell$ are by definition distance $\sigma \leq \delta\sigma$ from the diagonal.
Restricting to $W$, the pdf for the draws of $D^\ell = \LC (b_j,d_j) \RC_{j=1}^{\Ln N_\ell \Rn}$ is given by 
$p^\ell(b,d) = \frac{1}{\pi N_\ell \sigma^2} \sum_{j=1}^{N_\ell} e^{-\LP \LP x-\frac{b_j+d_j}{2}\RP^2 + \LP y-\frac{b_j+d_j}{2}\RP^2\RP/2\sigma^2}$. 
Consider the sets $U_j = B\LP\LP \frac{b_j+d_j}{2},\frac{b_j+d_j}{2}\RP, \delta\sigma \RP $ and $\mathcal{U} = \bigcup_{j=1}^{N_\ell} U_j$.
For each lower feature $(b,d) \in D^\ell$, mapping to the diagonal yields a bounding correspondence and the associated probability is bounded below by
$\P[d-b \leq \delta\sigma] = \int_{\Delta_0^{\delta\sigma}} p^\ell(x,y) \, dx \, dy \geq \int_{W \cap \mathcal{U}} p^\ell(x,y) \, dx \, dy $ since $W \cap U \subset \Delta_0^{\delta\sigma} = \LC (b,d) \in W : d-b \leq \delta\sigma\RC$. 
Next, we restrict the lower bounding integral for each term of $p^\ell$ to its matching subset $U_j$ and change variables to attain the desired form:
\begin{align*}
\int_{W \cap \mathcal{U}} p^\ell(x,y) \, dx \, dy &\geq \sum_{j=1}^{N_\ell} \int_{U_j} \frac{1}{2\pi N_\ell \sigma^2} e^{-\LP \LP x-\frac{b_j+d_j}{2}\RP^2 + \LP y-\frac{b_j+d_j}{2}\RP^2\RP/2\sigma^2} \, dx \, dy \\
&= \int_{B((0,0),\delta)} \frac{1}{2\pi} e^{-(x^2+y^2)/2} \,dx \, dy.
\end{align*}
Overall, this argument shows that with probability at least $\P(\Ln \bm{z} \Rn \leq \delta)^M$ there is a correspondence which bounds the bottleneck distance by $\delta\sigma$ and the result follows.

\subsection{Proof of Lemma \ref{lemma_bottleneck_control}}
Choose an arbitrary persistence diagram $\mathscr{D}$.
Since bottleneck distance is defined according to the sup-norm (see Eq. \eqref{eqn_bottleneck}), the bottleneck distance to the null persistence diagram (i.e., without any features) is precisely half the maximal persistence. 
Thus, we begin by showing that the maximal persistence moment is finite. 
Taking $Z = \LC \xi_1,...,\xi_N \RC$ with $\xi_i = (b_i,d_i,k_i)$, we have: 
\begin{equation} \label{eqn_persis_control_1}
\int_{\W} \max(d_i -b_i) \delta Z \leq \int_{\W} \LN Z \RN f(Z) \delta Z
\end{equation}
since $\max(d_i - b_i) \leq \max\LP\LN (b_i,d_i) \RN\RP \leq \LN Z \RN$.
Consider a compact set $K \subset \W$ which contains a neighborhood of the origin. 
Given assumptions $(A2)^*$ and $(A3)^*$, Eq. \eqref{eqn_persis_control_1} is bounded by the following finite expression. 
\begin{equation} \label{eqn_persis_control_2}
\int_{\W} \LN Z \RN f(Z) \delta Z \leq \int_{K} C_2 \LN Z \RN \, \delta Z + \sum_{N=1}^M \int_{h_N^{-1}(h_N(K)^c)} C_3 \LN Z \RN^{-2N-1} d\xi_1...d\xi_N.
\end{equation}

Lastly, we take advantage of the Minkowski inequality, which holds trivially for set integration since it is a linear combination of Lebesgue integrals. 
Indeed, the MAD centered at $\mathscr{D}_0$ is bounded as follows.
\begin{equation} \label{eqn_minkowski}
\int_{\W} W_\infty(\mathscr{D}_0,Z) f(Z) \delta Z \leq \int_{\W} W_\infty(\mathscr{D}_0,\emptyset) f(Z) \delta Z + \int_{\W} W_\infty(\emptyset,Z) f(Z) \delta Z
\end{equation}
where $\emptyset$ represents the null persistence diagram and the distance to the null persistence diagram is precisely half the maximal persistence. 
Since $f$ integrates to 1, the first integral simplifies to the finite distance $W_\infty(\mathscr{D}_0,\emptyset)$, while the second integral is finite according to Eq. \eqref{eqn_persis_control_2}.

\subsection{Proof of Theorem \ref{thm_moment}}
The MAD of $f$ with origin $\mathscr{D}_0$ is finite by Lemma \ref{lemma_bottleneck_control}.
To show convergence of the estimate, we begin by adding and subtracting the integral of the sample estimator for the MAD. 
Then, we split the sum into $n+1$ terms via the triangle inequality to obtain
\begin{equation} \label{eqn_moment_split}
\begin{split}
\hspace*{10mm} &\hspace*{-10mm} \Ln \int_{\W} W_\infty(\mathscr{D}_0,Z) f(Z) \delta Z - \int_{\W} W_\infty(\mathscr{D}_0,Z) \hat f(Z) \delta Z \Rn \\
&\leq \Ln \int_{\W} W_\infty(\mathscr{D}_0,Z) f(Z) \delta Z - \frac{1}{n} \sum_{i=1}^n \int_{\W} W_\infty(\mathscr{D}_0,\mathscr{D}_i) K_\sigma(Z,\mathscr{D}_i) \delta Z \Rn \\
&+ \frac{1}{n} \sum_{i=1}^n \Ln \int_{\W} W_\infty(\mathscr{D}_0,Z) K_\sigma(Z,\mathscr{D}_i) \delta Z - \int_{\W} W_\infty(\mathscr{D}_0,\mathscr{D}_i) K_\sigma(Z,\mathscr{D}_i) \delta Z \Rn.
\end{split}
\end{equation}

The term of the upper bound in Eq. \eqref{eqn_moment_split} trivially simplfies to obtain the sample estimator for the MAD:
\begin{equation} \label{eqn_moment_1}
\begin{split}
\hspace{47mm} & \hspace*{-47mm} \Ln \int_{\W} W_\infty(\mathscr{D}_0,Z) f(Z) \delta Z - \sum_{i=1}^n \frac{1}{n} \int_{\W} W_\infty(\mathscr{D}_0,\mathscr{D}_i) K_\sigma(Z,\mathscr{D}_i) \delta Z \Rn \\
&= \Ln \int_{\W} W_\infty(\mathscr{D}_0,Z)f(Z) \delta Z - \frac{1}{n} \sum_{i=1}^n W_\infty(\mathscr{D}_0,\mathscr{D}_i) \Rn.
\end{split}
\end{equation}
The MAD sample estimator converges since the MAD is finite, and thus this term vanishes as $n \goto \infty$.
The remaining term of the upper bound in Eq. \eqref{eqn_moment_split} is further bounded via the reverse triangle inequality; specifically,
\begin{equation} \label{eqn_moment_2}
\begin{split}
\hspace{70mm} & \hspace*{-70mm} \sum_{i=1}^n \frac{1}{n} \Ln \int_{\W} W_\infty(\mathscr{D}_0,Z) K_\sigma(Z,\mathscr{D}_i) \delta Z - \int_{\W} W_\infty(\mathscr{D}_0,\mathscr{D}_i) K_\sigma(Z,\mathscr{D}_i) \delta Z \Rn \\
&\leq \sum_{i=1}^n \frac{1}{n} \Ln \int_{\W} W_\infty(\mathscr{D}_i,Z) K_\sigma(Z,\mathscr{D}_i) \delta Z \Rn.
\end{split}
\end{equation}

Toward bounding Eq. \eqref{eqn_moment_2}, choose a threshold parameter $a = O(\sigma^\beta)$ for some $\beta \in (0,1)$, so that $a \goto 0$ but $a/\sigma \goto \infty$ in the sample size (and bandwidth) limit. 
Next, take $A_i = \LC Z \subset W : W_\infty(Z,\mathscr{D}_i) \leq a \RC$ and split the integral between $A_i$ and its complement as
\begin{align*}
\int_{\W} W_\infty(\mathscr{D}_i,Z) K_\sigma(Z,\mathscr{D}_i) \delta Z = \int_{A_i} W_\infty(\mathscr{D}_i,Z) K_\sigma(Z,\mathscr{D}_i) \delta Z + \int_{A_i^c} W_\infty(\mathscr{D}_i,Z) K_\sigma(Z,\mathscr{D}_i) \delta Z.
\end{align*}
The integral over $A_i$ is trivially bounded by $a$.
Integration over the complementary events is controlled via layered integration along with Prop. \ref{prop_large_dev}.
For $a/\sigma > 1$, which occurs when $n$ is large enough, we obtain
\begin{equation} \label{eqn_layer_cake}
\begin{split}
\int_{A_i^c} W_\infty(\mathscr{D}_i,Z) K_\sigma(Z,\mathscr{D}_i) \delta Z &= a\, \P^i \LB W_\infty(\mathscr{D}_i,Z) > a\RB + \int_a^\infty \P^i \LB W_\infty(\mathscr{D}_i,Z) > b \RB db \\
& \leq a \LP \P[\Ln z \Rn > a/\sigma]^M \RP + \int_a^\infty \LP \P[\Ln z \Rn > b/\sigma]^M \RP db,
\end{split}
\end{equation}
where $z = (x,y)$ is distributed as a pair of independent standard normals. 
We chose $a/\sigma = O(\sigma^{\beta-1}) \goto \infty$ and so $\P(\Ln z \Rn < a/\sigma) \goto 0$ exponentially fast and the last term vanishes quickly as $\sigma \goto 0$.

Indeed, let $g(Z) = W_\infty(\mathscr{D}_i, Z)$, then by the fundamental theorem of calculus and Fubini's theorem:
\begin{align*}
	\int_{A_i^c} g(Z) K_\sigma(Z,\mathscr{D}_i) \delta Z
	&= \int_{\LC Z : g(Z) > a \RC} \LP \int_0^{g(Z)} db \RP K_\sigma(Z,\mathscr{D}_i)\delta Z \\
	&= \int_0^\infty \int_{\LC Z:g(Z) > a \textrm{ and } g(Z) > b \RC} K_\sigma(Z, \mathscr{D}_i) \delta Z db \\
	&= \int_0^\infty \P^f[g(Z) > \max \LC a, b \RC ] db \\
	&= a \P^f[g(Z) > a] + \int_a^\infty \P^f[g(Z) > b ] db.
\end{align*}
Applying Prop. \ref{prop_large_dev} changes the probabilities on $g(Z)$ to normal tail probabilities.
Thus, both bounding terms in Eq. \eqref{eqn_moment_split} converge to zero and thus the kernel estimate converges to the true mean absolute deviation. 

\section{Extra Examples}
	Here we present two more examples of constructing a kernel density estimator (KDE) according to the kernel given in Eq. \eqref{eqn_construction}. 
	In these examples, we view slices of the KDE at various sample sizes and bandwidths. 
	In the first example, the underlying dataset consists of points sampled from a circle with relatively large noise, in contrast to Ex. \ref{ex2} in Subsection \ref{subsect:Examples}.
	This example demonstrates how, despite the symmetry of the unit circle and Gaussian noise of the underlying data, the resulting persistence diagram KDE and eventually its limiting behavior lacks Gaussian structure.
	In the second example, the underlying dataset consists of points sampled from a pinched circle.
	The underlying dataset has only one loop, but the persistence diagrams typically have a feature of long persistence and another feature of moderate persistence.
	Both features are captured by the KDE, and are clearly separable into distinct features despite their adjacency.
	To keep the presentation relatively simple to interpret, the same slices will be presented for each KDE (see Rmk. \ref{rmk_slices}).
	This allows one to track the convergence of the KDE as the sample size of persistence diagrams, $n$, increases and the bandwidth, $\sigma$, decreases.
	
	\begin{example} \label{S1} \rm
		Consider random underlying datasets each consisting of 25 points sampled uniformly from the unit circle, which are then perturbed by Gaussian noise with variance $(1/6)^2 I_2$, and their associated $\cech$ persistence diagrams for degree of homology $k=1$.
		An example dataset and its associated $\cech$ persistence diagram for $k=1$ are shown in Fig. \ref{fig_PD_UD_circ}.
		
		\begin{figure} 
			\begin{multicols}{2}
				\begin{center}
					\includegraphics[scale=0.5]{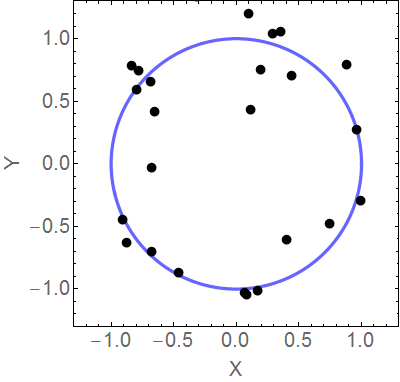} \\ (a) \\
					\includegraphics[scale=0.5]{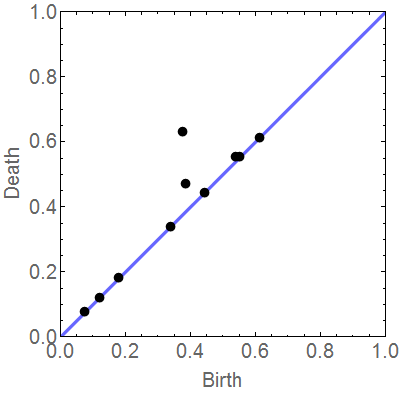} \\ (b)
				\end{center}
			\end{multicols}
			\caption{(a) An example of the underlying datasets generated for Ex. \ref{S1}.
					Each dataset consists of 25 points sampled uniformly on the unit circle which are then perturbed by i.i.d. Gaussian noise with variance $(1/6)^2 I_2$.
					(b) The persistence diagram associated to the $\cech$ filtration of the dataset}
			\label{fig_PD_UD_circ}
		\end{figure}
		
		Since the underlying datasets are sampled from a perfect circle perturbed by large noise, one expects the associated 1-homology to have a single persistent feature with several smaller features caused by noise. 
		We consider several KDEs as we simultaneously increase the number of persistence diagrams and narrow the bandwidth.
		The bandwidth was chosen to vary according to Silverman's rule of thumb (Silverman, 1986).
		Since the KDEs are defined on $\bigcup_N W^N$ for several input cardinalities $N$, we present $\hat f_{n, \sigma}(Z)$ in multiple slices by fixing a cardinality and then fixing all but one input feature, as explained in Rmk. \ref{rmk_slices}.
		For example, $g(\xi) = \hat f_{n, \sigma}(\xi, \xi_2',...,\xi_N')$ for fixed $\xi_j'$ is a function on $W$ and represents a slice of the local KDE on $W^N$.
		This progression of KDEs can be seen in Fig. \ref{fig_KDE_circ}, wherein the same slices are viewed for each choice of $n$ and $\sigma$.
		Modes of each slice are used as fixed features in the slices of higher cardinality inputs;
		consequently, the presented slices capture portions of the KDE with high probability density.
		
		\begin{figure}
			\begin{multicols}{3}
				\raisebox{27mm}{\large(1)}\includegraphics[scale = 0.35]{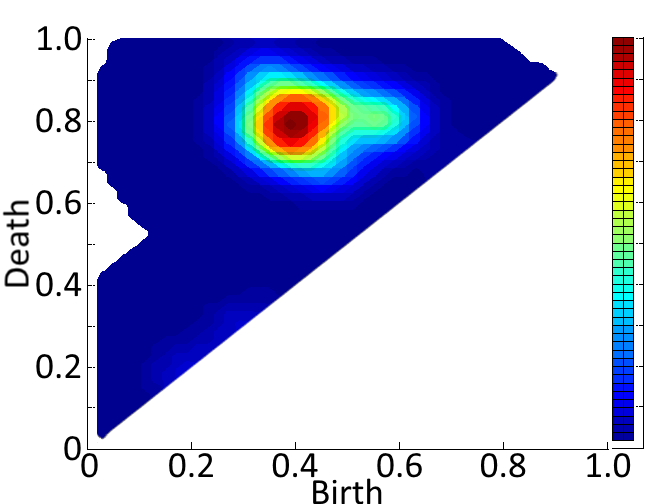}    \\
				\raisebox{27mm}{\large(2)}\includegraphics[scale = 0.35]{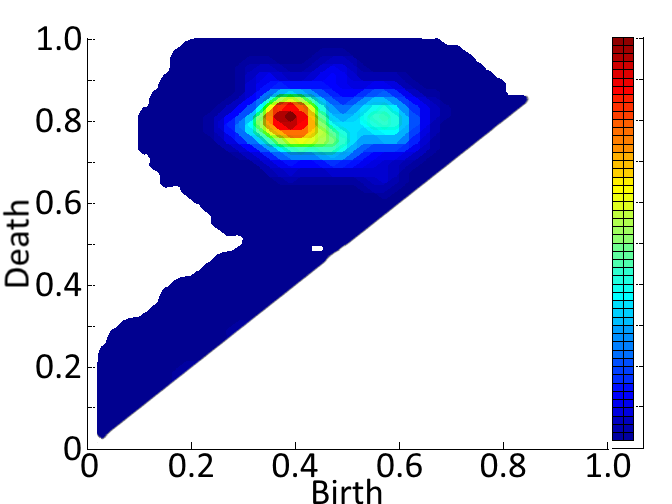}   \\
				\raisebox{27mm}{\large(3)}\includegraphics[scale = 0.35]{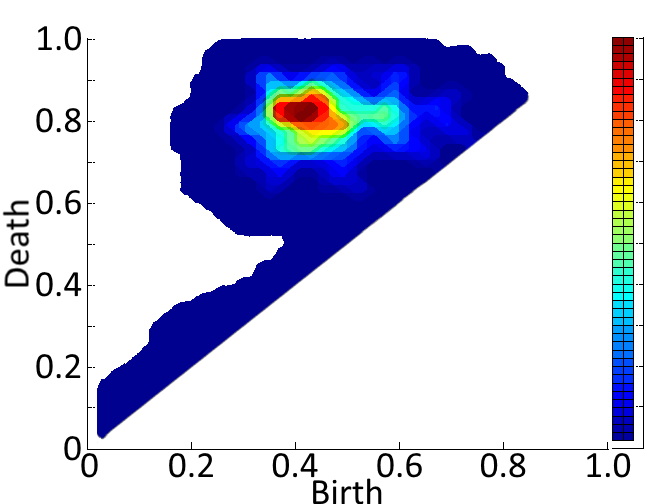}   \\
				\includegraphics[scale = 0.35]{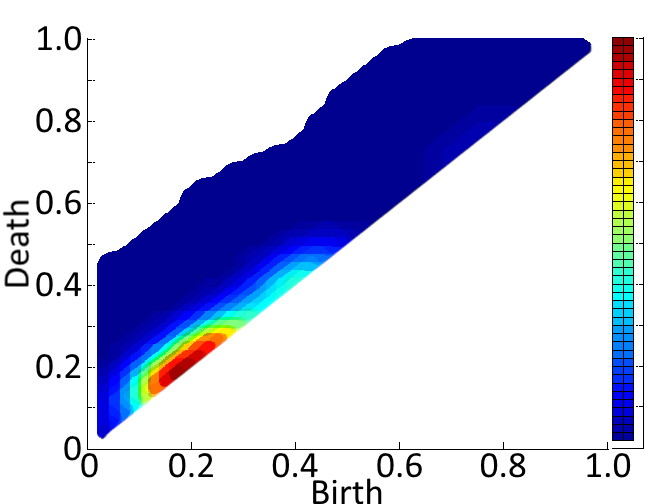}  \\
				\includegraphics[scale = 0.35]{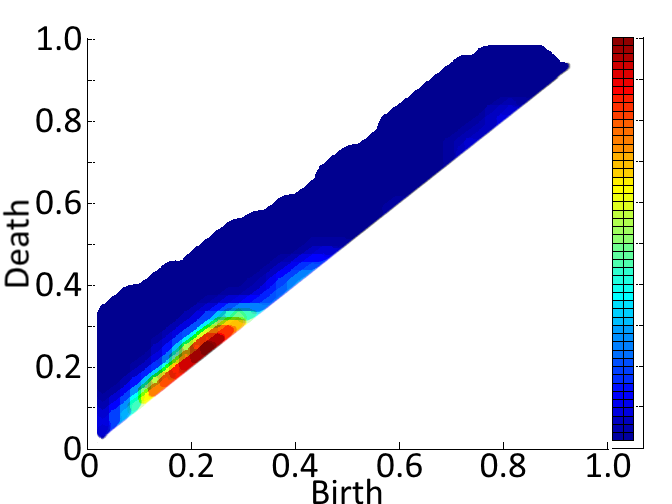} \\
				\includegraphics[scale = 0.35]{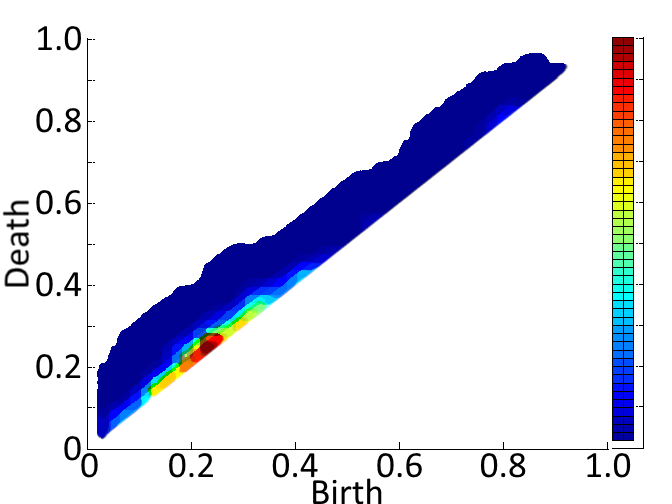} \\
				\includegraphics[scale = 0.35]{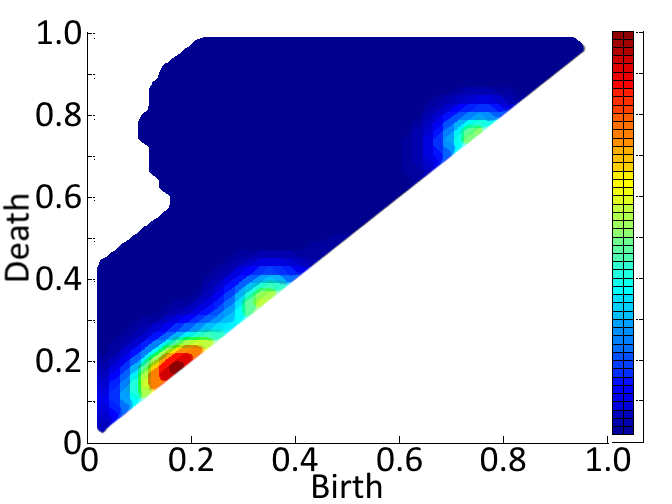}   \\
				\includegraphics[scale = 0.35]{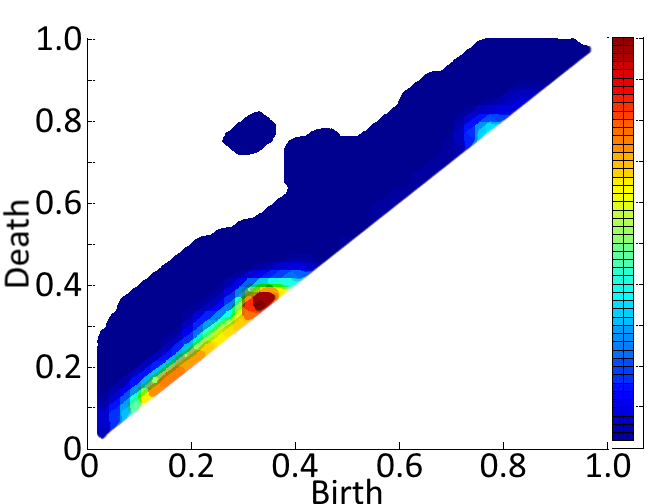} \\
				\includegraphics[scale = 0.35]{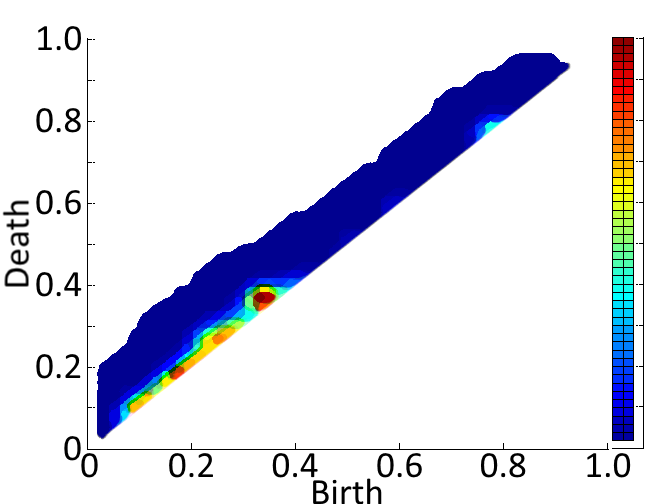}
			\end{multicols}
			\caption{Key slices of persistence diagram KDEs for Ex. \ref{S1}.
					Each column is a particular slice, while each row is a particular KDE: (1) $n = 20$ and $\sigma = 0.05$, (2) $n = 100$ and $\sigma = 0.03$, and (3) $n = 300$ and $\sigma = 0.02$.
					The first column are the local KDEs $\hat f_{n,\sigma}((b,d))$ evaluated at a diagram with only one feature.
					The second column are the local KDEs $\hat f_{n,\sigma}((b,d), (0.4,0.8))$ evaluated at a diagram with two features but one feature fixed.
					The third column are the local KDEs $\hat f_{n,\sigma}((b,d), (0.56,0.8))$ evaluated at a diagram with two features but with a different feature fixed.
					Overall, this figure demonstrates convergence of the KDE as the number of persistence diagrams increases and the bandwidth decreases.
					Indeed, the two modes on the left already stabilize after $n=300$, and the spread is no longer determined by the kernel bandwidth.} \label{fig_KDE_circ}
		\end{figure}
		
		Fig. \ref{fig_KDE_circ} demonstrates slower convergence of the KDEs than in Ex. \ref{ex2}, which is expected due to larger noise.
		Though the tail behavior of the KDEs remains Gaussian in nature, the limiting density is not Gaussian.
		In fact, the KDEs $\hat f(n, \sigma)$ are neither symmetric nor unimodal, even for a single input.
		Much like the kernel densities themselves, each KDE separates into upper and lower densities on $W$; however, the lower density varies depending on which upper mode is fixed in $\hat f(\xi, \xi_j')$.
		
		While the underlying dataspace is the unit circle in both Ex. \ref{ex2} and Ex. \ref{S1}, the precise presentation of the underlying data effects the pdf of the associated random persistence diagram.
		Precisely, two primary parameters for the underlying dataset are involved: (i) the scale of Gaussian noise and (ii) the sample size of the underlying dataset.	
		The persistence diagram (for degree of homology $k=1$) associated with the `'true' unit circle is not random and has a single feature at $(b,d) = (0,1)$.
		The random, discrete nature of these examples creates persistence diagrams which deviate from this `truth.'
		
		As described for Ex. \ref{ex2}, with very little noise all the sample points lie close to the unit circle, and so the $\cech$ complex becomes contractible at a radius $r \approx 1$.
		Consequently, the death value of the main topological feature is near the `true' value (e.g., the mode in Fig. \ref{fig_KDE_simp} is $d = 0.98 \approx 1$).
		However, since we are working with discrete points, this feature does not appear immediately: the gaps in the circle need to be filled in (this is even true without noise).
		In Ex. \ref{ex2}, the sample size is only 10, so the birth value is typically much larger than the `true' value (e.g., the mode in Fig. \ref{fig_KDE_simp} is $b = 0.77 >> 0$).
		
		In comparison to Ex. \ref{ex2}, Ex. \ref{S1} has relatively more noise; this results in a random persistence diagram with smaller death values for the main feature (e.g., the mode in Fig. \ref{fig_KDE_circ} is $d = 0.8 < 0.98$).
		It is evident from Fig. \ref{fig_KDE_circ} that while the noise is additive on the underlying data, its precise effect on the random persistence diagram is nonlinear.
		Moreover, Ex. \ref{S1} has a larger sample size (25 as opposed to 10), resulting in more consistent and smaller birth times for the main feature (e.g., the mode in Fig. \ref{fig_KDE_circ} is $b = 0.4 < 0.77$). 
		In addition, larger noise and sample size both result in more features near the diagonal in Ex. \ref{S1} as compared to Ex. \ref{ex2}.
		
	\end{example}

	\begin{example} \label{S2} \rm
		While Ex. \ref{S1} demonstrates the effect of noise on a persistence diagram pdf, this example will look into the effect of geometry.
		Consider random underlying datasets each consisting of 100 points sampled from a two-lobed polar curve, which are then perturbed by Gaussian noise with variance $(1/30)^2 I_2$, and their associated $\cech$ persistence diagrams for degree of homology $k=1$.
		An example dataset and its associated persistence diagram for $k=1$ are shown in Fig. \ref{fig_PD_UD_wob}.
		
		\begin{figure} 
			\begin{multicols}{2}
				\begin{center}
					\includegraphics[scale=0.5]{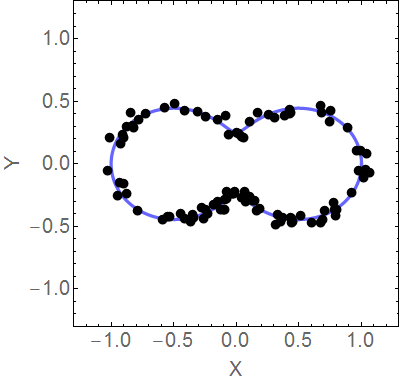} \\ (a) \\
					\includegraphics[scale=0.5]{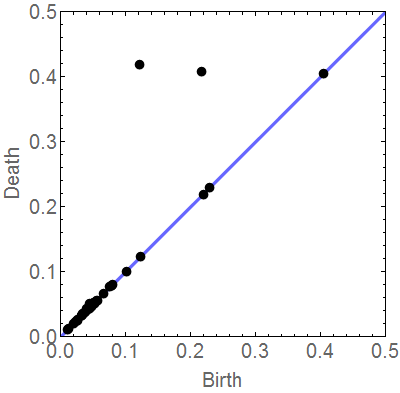} \\ (b)
				\end{center}
			\end{multicols}
			\caption{(a) An example of the underlying datasets generated for Ex. \ref{S1}.
					Each dataset consists of 100 points sampled uniformly (according to angle) on the two-lobed polar curve which are then perturbed by i.i.d. Gaussian noise with variance $(1/30)^2 I_2$.
					(b) The persistence diagram associated to the $\cech$ filtration of the underlying dataset.}
			\label{fig_PD_UD_wob}
		\end{figure}

		We consider several KDEs as we simultaneously increase the number of persistence diagrams and narrow the bandwidth.
		The bandwidth was chosen to vary according to Silverman's rule of thumb (Silverman, 1986).
		Since the KDEs are defined on $\bigcup_N W^N$ for several input cardinalities $N$, we present $\hat f_{n, \sigma}(Z)$ in multiple slices by fixing a cardinality and then fixing all but one input feature, as explained in Rmk. \ref{rmk_slices}.
		For example, $g(\xi) = \hat f_{n, \sigma}(\xi, \xi_2',...,\xi_N')$ for fixed $\xi_j'$ is a function on $W$ and represents a slice of the local KDE on $W^N$. 
		This progression of KDEs can be seen in Fig. \ref{fig_KDE_wob}, wherein the same slices are viewed for each choice of $n$ and $\sigma$.
		Modes of each slice are used as fixed features in the slices of higher cardinality inputs;
		consequently, the presented slices capture portions of the KDE with high probability density.
		Moreover, Fig. \ref{fig_KDE_wob} demonstrates that these slices tend to capture specific topological or geometric features of the underlying dataspace.
		
		\begin{figure}
			\begin{multicols}{4}
				\raisebox{19mm}{(1)}\includegraphics[scale = 0.25]{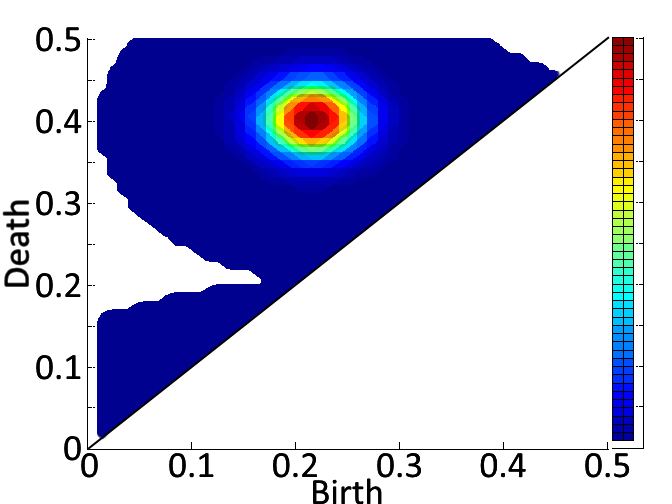}   \\
				\raisebox{19mm}{(2)}\includegraphics[scale = 0.25]{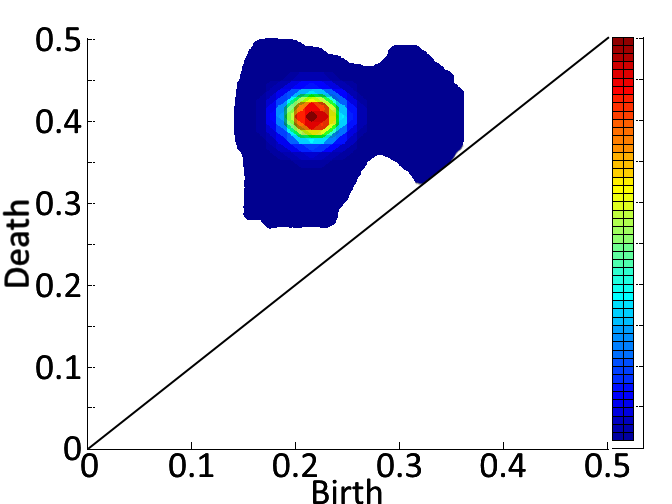}  \\
				\raisebox{19mm}{(3)}\includegraphics[scale = 0.20]{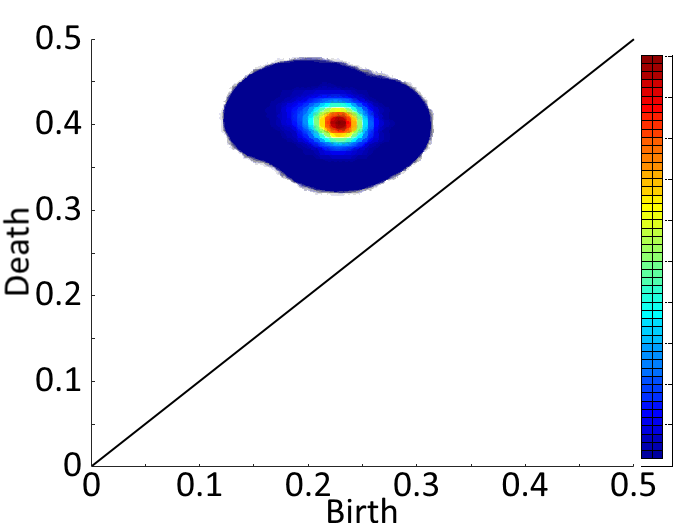}  \\
				\raisebox{19mm}{(4)}\includegraphics[scale = 0.20]{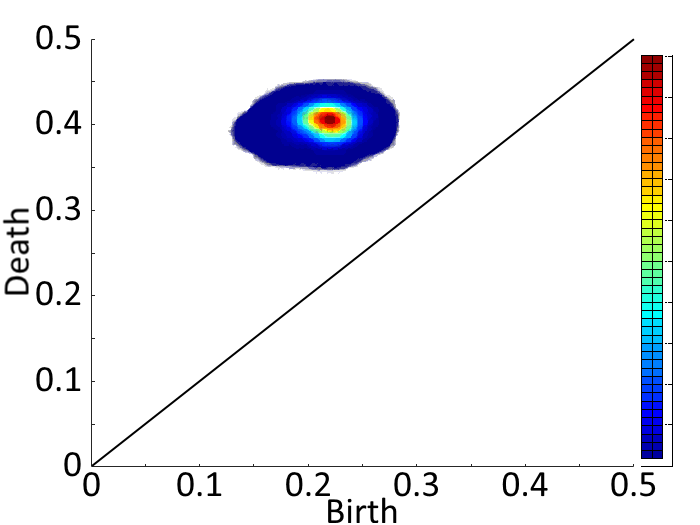}  \\
				\includegraphics[scale = 0.25]{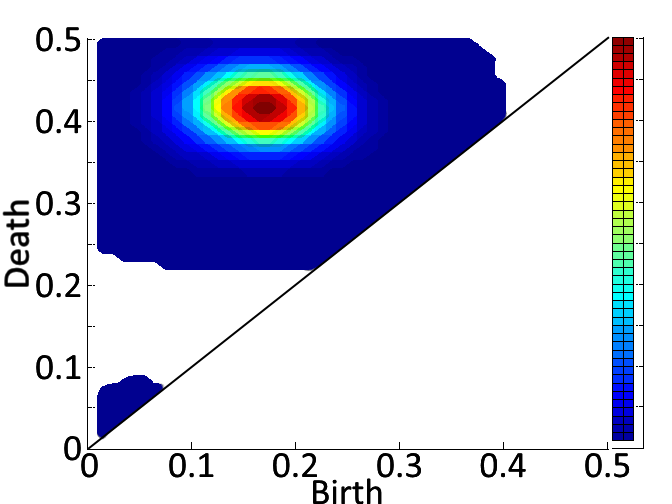}   \\
				\includegraphics[scale = 0.25]{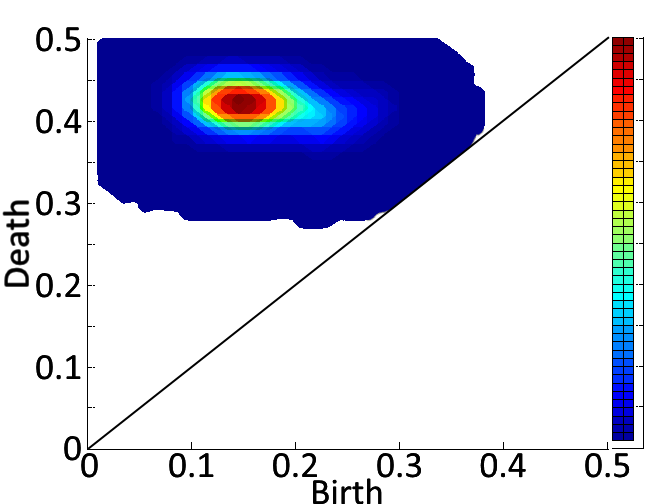}  \\
				\includegraphics[scale = 0.20]{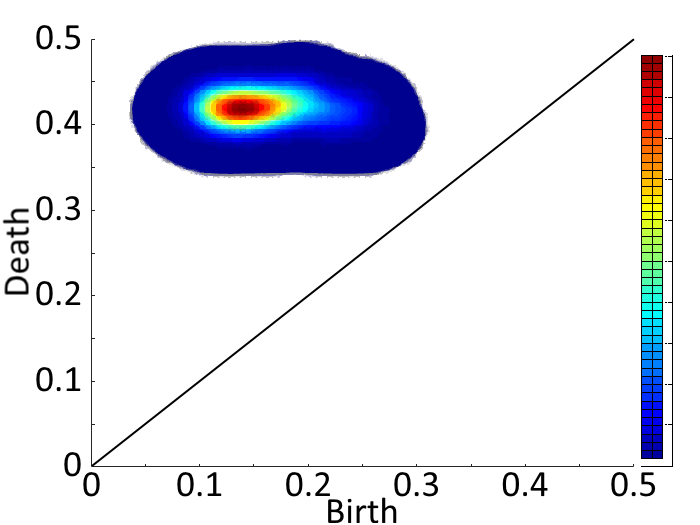}  \\
				\includegraphics[scale = 0.20]{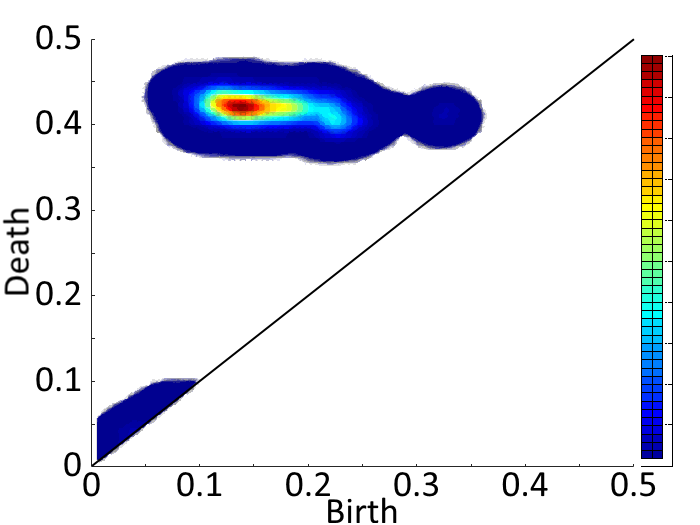}  \\
				\includegraphics[scale = 0.25]{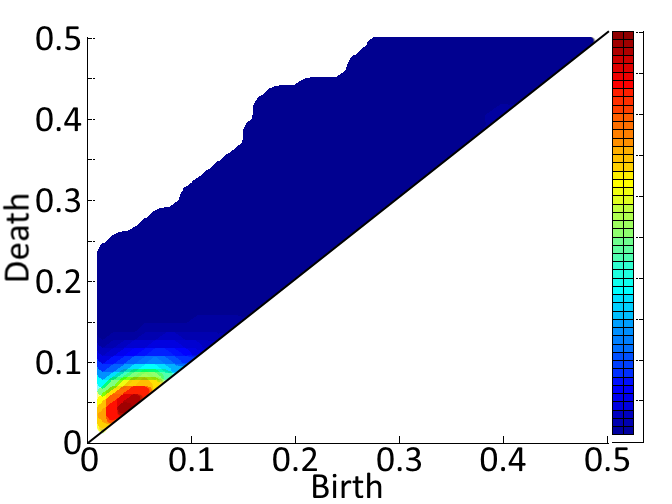}   \\
				\includegraphics[scale = 0.25]{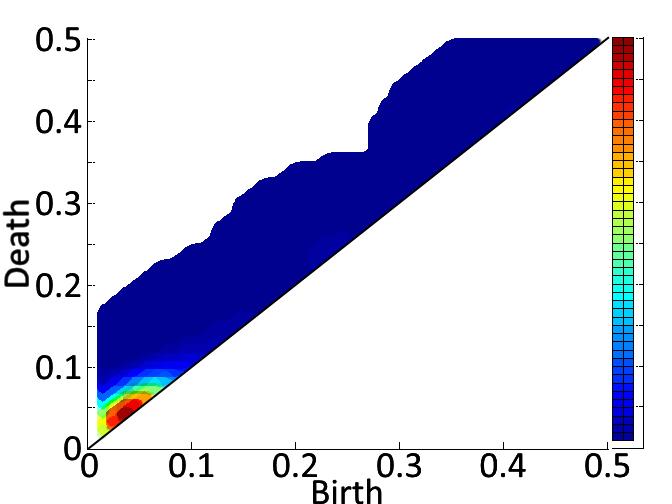}  \\
				\includegraphics[scale = 0.20]{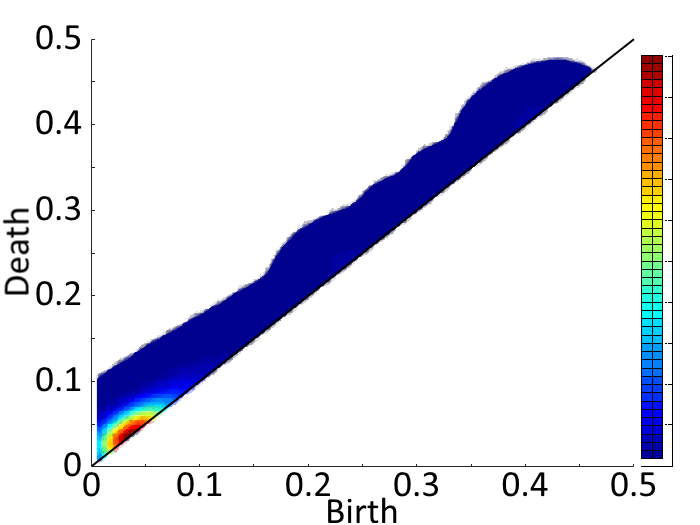}  \\
				\includegraphics[scale = 0.20]{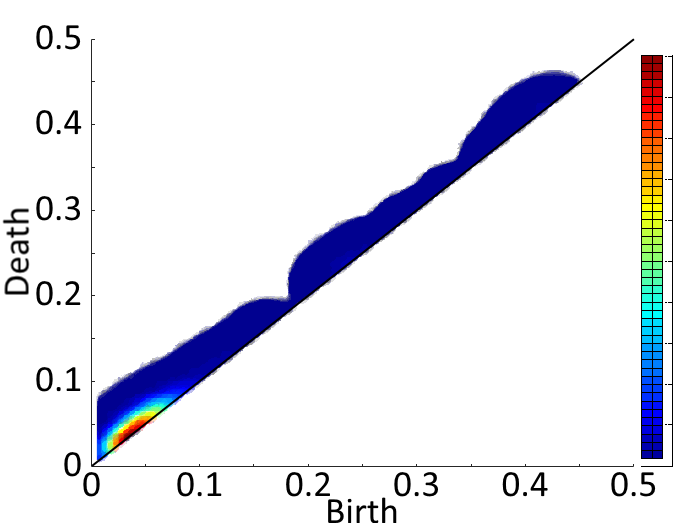}  \\
				\includegraphics[scale = 0.25]{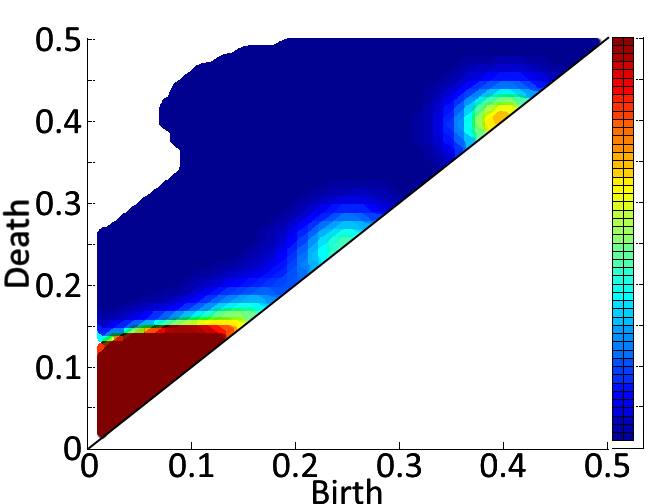}  \\
				\includegraphics[scale = 0.25]{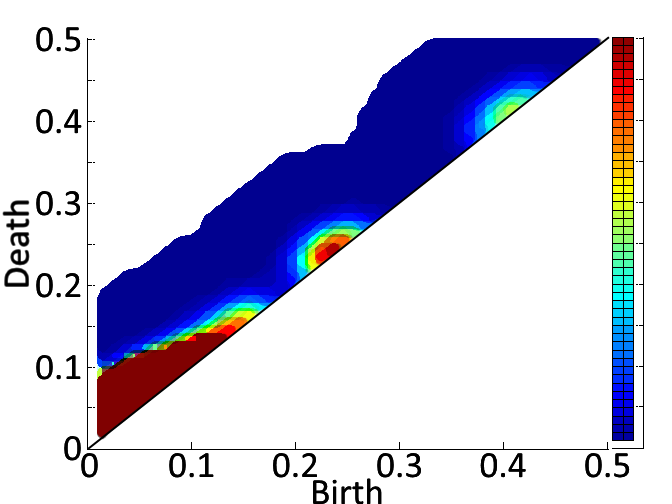} \\
				\includegraphics[scale = 0.20]{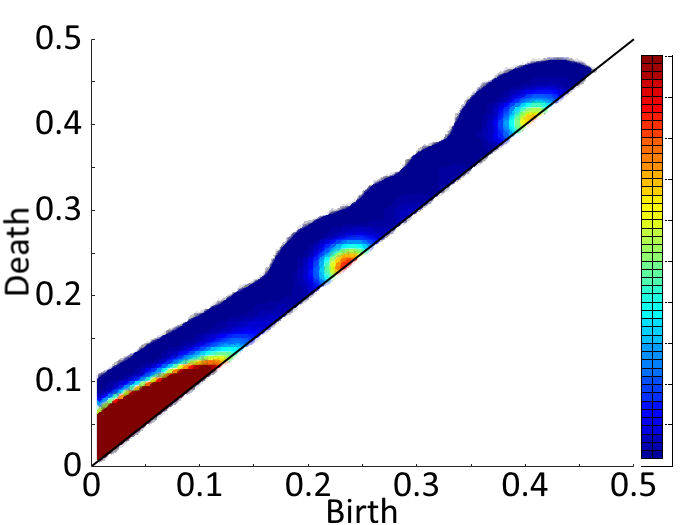} \\
				\includegraphics[scale = 0.20]{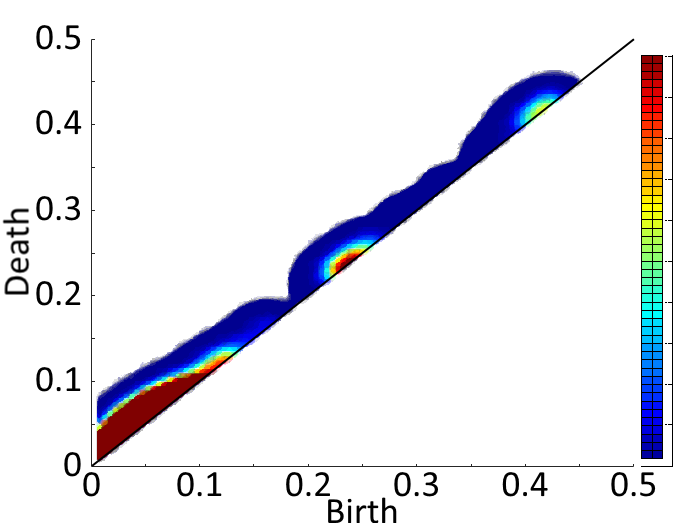}
			\end{multicols}
			\caption{Key slices of persistence diagram KDEs for Ex. \ref{S2}.
					Each column is a particular slice, while each row is a particular KDE (1) $n = 20$ and $\sigma = 0.03$, (2) $n = 100$ and $\sigma = 0.02$, (3) $n = 300$ and $\sigma = 0.015$, and (4) $n = 1000$ and $\sigma = 0.01$.
					The first column are the local KDEs $\hat f_{n,\sigma}((b,d))$ evaluated at a diagram with only one feature; the mode is $\xi_1' = (0.2,0.4)$.
					The second column are the local KDEs $\hat f_{n,\sigma}((b,d), (0.2,0.4))$ evaluated at a diagram with two features, but with one feature fixed; the mode is $\xi_2' = (0.14,0.42)$
					The third column are the local KDEs $\hat f_{n,\sigma}((b,d), (0.2,0.4), (0.14,0.42))$ evaluated at a diagram with three features, but with two features fixed.
					The fourth column shows the same slices as the third, but with the colormap shifted down to show the smaller modes.
					The variance of certain features effects the rate of convergence nearby, similar to Gaussian KDE in Euclidean space for a distribution with modes of different variance.}
			\label{fig_KDE_wob}
		\end{figure}
		
		The two-lobed curve in this example has a $\cech$ persistence diagram consisting of two features, a topological feature of very long persistence and a geometric feature of moderate persistence.
		The moderate persistence feature describes the pinching of the curve.
		These two features are captured as separate points by the KDEs, and are thus viewed in completely separate slices of the KDE.
		By observing the KDE in the last row of Fig. \ref{fig_KDE_wob}, the geometric feature with moderate persistence has considerably less variance.
		Indeed, while the birth time of the topological feature relies on bridging gaps around the entire shape, which can all vary, the larger birth time of the geometric feature has less variance since it relies solely only on the short circuit between the lobes. 
		As a result of this small variance, the geometric feature is emphasized for the local KDEs with a single input feature; also, the density takes longer to converge near this feature.
		
		The lower portion of the KDE shows three separate modes.
		Features which build the largest mode consists of small loops, caused by local noise and gaps along the curve.
		The two modes which appear at larger scale indicate short circuiting of the pinch (smaller) or one of the lobes (larger, like the second mode in the circle example);
		These two lower modes are separate from noise-based features and are indicative of geometry in the underlying data.
		
\end{example}

\end{document}